\documentclass{amsproc}

\usepackage{amsmath,amssymb}
\newtheorem{theo}{Theorem}[section]
\newtheorem{lemm}[theo]{Lemma}
\newtheorem{prop}[theo]{Proposition}

\theoremstyle{definition}
\newtheorem{defi}[theo]{Definition}

\newtheorem{coro}[theo]{Corollary}
\theoremstyle{remark}
\newtheorem{remark}[theo]{Remark}

\numberwithin{equation}{section}

\def\lg{\langle}
\def\rg{\rangle}

\def\lra{\longrightarrow}
\def\al{\alpha}
\def\be{\beta}
\def\ga{\gamma}

\def\vn{\varepsilon}

\def\ep{\epsilon}
\def\ot{\otimes}

\def\om{\omega}
\def\lra{\longrightarrow}
\def\ra{\rangle}
\def\la{\langle}

\begin{document}
\title[$\tau$-Invariant Generating Functions, Vertex Representations of $U_{r,s}(\widehat{\mathfrak{g}})$]
{Generating Functions with $\tau$-Invariance and Vertex Representations of
Quantum Affine Algebras $U_{r,s}(\widehat{\mathfrak{g}})$\\ (I): Simply-laced Cases}
\author[Hu]{Naihong Hu}
\address{Department of Mathematics, Shanghai Key Laboratory of Pure Mathematics and Mathematical Practice, East China Normal University,
Shanghai 200241, PR China} \email{nhhu@math.ecnu.edu.cn}
\author[Zhang]{Honglian Zhang$^\star$}
\address{Department of Mathematics, Shanghai University,
Shanghai 200444, PR China} \email{hlzhangmath@shu.edu.cn}
\thanks{$^\star$H.Z., Corresponding Author}


\date{Sept. 1, 2004 and, in revised form, May 18, 2005.}


\begin{abstract}
We put forward the exact version of two-parameter generating functions with $\tau$-invariance, which allows us to give a unified and inherent definition for the Drinfeld realization of two-parameter quantum affine algebras for all the untwisted types.
As verification, we first construct their level-one vertex representations of $U_{r,s}(\widehat{\mathfrak{g}})$ for simply-laced types, which in turn well-detect the effectiveness of our definitions both for $(r,s)$-generating functions and $(r,s)$-Drinfeld realization in the framework of establishing the two-parameter vertex representation theory.
\end{abstract}

\maketitle

\section{Introduction}
\medskip
In 2000, Benkart and Witherspoon  revitalized the research of
two-parameter quantum groups. They studied the structures of
two-parameter quantum groups $U_{r,s}(\frak g)$ for ${\frak
g}={\frak{gl_n}}$, or ${\frak{sl_n}}$ in \cite{BW1} previously
obtained by Takeuchi \cite{T}, and the finite-dimensional
representations and Schur-Weyl duality for type $A$ in \cite{BW2},
and obtained some new finite-dimensional pointed Hopf algebras
${\frak u}_{r,s}(\frak{sl}_n)$ in \cite{BW3}, which possess new
ribbon elements under some conditions (may yield new invariants of
$3$-manifolds). Since 2004, Bergeron, Gao and Hu \cite{BGH1} further
presented the structures of two-parameter quantum groups
$U_{r,s}({\frak{g}})$ for ${\frak g}=\frak {so}_{2n+1},\,
\frak{sp}_{2n},\, \frak{so}_{2n}$, and investigated the environment
condition on the Lusztig symmetries' existence for types $A$, $B$,
$C$, $D$ in \cite{BGH1} and type $G_2$ in \cite{HS}. A
generalization of this fact to the multi-parameter cases arising
from Drinfeld doubles of bosonizations of Nichols algebras of
diagonal type has been obtained by Heckenberger \cite{H}, which provides an explicit realization model for the abstract
concept ``Weyl groupoid" playing a key role in the classification of both Nichols algebras of diagonal type (cf. \cite{H1})
and finite-dimensional pointed Hopf algebras with abelian group algebras as the coradicals (\cite{AS}). The
finite-dimensional weight representation theory for type $A$
(\cite{BW2}), types $B$, $C$, $D$ (\cite{BGH2}), $E$ (\cite{BH}) was
done.

A unified definition for any types and the explicit formulae of two-parameter quantum groups as two kind of $2$-cocycle
deformations of one-parameter quantum groups with double group-like elements in generic case, and a general treatment for
the deformed finite-dimensional representation category have been
intrinsically described by Hu-Pei \cite{HP1, HP2} (Furthermore, for a multiparameter version, see Pei-Hu-Rosso \cite{PHR}).
Recently, this important observation on the explicit formula for $2$-cocycle deformation serves as a crucial point for categorifying two-parameter
quantum groups in generic case. In roots of unity case, the small quantum groups
structure ${\frak u}_{r,s}(\frak{g})$, together with the convex
PBW-type Lyndon bases was studied explicitly for types $B$, $G_2$
(\cite{HW1, HW2}), $C$ (\cite{CH}), $F_4$ (\cite{CxH}). Especially,
Isomorphism Theorem for ${\frak u}_{r,s}(\frak{g})$'s depending
on parameter-pairs $(r,s)$ established by Hu-Wang \cite{HW1, HW2} was used
to distinguish the iso-classes of new finite-dimensional pointed
Hopf algebras in type $A$ for small order of roots of unity (see Benkart at al \cite{BPW}).
Surprisingly, the study of two-parameter small quantum groups brings us new examples of non-semisimple and non-pointed Hopf algebras with non-pointed duals
when Garc\'ia \cite{G} studied the quantum subgroups of two-parameter quantum general linear group $GL_{\al,\be}(n)$.

On the other hand, Hu-Rosso-Zhang first studied two-parameter
quantum affine algebras associated to affine Lie algebras of type $A_n^{(1)}$, and gave the descriptions of
the structure and Drinfeld realization of
$U_{r,s}(\widehat{\frak{sl}_n})$, as well as the quantum affine
Lyndon basis (see \cite{HRZ}). The discussions for the other affine
cases of untwisted types and the corresponding vertex operators
constructions for all untwisted types have been done in \cite{HZ1, HZ2, Z}.
Recently, using a combinatorial model of Young diagrams,
Jing-Zhang (\cite{JZ1}) gave a fermionic realization of the
two-parameter quantum affine algebra of type $A_n^{(1)}$; while
\cite{JZ2} provided a group-theoretic realization of two-parameter
quantum toroidal algebras using finite subgroups of
$SL_2(\mathbb{C})$ via McKay correspondence.

In the present paper, we will study
two-parameter quantum affine algebras for untwisted types from a uniform approach via working out the exact two-parameter version of generating functions with $\tau$-invariance, and construct the level-one vertex operator representations
for simply-laced types, for the other types will be discussed in forthcoming work.

The paper is organized as follows. We first give the definition of two-parameter quantum
affine algebras $U_{r,s}(\widehat{\mathfrak{g}})$ ($\mathfrak{g}$ is of $ADE$-type) in the sense of Hopf algebra in Section 2.
Two-parameter quantum affine algebras $U_{r,s}(\widehat{\mathfrak g})$
are characterized as Drinfeld doubles ${\mathcal D}({\widehat{\mathcal
B}}, {\widehat{\mathcal B'}})$ of Hopf subalgebras
${\widehat{\mathcal B}}$ and ${\widehat{\mathcal B'}}$ with respect
to a skew-dual pairing we give. In Section 3, we obtain the
two-parameter Drinfeld realization via the generating functions we define in the two-parameter setting.
It is worthy to mention that a similar $2$-cocycle deformation between two-parameter and one-parameter quantum affine algebras
doesn't work yet for Drinfeld generators, even though it indeed exists
when one works for Chevalley-Kac-Lusztig generators (see \cite{HP2, PHR}).
From this point of view, there's the rub that explicit formulae for defining the Drinfeld realization in the two-parameter setting are  nontrivial.
Note that by comparison with \cite{HRZ}, the definition has been slightly revised, where
the canonical central element $c$ of affine Lie algebras involved plays a well-connected role in the definition such as a result of the product for the doubled group-like elements $\gamma$ and $\gamma'$ keeps group-like
(see Definitions 2.2 (X1), \& 3.1 (3.2)), this also behaviors well when one considers
especially those vertex representations of higher levels. While in the case $r=q=s^{-1}$, this phenomenon degenerates and
is invisible in one-parameter case. In order to recover those inherent features in the two-parameter setting, to some extent, representation theory serves as a nice sample here to help us to achieve some further necessary insights into the algebra structure itself.
In Section 4, we prove the Drinfeld Isomorphism from the two-parameter quantum
affine algebra $U_{r,s}(\widehat{\mathfrak{g}})$ ($\frak{g}$ is of type
$D$ and $E$) of Drinfeld-Jimbo type towards the two-parameter
Drinfeld realization $\mathcal U_{r,s}(\widehat{\mathfrak g})$ using the quantum calculations for $(r,s)$-brackets as developed
in \cite{HRZ} for type $A^{(1)}_n$.
Here we afford an alternative proof of
the homomorphism to be injective. In
Section 5, we start from the two-parameter quantum Heisenberg algebra to obtain the
Fock space, and construct the level-one vertex representations of
two-parameter quantum affine algebras for simply-laced cases, which
are irreducible. Also, these constructions in turn well-detect the effectiveness of the $(r,s)$-Drinfeld realization we defined.
We include details of some proofs of a few Lemmas as an appendix, through which readers can see the quantum calculations for $(r,s)$-brackets how to work effectively in the two-parameter setting.

\section{Quantum affine
algebras $U_{r,s}(\widehat{\frak{g}})$ and Drinfeld double}
\subsection{Notations and preliminaries }

From now on, denote by $\mathfrak{g}$ the finite-dimensional simple Lie algebras of simply-laced types with rank $n$. Let ${\mathbb
K}\supset{\mathbb Q}(r,s)$ denote an algebraically closed field,
where the two-parameters $r,\,s$ are nonzero complex numbers
satisfying $r^2\ne s^2$.  Let $E$ be a Euclidean space ${\mathbb R}^n$
with an inner product $(\,,)$ and  an orthonormal basis  $\{ \ep_1,
\cdots, \ep_n \}$.

Let $(a_{ij})_{ij\in I} \, (I=\{1,\,2,\,\cdots,\,n\})$ be a Cartan
matrix of simple Lie algebra $\mathfrak{g}$ with Cartan subalgebra
$\frak{h}$. Let $\Phi$ be a root system of $\mathfrak {g}$ and
$\al_i \,(i \in I)$ be the simple roots. It is possible to regard
$\Phi$ as a subset of a Euclidean space E. Denote by $\theta$ the
highest root of $\Phi$.

Let $\hat{\mathfrak g}$ be an affine Lie algebra
associated to simple Lie algebra $\mathfrak{g}$ with Cartan matrix
$(a_{ij})_{ij\in I_0}$, where $I_0=\{0\}\cup I$. In the following, we
list the affine Dynkin diagrams of simply-laced types, the labels
on vertices fix an identification between $I_0$ and
$\{0,\,1,\,\cdots,\,n\}$ such that $I$ corresponds to
$\{1,\,2,\,\cdots,\,n\}$.

\vspace{8pt}

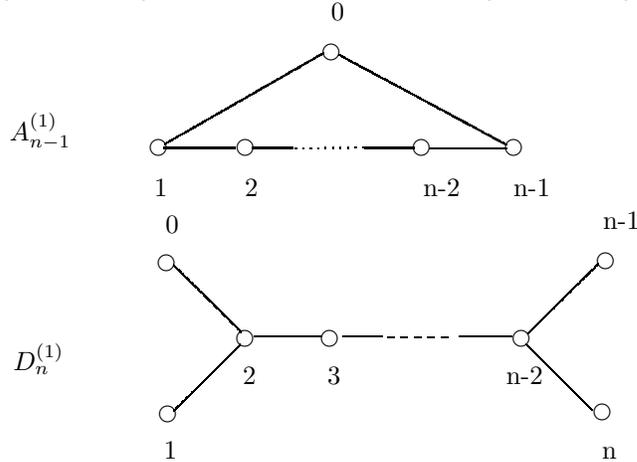
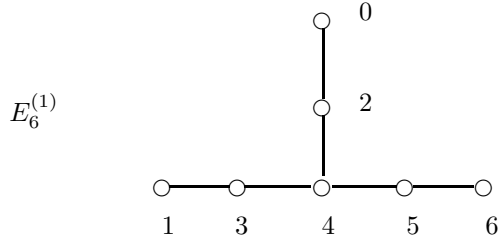
\unitlength 1mm 
\linethickness{0.4pt}
\ifx\plotpoint\undefined\newsavebox{\plotpoint}\fi 
\begin{picture}(74,27)(0,0)
\put(6.5,11.5){$A_{n-1}^{(1)}$} \put(26.25,10.5){\circle{2}}
\put(37.8,10.5){\circle{2}} \put(61.25,10.5){\circle{2}}
\put(73.5,10.5){\circle{2}} \put(49.25,23.25){\circle{2}}
\put(27,10.5){\line(1,0){9.5}} \put(38.75,10.5){\line(1,0){5.5}}
\put(53.75,10.5){\line(1,0){6.5}} \put(25.75,4){1} \put(37.75,4){2}
\put(61.5,4){n-2} \put(73.5,4){n-1} \put(49.25,27.5){0}
\put(62.31,10.43){\line(1,0){10.253}}
\multiput(27.05,10.96)(.058843496,.033730372){359}{\line(1,0){.058843496}}
\multiput(50.2,22.89)(.062639361,-.033671751){357}{\line(1,0){.062639361}}
\multiput(44.48,10.36)(.9919,.0098){10}{{\rule{.4pt}{.4pt}}}
\end{picture}

\unitlength 1mm 
\linethickness{0.4pt}
\ifx\plotpoint\undefined\newsavebox{\plotpoint}\fi 
\begin{picture}(86.81,35.53)(0,0)
\put(27.5,11){\circle{2}} \put(37.75,21){\circle{2}}
\put(49,21){\circle{2}} \put(39,21.25){\line(1,0){9}}
\put(50.75,21.25){\line(1,0){5.25}} \put(66.25,21.25){\line(1,0){7}}
\put(56.68,21.18){\line(1,0){.9}} \put(58.48,21.18){\line(1,0){.9}}
\put(60.28,21.18){\line(1,0){.9}} \put(62.08,21.18){\line(1,0){.9}}
\put(63.88,21.18){\line(1,0){.9}}
\put(27,5){1} \put(37.5,15){2} \put(48.75,15){3}
\put(7,16.75){$D_n^{(1)}$} \put(85.25,11.25){\circle{2}}
\put(74.5,21){\circle{2}} \put(85.25,5){n} \put(72.5,15){n-2}
\put(27.31,31.11){\circle{1.95}} \put(27.22,35){0}
\multiput(75.22,20.24)(.033655563,-.033655563){260}{\line(0,-1){.033655563}}
\multiput(75.4,21.66)(.033946249,.033626002){276}{\line(1,0){.033946249}}
\put(85.74,31.29){\circle{2.15}} \put(85.38,35.53){n-1}
\multiput(28.38,11.44)(.033894231,.033653846){260}{\line(1,0){.033894231}}
\multiput(28.31,30.75)(.033854167,-.033617424){264}{\line(1,0){.033854167}}
\end{picture}

\unitlength 1mm 
\linethickness{0.4pt}
\ifx\plotpoint\undefined\newsavebox{\plotpoint}\fi 
\begin{picture}(70.78, 31.5)(0, 0)
\put(6.75, 17.5){$E_6^{(1)}$} \put(27, 8.5){\circle{2}} \put(37,
8.5){\circle{2}} \put(48.25, 8.5){\circle{2}} \put(48.25,
19.25){\circle{2}} \put(48.25, 30.75){\circle{2}} \put(59.25,
8.5){\circle{2}} \put(69.75, 8.5){\circle{2}}
\put(28, 8.75){\line(1, 0){8}}
\put(38, 8.75){\line(1, 0){8.75}}
\put(49.75, 8.75){\line(1, 0){8.5}}
\put(60.5, 8.75){\line(1, 0){8}}
\put(48.5, 20.5){\line(0, 1){9.25}}
\put(48.5, 18.25){\line(0, -1){8}}
\put(27, 2.5){1} \put(36.75, 2.5){3} \put(48.25, 2.5){4} \put(59.5,
2.5){5} \put(70, 2.5){6} \put(53.25, 18.75){2} \put(53.25, 31){0}
\end{picture}

Let $\delta$ denote the primitive imaginary root of $\hat{\mathfrak
g}$, Take $\alpha_0=\delta-\theta$, then $\Pi'=\{\alpha_i\mid i\in
I_0\}$ is a base of simple roots of affine Lie algebra
$\hat{\mathfrak g}$. Denote by $c$ the canonical central element,
and $h$ the Coxeter number of affine Lie algebra $\hat{\mathfrak{g}}$.
We need the following data on (prime) root systems.

Type $A_{n-1}$:
\begin{equation*}
\begin{split}
&\Pi=\{\al_i=\ep_i-\ep_{i+1}\mid 1\le i\le n\},\\
&\Psi=\{\pm(\ep_i-\ep_j)\mid 1\le i<j\le n+1\}, \\
&\theta=\alpha_1+ \cdots + \alpha_n,\\
&\alpha_0=\delta-\theta=\delta-\ep_1+\ep_n ,\\
&\Pi'=\{\al_0, \al_1, \cdots, \al_n\}.
\end{split}
\end{equation*}

Type $D_{n}$:
\begin{equation*}
\begin{split}
&\Pi=\{\al_i=\ep_i-\ep_{i+1}\mid 1\le
i<n\}\cup\{\al_n=\ep_{n-1}+\ep_n\},\\
&\Psi=\{\pm\ep_i\pm\ep_j\mid 1\le i\ne j\le n\},\\
&\theta=\alpha_1+2\al_2+ \cdots +2\al_{n-2}+\al_{n-1}+\alpha_n,\\
&\alpha_0=\delta-\theta=\delta-\ep_1-\ep_2,\\
&\Pi'=\{\al_0, \al_1, \cdots, \al_n\}.
\end{split}
\end{equation*}

Type $E_6$:
\begin{equation*}
\begin{split}
&\Pi=\bigl\{\,\al_1=\frac{1}{2}(\ep_1+\ep_8-(\ep_2+\cdots+\ep_7)),\,
\al_2=\ep_1+\ep_2,\,\\
&\qquad\al_3=\ep_2-\ep_1,\,\al_4=\ep_3-\ep_2,\,\al_5=\ep_4-\ep_3,\,\al_6
=\ep_5-\ep_4\,\bigr\},\\
&\Psi=\bigl\{\,\pm\ep_i\pm\ep_j\mid 1\le i\ne j\le n=5\,\bigr\}\cup
\Bigl\{\frac{1}{2}\sum_{i=1}^{5}(-1)^{k(i)}\ep_i-\frac{1}2(\ep_6+\ep_7-\ep_8)\;\Big|\\
&\hskip2.8cm k(i)=0, 1, \;\hbox{add up to an odd integer}\;\Bigr\},\\
&\theta=\alpha_1+2\al_2+2\al_{3}+3\al_{4}+2\alpha_5+\al_6,\\
&\alpha_0=\delta-\theta=\delta-\frac{1}{2}(\ep_1+\ep_2+\ep_3+\ep_4+\ep_5-\ep_6-\ep_7+\ep_8),\\
&\Pi'=\{\al_0, \al_1, \cdots, \al_6\}.
\end{split}
\end{equation*}

\subsection{Two-parameter
quantum affine algebras $U_{r,s}(\widehat{\frak{g}})$}

In this paragraph, we give the definition of the two-parameter
quantum affine algebras $U_{r,s}(\hat{\mathfrak g})$ (see \cite{HRZ} for
${\hat{\mathfrak g}}=A^{(1)}_{n}$).

Assigned to $\Pi'$, there are two sets of mutually-commutative
symbols $W=\{\om_i^{\pm1}\mid 0\le i\le n\}$ and
$W'=\{{\om_i'}^{\pm1}\mid 0\le i\le n\}$. Define a pairing $\lg\, ,
\rg:\, W'\times W\lra {\mathbb K}$ as follows
\begin{equation*}
J=(\lg \om_i', \om_j\rg)=(\lg i,j\rg) \qquad \qquad\quad\text{\it
for}\quad A_{n-1}^{(1)}, \tag{$1_{\hat A_{n-1}}$}
\end{equation*}
where $J=\left(\begin{array}{cccccc}
rs^{-1}& r^{-1}& 1 & \cdots & 1 & s \\
s & rs^{-1} & r^{-1}  & \cdots & 1 & 1\\
\cdots &\cdots &\cdots & \cdots & \cdots & \cdots\\
1 & 1 & 1  & \cdots & rs^{-1} & r^{-1}\\
 r^{-1} & 1 & 1 & \cdots & s & rs^{-1}
\end{array}\right).$

\begin{equation*}
J=(\lg \om_i', \om_j\rg)=(\lg i,j\rg) \qquad \qquad\quad\text{\it
for}\quad D_{n}^{(1)}, \tag{$1_{\hat D_{n}}$}
\end{equation*}
where $J=\left(\begin{array}{cccccc}
rs^{-1}& (rs)^{-1}& r^{-1} & \cdots & 1 & (rs)^2 \\
rs & rs^{-1} & r^{-1}  & \cdots & 1 & 1\\
\cdots &\cdots &\cdots & \cdots & \cdots & \cdots\\
1 & 1 & 1  & \cdots & rs^{-1} & (rs)^{-1}\\
 (rs)^{-2} & 1 & 1 & \cdots & rs & rs^{-1}
\end{array}\right)$


\begin{equation*}
J=(\lg \om_i', \om_j\rg)=(\lg i,j\rg) \qquad \qquad\quad\text{\it
for}\quad E_6^{(1)}, \tag{$1_{\hat E_6}$}
\end{equation*}
where $J=
 \left(\begin{array}{ccccccc}
rs^{-1}& (rs)^{-1}& r^{-2}s^{-1} &(rs)^{-1} & rs & rs & rs\\
rs & rs^{-1} & 1 & r^{-1} & 1 & 1 & 1\\
rs^{2} &1 &rs^{-1} & 1 & r^{-1} &1 &1 \\
rs & s & 1 & rs^{-1} & r^{-1} & 1 & 1\\
(rs)^{-1}& 1 & s & s& rs^{-1}& r^{-1}& 1\\
(rs)^{-1}& 1 &1 & 1& s &rs^{-1} & r^{-1} \\
 (rs)^{-1} & 1 & 1 &1 & 1 & s &rs^{-1}
\end{array}\right).$

\begin{equation*}\lg {\om_i'}^{\pm1},\om_j^{-1}\rg=\lg
{\om_i'}^{\pm1},\om_j\rg^{-1}=\lg \om_i',\om_j\rg^{\mp1},\qquad
\text{\it for any } \ \mathfrak {\hat g}.  \tag{2}
\end{equation*}

\smallskip

\begin{remark} \ The above structure constant matrix $J$ is said the two-parameter quantum affine Cartan matrix,
which is the generalization of the classical quantum affine Cartan matrix under the condition $r=s^{-1}=q$.
\end{remark}

\begin{defi}
Let $U_{r,s}(\hat{\mathfrak g})$ be the unital associative algebra
over $\mathbb{K}$ generated by the elements $e_j,\, f_j,\,
\omega_j^{\pm 1},\, \omega_j'^{\,\pm 1}\, (j\in I_0),\,
\gamma^{\pm\frac{1}2},\,\gamma'^{\pm\frac{1}2},\, D^{\pm1},
D'^{\,\pm1}$, satisfying the following relations (where $c$ is the canonical central element of $\hat{\mathfrak g}$):

\medskip
\noindent $(\textrm{X1})$ \
$\gamma^{\pm\frac{1}2},\,\gamma'^{\pm\frac{1}2}$ are central with
$\gamma=\om'^{-1}_\delta$, $\gamma'=\om^{-1}_\delta$, $\gamma\gamma'=(rs)^c$, such
that $\omega_i\,\omega_i^{-1}=\omega_i'\,\omega_i'^{\,-1}=1
=DD^{-1}=D'D'^{-1}$, and
\begin{equation*}
\begin{split}[\,\omega_i^{\pm 1},\omega_j^{\,\pm 1}\,]&=[\,\om_i^{\pm1},
D^{\pm1}\,]=[\,\om_j'^{\,\pm1}, D^{\pm1}\,] =[\,\om_i^{\pm1},
D'^{\pm1}\,]=0\\
&=[\,\omega_i^{\pm 1},\omega_j'^{\,\pm 1}\,]=[\,\om_j'^{\,\pm1},
D'^{\pm1}\,]=[D'^{\,\pm1}, D^{\pm1}]=[\,\omega_i'^{\pm
1},\omega_j'^{\,\pm 1}\,].
\end{split}
\end{equation*}
 $(\textrm{X2})$ \ For $\,i,\,j\in I_0$,
\begin{equation*}
\begin{array}{ll}
& D\,e_i\,D^{-1}=r^{\delta_{0i}}\,e_i,\qquad\qquad\qquad\qquad\;
D\,f_i\,D^{-1}=r^{-\delta_{0i}}\,f_i,\\
&\omega_j\,e_i\,\omega_j^{\,-1}=\lg \om_i',\om_j\rg
\,e_i,\qquad\qquad\quad \omega_j\,f_i\,\omega_j^{\,-1}=\lg
\om_i',\om_j\rg^{-1}\,f_i.
\end{array}
\end{equation*}\\
$(\textrm{X3})$ \ For $\,i,\,j\in I_0$,
\begin{equation*}
\begin{array}{ll}
& D'\,e_i\,D'^{-1}=s^{\delta_{0i}}\,e_i,\qquad\qquad\qquad\quad\ \
D'\,f_i\,D'^{-1}=s^{-\delta_{0i}}\,f_i,\\
&\omega'_j\,e_i\,\omega'^{\,-1}_j=\lg \om_j',\om_i\rg^{-1}\,e_i,
\qquad\qquad  \omega'_j\,f_i\,\omega'^{\,-1}_j=\lg \om_j',\om_i\rg\,f_i.
\end{array}
\end{equation*}\\
$(\textrm{X4})$ \ For $\,i,\, j\in I_0$, we have
 $$[\,e_i, f_j\,]=\frac{\delta_{ij}}{r-s}(\omega_i-\omega'_i).$$
$(\textrm{X5})$  For any $i\ne j$, we have the $(r,s)$-Serre
relations
\begin{gather*}
\bigl(\text{ad}_l\,e_i\bigr)^{1-a_{ij}}\,(e_j)=0,\\
\bigl(\text{ad}_r\,f_i\bigr)^{1-a_{ij}}\,(f_j)=0,
\end{gather*}
where the definitions of the left-adjoint action $\text{ad}_l\,e_i$
and the right-adjoint action $\text{ad}_r\,f_i$ are given in the
following sense
$$
\text{ad}_{ l}\,a\,(b)=\sum_{(a)}a_{(1)}\,b\,S(a_{(2)}), \quad
\text{ad}_{ r}\,a\,(b)=\sum_{(a)}S(a_{(1)})\,b\,a_{(2)}, \quad
\forall\; a, b\in U_{r,s}(\hat{\mathfrak g}),
$$
where $\Delta(a)=\sum_{(a)}a_{(1)}\ot a_{(2)}$ is given by
Proposition 2.3 below.
\end{defi}

\smallskip
\subsection{Hopf algebra and Drinfeld double} The following is straightforward.
\begin{prop}  \ The algebra $U_{r, s}(\hat{\mathfrak g})$
$($\,$\hat{\mathfrak g}=A_{n-1}^{(1)}$ , $D_n^{(1)}$ and
$E_6^{(1)}$\,$)$ is a Hopf algebra under the comultiplication, the
counit and the antipode defined below  $\,(0 \leq i \leq n)$
\begin{gather*}
\Delta(\om_i^{\pm1})=\om_i^{\pm1}\ot\om_i^{\pm1}, \qquad
\Delta({\om_i'}^{\pm1})={\om_i'}^{\pm1}\ot{\om_i'}^{\pm1},\\
\Delta(e_i)=e_i\ot 1+\om_i\ot e_i, \qquad \Delta(f_i)=1\ot
f_i+f_i\ot \om_i',\\
\vn(\om_i^{\pm})=\vn({\om_i'}^{\pm1})=1, \qquad
\vn(e_i)=\vn(f_i)=0,\\
S(\om_i^{\pm1})=\om_i^{\mp1}, \qquad
S({\om_i'}^{\pm1})={\om_i'}^{\mp1},\\
S(e_i)=-\om_i^{-1}e_i,\qquad S(f_i)=-f_i\,{\om_i'}^{-1}.
\end{gather*}\hfill\qed
\end{prop}

\begin{remark}
\ When $r=s^{-1}=q$, Hopf algebra $U_{r, s}(\hat{\mathfrak g})$
modulo the Hopf ideal generated by the elements $\om_i'-\om_i^{-1}$
$(0\le i\le n)$, is just the quantum group $U_q(\hat{\mathfrak g})$
of Drinfel'd-Jimbo type.
\end{remark}

\begin{defi} \ {\it A skew-dual pairing of two Hopf
algebras ${\mathcal A}$ and ${\mathcal U}$ is a bilinear form
$\langle\,,\rangle:\; {\mathcal U}\times {\mathcal A}\lra \mathbb K$
such that
\begin{gather*} \langle f, 1_{\mathcal A}\rangle=\vn_{\mathcal U}(f),\qquad
\langle 1_{\mathcal U}, a\rangle=\vn_{\mathcal A}(a),\\
\langle f, a_1a_2\rangle=\langle \Delta^{\text{op}}_{\mathcal U}(f),
a_1\ot a_2\rangle, \qquad \langle f_1f_2, a \rangle=\langle f_1\ot
f_2, \Delta_{\mathcal A}(a)\rangle,
\end{gather*}
for all $f,\, f_1,\, f_2\in\mathcal U$, and
$a,\,a_1,\,a_2\in\mathcal A$, where $\vn_{\mathcal U}$ and
$\vn_{\mathcal A}$ denote the counits of $\mathcal U$ and $\mathcal
A$, respectively, and $\Delta_{\mathcal U}$ and $\Delta_{\mathcal
A}$ are their respective comultiplications.}
\end{defi}

Let $\hat{\mathcal B}=\hat {B}(\hat{\mathfrak g})$ (resp.
$\hat{\mathcal B'}=\hat {B'}(\hat{\mathfrak g})$\,) denote the Hopf
subalgebra of $\hat {U}=U_{r,s}(\hat {\mathfrak g})$ generated by
$e_j$, $\om_j^{\pm1}$ (resp. $f_j$, ${\om_j'}^{\pm1}$\,) with $0\le
j\le n$ for $\hat{\mathfrak g}=A_{n-1}^{(1)}$, and with $0\le j\le
n$ for $\hat{\mathfrak g}=D_n^{(1)}$, respectively. The following
result was obtained for the type $A_{n-1}^{(1)}$ case by \cite{HRZ}.

\begin{prop}  There exists a unique skew-dual pairing
$\langle\,,\rangle:\, \hat{\mathcal B'}\times \hat{\mathcal
B}\lra\mathbb K$ of the Hopf subalgebras $\hat{\mathcal B}$ and
$\hat{\mathcal B'}$ in $U_{r,s}(\hat{\mathfrak g})$ such that $\lg
f_i, e_j\rg=\frac{\delta_{ij}}{s_i-r_i}$, and the conditions $(1_X)$
and $(2)$ are satisfied, and all other pairs of generators are $0$.
Moreover, we have $\lg S(a), S(b)\rg=\lg a, b\rg$ for
$a\in\hat{\mathcal B'},\,b\in\hat{\mathcal B}$.\hfill\qed
\end{prop}

\begin{defi} \ {\it For any two skew-paired Hopf algebras
$\mathcal A$ and $\mathcal U$ by a skew-dual pairing $\lg\,,\rg$,
one may form the Drinfel'd double $\mathcal D(\mathcal A,\mathcal
U)$ as in $[\rm KS, 8.2]$, which is a Hopf algebra whose underlying
coalgebra is $\mathcal A\ot\mathcal U$ with the tensor product
coalgebra structure, and whose algebra structure is defined by
$$
(a\ot f)(a'\ot f')=\sum \lg S_{\mathcal U}(f_{(1)}), a'_{(1)}\rg\lg
f_{(3)},a'_{(3)}\rg \,aa'_{(2)}\ot f_{(2)}f',\leqno(3)$$ for $a,
a'\in \mathcal A$ and $f, f'\in\mathcal U$. The antipode $S$ is
given by}
$$
S(a\ot f)=(1\ot S_{\mathcal U}(f))(S_{\mathcal A}(a)\ot 1).
$$
\end{defi}

Clearly, both mappings $\mathcal A\ni a\mapsto a\ot 1\in\mathcal
D(\mathcal A,\mathcal U)$ and $\mathcal U\ni f\mapsto 1\ot
f\in\mathcal D(\mathcal A, \mathcal U)$ are injective Hopf algebra
homomorphisms. Let us denote the image $a\ot 1$ (resp. $1\ot f$) of
$a$ (resp. $f$) in $\mathcal D(\mathcal A,\mathcal U)$ by $\hat a$
(resp. $\hat f$). By (3), we have the following cross commutation
relations between elements $\hat a$ (for $a\in\mathcal A$) and $\hat
f$ (for $f\in\mathcal U$) in the algebra $\mathcal D(\mathcal
A,\mathcal U)$:
\begin{gather*} \hat f\,\hat a=\sum\, \lg
S_{\mathcal U}(f_{(1)}), a_{(1)}\rg\,\lg
f_{(3)},a_{(3)}\rg\;\hat a_{(2)}\hat f_{(2)},\tag{4}\\
\sum\lg f_{(1)}, a_{(1)}\rg\,\hat f_{(2)}\,\hat a_{(2)}= \sum \hat
a_{(1)}\,\hat f_{(1)}\,\lg f_{(2)},a_{(2)}\rg.\tag{5}
\end{gather*}
In fact, as an algebra the double $\mathcal D(\mathcal A,\mathcal
U)$ is the universal algebra generated by the algebras $\mathcal A$
and $\mathcal U$ with cross relations (4) or, equivalently, (5).

\begin{theo} \ The two-parameter quantum affine
algebra $U=U_{r,s}(\hat{\mathfrak g})$  is isomorphic to the
Drinfel'd quantum double $\mathcal D(\hat{\mathcal B}, \hat{\mathcal
B'})$.
\end{theo}

\subsection{Triangular decomposition of $U_{r,s}(\hat {\mathfrak g})$}\,
Let $U_0=\mathbb K[\om_0^{\pm1},\cdots,\om_n^{\pm1}]$, $U^0=\mathbb
K[\om_0^{\pm1},\cdots,\om_n^{\pm1},{\om_0'}^{\pm1},\cdots,{\om_n'}^{\pm1}]$,
and $U_0'=\mathbb
K[{\om_0'}^{\pm1},\cdots,{\om_n'}^{\pm1}]$ denote the Laurent
polynomial subalgebras of $U_{r,s}(\widehat{\frak {g}})$,
$\widehat{\mathcal B}$,
 and $\widehat{\mathcal B'}$ respectively. Clearly, $U^0=U_0U_0'=U_0'U_0$.
Denote by $U_{r,s}(\widehat{\frak n})$ $($resp.
$U_{r,s}(\widehat{\frak n}^-)$\,$)$ the subalgebra of
$\widehat{\mathcal B}$ $($resp. $\widehat{\mathcal B'})$ generated
by $e_i$ $($resp. $f_i$$)$ for all $i\in I_0$. By definition,
$\widehat{\mathcal B}=U_{r,s}(\widehat{\frak n})\rtimes
U_0$, $\widehat{\mathcal B'}=U_0'\ltimes U_{r,s}(\widehat{\frak
n}^-)$, so that the double $\mathcal D(\widehat{\mathcal
B},\widehat{\mathcal B'})\cong U_{r,s}(\widehat{\frak n})\ot U^0\ot
U_{r,s}(\widehat{\frak n}^-)$, as vector spaces. On the other hand,
if we consider $\la\, ,\, \ra^-: \widehat{\mathcal B'}\times
\widehat{\mathcal B}\lra\mathbb K$ by $\la b', b\ra^-:=\la S(b'),
b\ra$, the convolution inverse of the skew-dual paring $\la\, ,\ra$
in Proposition 2.6, the composition with the flip mapping $\sigma$
then gives rise to a new skew-dual paring
$\la\,|\,\ra:=\la\,,\ra^-\circ\sigma: \widehat{\mathcal
B}\times\widehat{\mathcal B'}\lra \mathbb K$, given by $\la
b|b'\ra=\la S(b'),b\ra$. As a byproduct of Theorem 2.8 (see
[\,BGH1, Coro. 2.6]), we get the standard triangular decomposition of
$U_{r,s}(\widehat{\frak{g}})$.

\begin{coro}
$U_{r,s}(\widehat{\frak{g}})\cong U_{r,s}(\widehat{\frak n}^-)\ot
U^0\ot U_{r,s}(\widehat{\frak n})$, as vector spaces.\hfill\qed
\end{coro}

\begin{defi} $($Prop. 3.2 \cite{HRZ}$)$
Let $\tau$ be the $\mathbb{Q}$-algebra anti-automorphism of
$U_{r,s}(\widehat{\frak {g}})$ such that $\tau(r)=s$, $\tau(s)=r$,
$\tau(\la \om_i',\om_j\ra^{\pm1})=\la \om_j',\om_i\ra^{\mp1}$, and
\begin{gather*}
\tau(e_i)=f_i, \quad \tau(f_i)=e_i, \quad \tau(\om_i)=\om_i',\quad
\tau(\om_i')=\om_i,\\
\tau(\gamma)=\gamma',\quad
\tau(\gamma')=\gamma,\quad\tau(D)=D',\quad \tau(D')=D.
\end{gather*}
Then ${\widehat{\mathcal B'}}=\tau({\widehat{\mathcal B}})$ with
those induced defining relations from ${\widehat{\mathcal B}}$, and
those cross relations in $(\textrm{X2})$---$(\textrm{X4})$ are
antisymmetric with respect to $\tau$.\hfill\qed
\end{defi}

\smallskip
\section{Drinfeld Realization via Generating
Functions with $\tau$-invariance}

\subsection{Generating functions with $\tau$-invariance and Drinfeld realization}
In order to obtain the intrinsic definition of Drinfeld realization of the
two-parameter quantum affine algebra $U_{r,s}(\widehat{\frak{g}})$,
we need first to construct the generating functions $g_{ij}^{\pm}(z)$ ($1\le i,\,j\le n$) with $\tau$-invariance, which
is due to the first author defined as follows (This was initially motivated in part by section 1.1 in \cite{Gr} in the one-parameter setting regardless of $\tau$-invariance there).

\smallskip
Let $\al_i,\,\al_j \in \Delta$, we set $g_{ij}^{\pm}(z)=\sum_{n\in
\mathbb{Z}_+}{^{\pm}}c_{\al_i,\,\al_j}^{(n)}z^{n}:=\sum_{n\in
\mathbb{Z}_+}{^{\pm}}c^{(n)}_{ij}z^{n}$, a formal power series in $z$,
where the coefficients $^\pm c^{(n)}_{ij}$ are determined from the Taylor
series expansion in the variable $z$ at $0\in \mathbb{C}$ of the
function
$$
\sum_{n\in \mathbb{Z}_+}{^{\pm}}c^{(n)}_{ij}z^{n}=g^{\pm}_{ij}(z)=
\frac{G_{ij}^{\pm}(z, 1)} {F_{ij}^{\pm}(z, 1)},
$$
where we got some observations from the discussions for type $A^{(1)}_n$ we did in \cite{HRZ} to define $F_{ij}^\pm(z,\,w),\,G_{ij}^\pm(z,\,w)$ as follows.
\begin{gather*}
F_{ij}^\pm(z,w):=z-(\lg i,j\rg\lg j,i\rg)^{\pm\frac1{2}}w,\\
G_{ij}^\pm(z,w):=\lg j,i\rg^{\pm1}z-(\lg j,i\rg\lg
i,j\rg^{-1})^{\pm\frac1{2}}w.
\end{gather*}

We have a uniform formula for $g^{\pm}_{ij}(z)$ as below
\begin{eqnarray}
g^{\pm}_{ij}(z)&=&\frac{\lg j,i\rg^{\pm1}z-(\lg j,i\rg\lg
i,j\rg^{-1})^{\pm\frac1{2}}}{z-(\lg i,j\rg\lg j,i\rg)^{\pm\frac1{2}}} \vspace{2mm}\nonumber\\
&=&(\lg j,i\rg\lg i,j\rg^{-1})^{\pm\frac1{2}}\frac{(\lg i,j\rg\lg
j,i\rg)^{\pm\frac1{2}}z-1}{z-(\lg i,j\rg\lg j,i\rg)^{\pm\frac1{2}}}.
\end{eqnarray}

Then in both cases (whenever $i=j$ or $i\ne j$), we have a uniform
expansion formula for $g^{\pm}_{ij}(z)=\sum_{k\ge0}{^\pm
c}_{ij}^{(k)}z^k$ with ${^\pm c}_{ij}^{(0)}= \lg i,j\rg^{\mp1}$ and
$$
{^\pm c}_{ij}^{(k)}={^\pm c}_{ij}^{(0)} \lg
i,i\rg^{\mp\frac{(k{-}1)a_{ij}}2}\left(\lg
i,i\rg^{\mp\frac{a_{ij}}2}-\lg i,i\rg^{\pm \frac{a_{ij}}{2}}\right),
\quad \textit{for} \ k>0.
$$

Define the other generating functions in a formal variable $z$ (also see \cite{HRZ}):
\begin{gather*}
\delta(z)=\sum_{n\in\mathbb{Z}}z^n, \qquad\quad\qquad x_i^{\pm}(z) = \sum_{k \in \mathbb{Z}}x_i^{\pm}(k) z^{-k},\\
\om_i(z) = \sum_{m \in \mathbb{Z}_+}\om_i(m) z^{-m}=\om_i\exp\Bigl((r{-}s)\sum_{\ell\ge1}a_i(\ell)z^{-\ell}\Bigr), \\
\om'_i(z) = \sum_{n \in \mathbb{Z}_+}\om'_i(-n) z^n=\om'_i\exp\Bigl(-(r{-}s)\sum_{\ell\ge1}a_i(-\ell)z^\ell\Bigr).
\end{gather*}

The following property for the generating functions $g^{\pm}_{ij}(z)$ is rather crucial for deriving the inherent definition of
Drinfeld realization in the two-parameter version. 

\begin{prop} $(${\bf The $\tau$-invariance of generating functions}$)$ \
Assume $\tau(r)=s$, $\tau(s)=r$, $\tau(\lg i, j\rg)=\lg j,i\rg^{-1}$, $\tau(z)=z^{-1}$, $\tau(x_i^{\pm}(k))=x_i^{\mp}(-k)$, $\tau(\om_i(m))=\om'_i(-m)$,
$\tau(\om'_i(-m))=\om_i(m)$, for $m\in\mathbb Z_+$ with $\om_i(0)=\om_i$, $\om'_i(0)=\om_i'$, then

$(\text{\rm i})$ $\tau(g_{ij}^{\pm}(z))=g_{ij}^{\pm}(z)$, and $g^{\pm}_{ij}(z)^{-1}=g^{\mp}_{ij}(z)$, $g^{\pm}_{ij}(z^{-1})=g^{\mp}_{ji}(z)=g_{ji}^{\pm}(z)^{-1}$.

$(\text{\rm ii})$
$\tau(\delta(z))=\delta(z)$, $\tau(x_i^{\pm}(z))=x_i^{\mp}(z)$, $\tau(\om_i(z))=\om_i'(z)$ and $\tau(\om_i'(z))=\om_i(z)$.
\end{prop}

Now let us formulate the {\bf inherent definition of Drinfeld realization for
two-parameter quantum affine algebra} $U_{r,s}(\widehat{\frak{g}})$ via our generating functions with $\tau$-invariance. 
Set $r_i=r^{d_i}$, $s_i=s^{d_i}$, where $A=DB$, $D=\text{diag}\{d_0,{\cdots},d_n\}$ and $B$ is symmetric.

\begin{defi} $(${\bf Theorem.}$)$  The $(r,s)$-Drinfeld realization $\mathcal U_{r,s}(\widehat{\mathfrak{g}})$ associated to the two-parameter quantum affine algebra
$U_{r,s}(\widehat{\mathfrak{g}})$
is the assaciative algebra with unit
$1$ and generators
$
\bigl\{ x_i^{\pm}(k), \, \om_i(m), \, \om'_i(-n), \,
\gamma^{\pm\frac{1}{2}}, \, {\gamma'}^{\pm\frac{1}{2}},\, D^{\pm1},\, D'^{\,\pm1}\,\big|\, i\in I, ~k \in
\mathbb{Z}, \, m, n \in \mathbb{Z}_+\bigr\}
$
satisfying the relations below with $\tau$-invariance, written in terms of $\tau$-invariant generating functions of formal variables $z$, $w$ with $\mathbb Q$-anti-involution $\tau$ such that $\tau(\ga)=\ga'$, $\tau(\ga')=\ga$ (where $c$ is the canonical central element of $\hat{\mathfrak g}$ and $\tau$ is defined as in Proposition 3.1, set $g_{ij}(z):=g^+_{ij}(z)$):
\begin{eqnarray}
&&\quad\gamma^{\pm\frac1{2}},\, \gamma'^{\pm\frac1{2}} \ \hbox{are
central and mutually inverse such that } \gamma\gamma'=(rs)^c,
\\
&&\quad\om_{i}(0)^{\pm1}, \ \om'_{j}(0)^{\pm1} \ \hbox{mutually commute}, \ \text{where } \om_i(0)=\om_i, \ \om'_j(0)=\om_j',\\
&&\quad \om_i(z)\om_j(w)=\om_j(w)\om_i(z),\ \qquad \om'_i(z)\om'_j(w)=\om'_j(w)\om'_i(z),\\
&&\quad g_{ij}\Bigl(zw^{-1}(\gamma\gamma')^{\frac{1}{2}}\gamma\Bigr)\om'_i(z)\om_j(w)
=g_{ij}\Bigl(zw^{-1}(\gamma\gamma')^{\frac{1}{2}}\gamma'\Bigr)
\om_j(w)\om'_i(z),\\
&&\quad Df_i(z)D^{-1}{=}f_i\Bigl(\frac{z}{r_i}\Bigr), D'f_i(z)D'^{-1}{=}f_i\Bigl(\frac{z}{s_i}\Bigr), \textit{ for } f_i(z){=}x^{\pm}_i(z), \om_i(z), \om'_i(z),\\
&&\quad \om'_i(z)x_j^{\pm}(w)\om'_i(z)^{-1}
=g_{ij}\Bigl(\frac{z}{w}(\gamma\gamma')^{\frac{1}{2}}\gamma^{\mp
\frac{1}{2}}\Bigr)^{\pm1}x_j^{\pm}(w),  \\
&&\quad \om_i(z)x_j^{\pm}(w)\om_i(z)^{-1}=g_{ji}\Bigl(\frac{w}{z}(\gamma\gamma')^{\frac{1}{2}}\gamma'^{\pm
\frac{1}{2}}\Bigr)^{\mp1}x_j^{\pm}(w), \\
&&\quad [\,x_i^+(z),
x_j^-(w)\,]=\frac{\delta_{ij}}{r_i-s_i}\Big(\delta(zw^{-1}\gamma')\om_i(w\gamma^{\frac{1}2})
-\delta(zw^{-1}\gamma)\om'_i(z\gamma'^{-\frac1{2}})\Big),\\
&&\quad F_{ij}^\pm(z,\,w)\,x_i^{\pm}(z)x_j^{\pm}(w)=G_{ij}^\pm(z,\,w)\,x_j^{\pm}(w)\,x_i^{\pm}(z),\\
&&\quad x_i^{\pm}(z)x_j^{\pm}(w)=\lg
j,i\rg^{\pm1}x_j^{\pm}(w)x_i^{\pm}(z), \qquad\hbox{for }
\ a_{ij}=0,\\
&&\quad Sym_{z_1,\cdots,
z_n}\sum_{k=0}^{n=1-a_{ij}}(-1)^k(r_is_i)^{\pm\frac{k(k-1)}{2}}
\Big[{1-a_{ij}\atop  k}\Big]_{\pm{i}}x_i^{\pm}(z_1)\cdots x_i^{\pm}(z_k) x_j^{\pm}(w)\\
&&\hskip1.8cm \times x_i^{\pm}(z_{k+1})\cdots x_i^{\pm}(z_{n})=0,
\quad\hbox{for} \quad a_{ij}< 0, \quad  1\leq i<j<n,\nonumber\\
&&\quad Sym_{z_1,\cdots,
z_n}\sum_{k=0}^{n=1-a_{ij}}(-1)^k(r_is_i)^{\mp\frac{k(k-1)}{2}}
\Big[{1-a_{ij}\atop  k}\Big]_{\mp{i}}x_i^{\pm}(z_1)\cdots x_i^{\pm}(z_k) x_j^{\pm}(w)\\
&&\hskip1.8cm \times x_i^{\pm}(z_{k+1})\cdots x_i^{\pm}(z_n)=0,
\quad\hbox{for } \quad a_{ij}< 0, \quad  1\leq j<i<n,\nonumber
\end{eqnarray}
where $\textit{Sym}_{z_1,\cdots, z_n}$  denotes symmetrization w.r.t. the
indices $(z_1, {\cdots}, z_n)$.
In particular, $\tau$ keeps each term among the relations (3.2)---(3.6), (3.9)---(3.11); but interchanges the relations between (3.7) and (3.8), the ones between (3.12) and (3.13).
\end{defi}

\begin{remark} \ (1) \ When $r=q=s^{-1}$, $g_{ij}(z)$ is the same as that of the one-parameter quantum affine algebras (cf. \cite{Gr}).

(2) \ When $r=q=s^{-1}$, the algebra ${\mathcal U}_{q,
q^{-1}}(\widehat{\frak {g}})$ modulo the ideal generated by the set
$\{\,\om_i'-\om_i^{-1}$ $(i\in I)$,
$\gamma'^{\,\frac{1}2}-\gamma^{-\frac{1}2}\,\}$, is the usual Drinfeld
realization ${\mathcal U}_q(\widehat{\frak {g}})$.

(3) \ Denote ${^\pm
\bar c}_{ij}^{(k)}:={^\pm c}_{ij}^{(k)}/{^\pm c}_{ij}^{(0)}$, $t:=r{-}s$. (3.7) is equivalent to the following
$$\textrm{exp}\bigl(-t\sum_{\ell>0}a_i(-\ell)z^\ell\bigr)\cdot
\,x_j^\pm(w)\cdot\textrm{exp}\bigl(t\sum_{\ell>0}a_i(-\ell)z^\ell\bigr)
=\sum_{k\ge0}{^\pm
\bar c}_{ij}^{(k)}\gamma^{\mp\frac{k}2}\bigl(\frac{(\gamma\gamma')^{\frac{1}2}z}w\bigr)^kx_j^\pm(w).$$
(3.8) is equivalent to the following
$$\textrm{exp}\bigl(t\sum_{\ell>0}a_i(\ell)z^{-\ell}\bigr)\cdot
\,x_j^\pm(w)\cdot\textrm{exp}\bigl(-t\sum_{\ell>0}a_i(\ell)z^{-\ell}\bigr)
=\sum_{k\ge0}{^\mp\bar
c}_{ji}^{(k)}\gamma'^{\pm\frac{k}2}\bigl(\frac{(\gamma\gamma')^{\frac{1}2}w}z\bigr)^kx_j^\pm(w).$$
Both expansions give rise to the relations $($D6$_1)$ and
$($D6$_2)$ below.
\end{remark}

\begin{defi}
({\bf Equivalent Definition.}) The unital associative algebra ${\mathcal
U}_{r,s}(\widehat{\frak {g}})$
 over $\mathbb{K}$  is generated by the
elements  $x_i^{\pm}(k)$, $a_i(\ell)$, $\om_i^{\pm1}$,
${\om'_i}^{\pm1}$, $\gamma^{\pm\frac{1}{2}}$,
${\gamma'}^{\,\pm\frac{1}2}$, $D^{\pm1}$, $D'^{\,\pm1}$ $(i\in I$,
$k,\,k' \in \mathbb{Z}$, $\ell\in \mathbb{Z}\backslash
\{0\})$, subject to the following defining relations:

\medskip
\noindent $(\textrm{D1})$ \  $\gamma^{\pm\frac{1}{2}}$,
$\gamma'^{\,\pm\frac{1}{2}}$ are central such that
$\gamma\gamma'=(rs)^c $,\,
$\omega_i\,\omega_i^{-1}=\omega_i'\,\omega_i'^{\,-1}=1$ $(i\in I)$,
and for $i,\,j\in I$, one has
\begin{equation*}
\begin{split}
[\,\omega_i^{\pm 1},\omega_j^{\,\pm 1}\,]&=[\,\om_i^{\pm1},
D^{\pm1}\,]=[\,\om_j'^{\,\pm1}, D^{\pm1}\,] =[\,\om_i^{\pm1},
D'^{\pm1}\,]=0\\
&=[\,\omega_i^{\pm 1},\omega_j'^{\,\pm 1}\,]=[\,\om_j'^{\,\pm1},
D'^{\pm1}\,]=[D'^{\,\pm1}, D^{\pm1}]=[\,\omega_i'^{\pm
1},\omega_j'^{\,\pm 1}\,].
\end{split}
\end{equation*}
$$[\,a_i(\ell),a_j(\ell')\,]
=\delta_{\ell+\ell',0}\frac{ (\ga\ga')^{\frac{|\ell|}{2}}(\la
i,\,i\ra^{\frac{\ell a_{ij}}{2}}- \la i,\,i\ra^{\frac{-\ell
a_{ij}}{2}})} {|\ell|(r_i-s_i)}
\cdot\frac{\gamma^{|\ell|}-\gamma'^{|\ell|}}{r_i-s_i}.
 \leqno(\textrm{D2})
$$
$$[\,a_i(\ell),~\om_j^{{\pm }1}\,]=[\,\,a_i(\ell),~{\om'}_j^{\pm
1}\,]=0.\leqno(\textrm{D3})
$$
$$
\begin{array}{lll}
D\,x_i^{\pm}(k)\,D^{-1}=r^k\, x_i^{\pm}(k), \qquad\ \
D'\,x_i^{\pm}(k)\,D'^{\,-1}=s^k\, x_i^{\pm}(k),
\\
D\, a_i(\ell)\,D^{-1}=r^\ell\,a_i(\ell), \qquad\qquad D'\,
a_i(\ell)\,D'^{\,-1}=s^\ell\,a_i(\ell).
\end{array}\leqno{(\textrm{D4})}
$$
$$
\om_i\,x_j^{\pm}(k)\, \om_i^{-1} =  \langle \omega_j',
\omega_i\rangle^{\pm 1} x_j^{\pm}(k), \qquad \om'_i\,x_j^{\pm}(k)\,
\om_i'^{\,-1} =  \langle \omega'_i, \omega_j\rangle
^{\mp1}x_j^{\pm}(k).\leqno(\textrm{D5})
$$
$$
\begin{array}{lll}
[\,a_i(\ell),x_j^{\pm}(k)\,]=\pm\frac{(\ga\ga')^{\frac{\ell}{2}} (\la
i,\,i\ra^{\frac{\ell a_{ij}}{2}}-\la i,\,i\ra^{\frac{-\ell
a_{ij}}{2}})}
{\ell(r_i-s_i)}\gamma'^{\pm\frac{\ell}2}x_j^{\pm}(\ell{+}k),\quad
\textit{for} \quad \ell>0,
\end{array}\leqno{(\textrm{D$6_1$})}
$$
$$
\begin{array}{lll}
[\,a_i(\ell),x_j^{\pm}(k)\,]=\pm\frac{(\ga\ga')^{\frac{-\ell}{2}}(\la
i,\,i\ra^{\frac{\ell a_{ij}}{2}}-\la i,\,i\ra^{\frac{-\ell
a_{ij}}{2}})}
{\ell(r_i-s_i)}\gamma^{\pm\frac{\ell}2}x_j^{\pm}(\ell{+}k), \quad
\textit{for} \quad \ell<0.
\end{array}\leqno{(\textrm{D$6_2$})}
$$
$$
\begin{array}{lll}
x_i^{\pm}(k{+}1)\,x_j^{\pm}(k') - \langle j,i\rangle^{\pm1} x_j^{\pm}(k')\,x_i^{\pm}(k{+}1)\\
=-\Bigl(\langle j,i\rangle\langle
i,j\rangle^{-1}\Bigr)^{\pm\frac1{2}}\,\Bigl(x_j^{\pm}(k'{+}1)\,x_i^{\pm}(k)-\langle
i,j\rangle^{\pm1} x_i^{\pm}(k)\,x_j^{\pm}(k'{+}1)\Bigr).
\end{array}\leqno{(\textrm{D7})}
$$

$$
[\,x_i^{+}(k),~x_j^-(k')\,]=\frac{\delta_{ij}}{r_i-s_i}\Big(\gamma'^{-k}\,{\gamma}^{-\frac{k+k'}{2}}\,
\om_i(k{+}k')-\gamma^{k'}\,\gamma'^{\frac{k+k'}{2}}\,\om'_i(k{+}k')\Big),\leqno(\textrm{D8})
$$
where $\om_i(m)$, $\om'_i(-m)~(m\in \mathbb{Z}_{\geq 0})$ such that
$\om_i(0)=\om_i$ and  $\om'_i(0)=\om_i'$ are defined as below:
\begin{gather*}\sum\limits_{m=0}^{\infty}\om_i(m) z^{-m}=\om_i \exp \Big(
(r_i{-}s_i)\sum\limits_{\ell=1}^{\infty}
 a_i(\ell)z^{-\ell}\Big),\quad \bigl(\om_i(-m)=0, \ \forall\;m>0\bigr); \\
\sum\limits_{m=0}^{\infty}\om'_i(-m) z^{m}=\om'_i \exp
\Big({-}(r_i{-}s_i)
\sum\limits_{\ell=1}^{\infty}a_i(-\ell)z^{\ell}\Big), \quad
\bigl(\om'_i(m)=0, \ \forall\;m>0\bigr).
\end{gather*}
$$x_i^{\pm}(m)x_j^{\pm}(k)=\langle j,i\rangle^{\pm1}x_j^{\pm}(k)x_i^{\pm}(m),
\qquad\ \hbox{for} \quad a_{ij}=0,\leqno(\textrm{D$9_1$})$$
$$
\begin{array}{lll}
& Sym_{m_1,\cdots
m_{n}}\sum_{k=0}^{n=1-a_{ij}}(-1)^k(r_is_i)^{\pm\frac{k(k-1)}{2}}
\Big[{1-a_{ij}\atop  k}\Big]_{\pm{i}}x_i^{\pm}(m_1)\cdots x_i^{\pm}(m_k) x_j^{\pm}(\ell)\\
&\hskip1.8cm \times x_i^{\pm}(m_{k+1})\cdots x_i^{\pm}(m_{n})=0,
\quad\hbox{for} \quad a_{ij}\neq 0, \quad  1\leq i<j<n,
\end{array} \leqno{(\textrm{D$9_2$})}
$$
$$
\begin{array}{lll}
& Sym_{m_1,\cdots
m_{n}}\sum_{k=0}^{n=1-a_{ij}}(-1)^k(r_is_i)^{\mp\frac{k(k-1)}{2}}
\Big[{1-a_{ij}\atop  k}\Big]_{\mp{i}}x_i^{\pm}(m_1)\cdots x_i^{\pm}(m_k) x_j^{\pm}(\ell)\\
&\hskip1.8cm \times x_i^{\pm}(m_{k+1})\cdots x_i^{\pm}(m_{n})=0,
\quad\hbox{for } \quad a_{ij}\neq 0, \quad  1\leq j<i<n,
\end{array} \leqno{(\textrm{D$9_3$})}
$$
where $[m]_{\pm{i}}=\frac{r_i^{\pm m}-s_i^{\pm m}}{r_i-s_i}$, $[m]_{\pm i}!=[m]_{\pm i}\cdots[2]_{\pm i}[1]_{\pm i}$,
$\Bigl[{m\atop n}\Bigr]_{\pm i}=\frac{[m]_{\pm i}!}{[n]_{\pm i}![m-n]_{\pm i}!}$,
$\textit{Sym}_{m_1,\cdots, m_n}$  denotes symmetrization w.r.t. the
indices $(m_1, \cdots, m_n)$.
\end{defi}

As one of crucial observations of considering the compatibilities of
the defining system above, we have
\begin{prop}
There exists the $\mathbb{Q}$-algebra antiautomorphism $\tau$ of
$\,{\mathcal U}_{r,s}(\widehat{\frak {g}})$  such that $\tau(r)=s$,
$\tau(s)=r$,
$\tau(\la\om_i',\om_j\ra^{\pm1})=\la\om_j',\om_i\ra^{\mp1}$ and
\begin{gather*}
\tau(\om_i)=\om_i',\quad \tau(\om_i')=\om_i,\quad
\tau(\gamma)=\gamma',\quad \tau(\gamma')=\gamma,\quad\tau(D)=D',\quad \tau(D')=D,\\
\tau(x_i^{\pm}(m))=x_i^{\mp}(-m), \quad
\tau(a_i(\ell))=a_i(-\ell),\\
\tau(\phi_i(m))=\varphi_i(-m), \quad\tau(\varphi_i(-m))=\phi_i(m),
\end{gather*}
and $\tau$ preserves each defining relation $($\hbox{{\rm D}n}$)$ in
Definition 3.1 for $n=1,\cdots,n$.\hfill\qed
\end{prop}

\subsection{Quantum Lie bracket} \ In this paragraph, we first  establish an
algebraic isomorphism
 between  the two realizations
of two-parameter quantum affine algebras
$U_{r,s}(\widehat{\frak{g}})$ in the above, which is called Drinfeld
isomorphism in one-parameter quantum affine algebras. We need to
make some preliminaries on the
definition of quantum Lie bracket that appears to be regardless to
degrees of relative elements (see the properties (3.16) \& (3.17)
below). This a bit generalized quantum Lie bracket compared to the
one used in the usual construction of the quantum Lyndon basis (for
definition, see [R2]), which is consistent with the cases when
adding the bracketing on those corresponding Lyndon words, is
crucial to our proving later on.

\begin{defi} For $q_i\in \mathbb K^*=\mathbb{K}\backslash \{0\}$ and $i=1,2,\cdots s-1$,
The quantum Lie brackets $[\,a_1, a_2,\cdots,
a_s\,]_{(q_1,\,q_2,\,\cdots,\, q_{s-1})}$ and
$[\,a_1, a_2, \cdots,
a_s\,]_{\la q_1,\,q_2,\,\cdots, \,q_{s-1}\ra}$ are defined
inductively by
\begin{eqnarray*}
\begin{split}
[\,a_1, a_2\,]_{q_1}&=a_1a_2-q_1\,a_2a_1,\\
[\,a_1, a_2, \cdots, a_s\,]_{(q_1,\,q_2,\,\cdots,
\,q_{s-1})}&=[\,a_1, [\,a_2, \cdots,
a_s\,]_{(q_2,\,\cdots,\,q_{s-1})}\,]_{q_{1}},\\
[\,a_1, a_2, \cdots, a_s\,]_{\la q_1,\,q_2,\,\cdots,
\,q_{s-1}\ra}&=[\,[\,a_1, \cdots, a_{s-1}\,]_{\la
q_1,\,\cdots,\,q_{s-2}\ra}, a_s\,]_{q_{s-1}},
\end{split}
\end{eqnarray*}
\end{defi}
By consequences of the above definitions, the following identities follow
\begin{eqnarray}
&&[\,a, bc\,]_v=[\,a, b\,]_q\,c+q\,b\,[\,a, c\,]_{\frac{v}q},\\
&&[\,ab, c\,]_v=a\,[\,b, c\,]_q+q\,[\,a, c\,]_{\frac{v}q}\,b, \\
&& [\,a,[\,b,c\,]_u\,]_v=[\,[\,a,b\,]_q,
c\,]_{\frac{uv}q}+q\,[\,b,[\,a,c\,] _{\frac{v}q}\,]_{\frac{u}q},\label{b:1}\\
&&[\,[\,a,b\,]_u,c\,]_v=[\,a,[\,b,c\,]_q\,]_{\frac{uv}q}+q\,[\,[\,a,c\,]
_{\frac{v}q},b\,]_{\frac{u}q}.\label{b:2}
\end{eqnarray}

In particular, we get immediately,
\begin{eqnarray}
&&[\,a, [\,b_1, \cdots, b_s\,]_{(v_1,\,\cdots,\,
v_{s-1})}\,]=\sum\limits_i[\,b_1,\cdots,[\,a, b_i\,],
\cdots,b_s\,]_{(v_1,\,\cdots,\, v_{s-1})},\label{b:3}\hskip0.2cm \\
&&[\,a, a, b\,]_{(u,\,
v)}=a^2b-(u{+}v)\,aba+(uv)\,ba^2=(uv)[\,b, a, a\,]_{\la u^{-1},v^{-1}\ra},\label{b:4}\hskip0.2cm \\
&&[\,a, a, a, b\,]_{(u^2,\,uv,\,v^2)}=a^3b-[3]_{u,v}\,a^2ba
+(uv)[3]_{u,v}aba^2-(uv)^3ba^3, \label{b:5}\hskip0.2cm\\
&&[\,a, a, a, a, b\,]_{(u^3, u^2v, uv^2,
v^3)}=a^4b-[4]_{u,v}\,a^3ba+uv\left[4\atop 2\right]_{u,v}\,a^2ba^2 \label{b:6}\hskip0.2cm\\
&&\hskip5cm-\,(uv)^3[4]_{u,v}\,aba^3+(uv)^6ba^4,\nonumber
\end{eqnarray}
where $[n]_{u,v}=\frac{u^n{-}v^n}{u{-}v}$,
$[n]_{u,v}!:=[n]_{u,v}\cdots [2]_{u,v}[1]_{u,v}$, $\left[n\atop
m\right]_{u,v}:=\frac{[n]_{u,v}!}{[m]_{u,v}![n-m]_{u,v}!}$.

By the definition above, the formula $(\textrm{D7})$ will take the
convenient form as
\begin{equation}{\label{c:1}
\left[\,x^{\pm}_i(k),\, x_j^{\pm}(k'{+}1)\,\right]_{\langle
i,j\rangle^{\mp1}}=-\Bigl(\langle j,i\rangle\langle
i,j\rangle^{-1}\Bigr)^{\pm\frac1{2}} \left[\,x^{\pm}_j(k'),\,
x_i^{\pm}(k{+}1)\,\right]_{\langle j,i\rangle^{\mp1}}. }
\end{equation}

\subsection{Quantum root vectors} \,

 Furthermore, for each $\alpha=\alpha_{i_1}+\alpha_{i_2}+\cdots
+\alpha_{i_n}:=\alpha_{i_1,i_2,\cdots,i_n}\in \dot\Delta^+$, by
[R2], we can construct the quantum root vector $x_\alpha^+(0)$ as a
$(r,s)$-bracketing in an inductive fashion, for more details, see \cite{HRZ} :
\begin{equation*}
\begin{split}
x_{\alpha}^+(0):&=\bigl[\,x_{\alpha_{i_1,i_2,\cdots,i_{n-1}}}^+(0),
x_{i_n}^+(0)\,\bigr]
_{\la \om_{\alpha_{i_1,i_2,\cdots,i_{n-1}}}',\,\om_{i_n}\ra^{-1}}\\
&=\bigl[\,\cdots\bigl[\,x_{i_1}^+(0), x_{i_2}^+(0)\,\bigr]_{\la
i_1,i_2\ra^{-1}},\cdots,x_{i_n}^+(0)\,\bigr] _{\la
\om_{\alpha_{i_1,i_2,\cdots,i_{n-1}}}',\,\om_{i_n}\ra^{-1}}.
\end{split}\tag{*}
\end{equation*}

Applying $\tau$ to (*), we can obtain the definition of quantum
root vector $x_{\alpha}^-(0)$ as below:
\begin{equation*}
\begin{split}
x_{\alpha}^-(0):&=
\bigl[\,x_{i_n}^-(0),\,x_{\alpha_{i_1,i_2,\cdots,i_{n-1}}}^-(0)\bigr]
_{\la\om'_{i_n},\, \om_{\alpha_{i_1,i_2,\cdots,i_{n-1}}}\ra}\\
&=\bigl[\,x_{i_n}^-(0) \cdots\bigl[\,x_{i_2}^-(0),
x_{i_1}^-(0)\,\bigr]_{\la i_2,i_1\ra}\cdots \bigr]
_{\la\om'_{i_n},\, \om_{\alpha_{i_1,i_2,\cdots,i_{n-1}}}\ra}.
\end{split}
\end{equation*}

\begin{defi} $($see Definition 3.9 in \cite{HRZ}$)$
For $\alpha=\alpha_{i_1,i_2,\cdots,i_n}\in\dot\Delta^+$, we define
the {\it quantum affine root vectors} $x_{\alpha}^\pm(k)$ of
nontrivial {\it level }$k$ by
\begin{gather*}
x_{\alpha}^+(k):=\bigl[\,\cdots\bigl[\,x_{i_1}^+(k),
x_{i_2}^+(0)\,\bigr]_{\la
i_1,i_2\ra^{-1}},\cdots,x_{i_n}^+(0)\,\bigr]_{\la
\om_{\alpha_{i_1,i_2,\cdots,i_{n-1}}}',\,\om_{i_n}\ra^{-1}}\,,\\
x_{\alpha}^-(k):=\bigl[\,x_{i_n}^-(0),\cdots, \bigl[\,x_{i_2}^-(0),
x_{i_1}^-(k)\,\bigr]_{\la
i_2,i_1\ra}\cdots\,\bigr]_{\la\om'_{i_n},\,
\om_{\alpha_{i_1,i_2,\cdots,i_{n-1}}}\ra}\,,
\end{gather*}
where $\tau\bigl(x_{\alpha}^\pm(\pm k)\bigr)=x_{\alpha}^\mp(\mp k)$.
\end{defi}
\begin{remark}\, Using the definition, we fix an ordering of the maximal root $\theta$, and give the maximal quantum root vectors $x_{\theta}^-(1)$ and $x_{\theta}^+(-1)$ as follows.

For the case of $A_n^{(1)}$, we fix the maximal root $\theta=\alpha_1+\alpha_2+\cdots+ \alpha_n$, and
\begin{equation*}
\begin{split}
x_{\theta}^-(1)&=[\,x_n^-(0),\,
x_{n-1}^-(0),\,\cdots,\,x_2^-(0),\,x_1^-(1) \,]_{(s,\,\cdots,\,s)},\\
x_{\theta}^+(-1)&
=[\,x_1^+(-1),\,x_2^+(0),\,\cdots
,\,x_{n}^+(0)\,]_{\langle r,\cdots, r\rangle}.
\end{split}
\end{equation*}

For the case of $D_n^{(1)}$, we fix the maximal root $$\theta=\alpha_1+\alpha_2+\cdots+ \alpha_{n-2}+\alpha_n+\alpha_{n-1}+\cdots+\alpha_2,$$ and
\begin{equation*}
\begin{split}
x_{\theta}^-(1)&
=[\,x_2^-(0),\,\cdots,\,x_n^-(0),\,
x_{n-2}^-(0),\,\cdots,\,x_2^-(0),\,x_1^-(1) \,]_{(s,\,\cdots,\,s,\,
r^{-1},\,\cdots,\,r^{-1})},\\
x_{\theta}^+(-1)&=[\,x_1^+(-1),\,x_2^+(0),\,\cdots
,\,x_{n-2}^+(0),\,
x_{n}^+(0),\cdots, \,x_{2}^+(0)\,]_{\langle r,\cdots, r, s^{-1},\cdots, s^{-1}\rangle}.
\end{split}
\end{equation*}
For the case of $E_6^{(1)}$, we fix the maximal root $$\theta=\alpha_1+\alpha_3+\cdots+ \alpha_{6}+\alpha_2+\alpha_{4}+\alpha_3+\alpha_5+\alpha_4+\alpha_2,$$ and
\begin{equation*}
\begin{split}
x_{\theta}^-(1)&=x_{\alpha_{13456243542}}^-(1)=
[\,x_2^-(0),\,x_{\alpha_{1345624354}}^-(1)\,]_{r^{-2}s^{-1}}=\cdots\\
&=[\,x_2^-(0),\,x_4^-(0),\,x_5^-(0),\,x_3^-(0),\,x_4^-(0),\,x_2^-(0)
,\,x_6^-(0),\cdots,\,\\
&\hskip2cm x_3^-(0),\,x_1^-(1) \,]_{(s,\,\cdots,\,s,\,
r^{-1},\,s,\,r^{-1},\,s,\,s,\,r^{-2}s^{-1})}\,,\\
x_{\theta}^+(-1)&=x_{\alpha_{13456243542}}^+(-1)=
[\,x_{\alpha_{1345624354}}^+(-1),\,x_2^+(0)\,]_{r^{-1}s^{-2}}\\
&=[\,x_1^+(-1),\,x_3^+(0),\,\cdots
,\,x_{6}^+(0),\,
x_{2}^+(0),\,x_{4}^+(0),\,\\
&\hskip1.2cm x_3^+(0),\, x_5^+(0),\,x_4^+(0)
,\,x_2^+(0)\,]_{\langle r,\,\cdots,\,r,\,
s^{-1},\,r,\,s^{-1},\,r,\,r,\,r^{-1}s^{-2}\rangle}\,.
\end{split}
\end{equation*}
\end{remark}

\subsection{Two-parameter Drinfel'd Isomorphism Theorem}  We state the main theorem as follows.
\begin{theo} Given a simple Lie algebra $\frak{g}$ of simply-laced type,  let
$\theta=\alpha_{i_1}+\cdots+\alpha_{{i_{h-1}}}$ be the maximal root with respect to a chosen prime root system $\Pi$.
 Then there exists an algebra isomorphism
$\Psi: U_{r,s}(\widehat{\frak{g}}) \longrightarrow {\mathcal
U}_{r,s}(\widehat{\frak{g}})$ defined by: for $i\in I,$
\begin{eqnarray*}
\omega_i&\longmapsto& \om_i\\
\omega'_i&\longmapsto& \om'_i \\
\omega_0&\longmapsto& \gamma'^{-1}\, \om_{\theta}^{-1}\\
\omega'_0&\longmapsto& \gamma^{-1}\, \om_{\theta}'^{-1}\\
\gamma^{\pm\frac{1}2}&\longmapsto& \gamma^{\pm\frac{1}2}\\
\gamma'^{\,\pm\frac{1}2}&\longmapsto& \gamma'^{\,\pm\frac{1}2}\\
D^{\pm1}&\longmapsto& D^{\pm1}\\
D'^{\,\pm1}&\longmapsto& D'^{\,\pm1}\\
e_i&\longmapsto& x_i^+(0)\\
f_i&\longmapsto& x_i^-(0)\\
e_0&\longmapsto& x^-_{\theta}(1)\cdot(\gamma'^{-1}\,\om_{\theta}^{-1})\\
f_0 &\longmapsto& a\tau\Bigl(x^-_{\theta}(1)
      \,\cdot(\gamma'^{-1}\,\om_{\theta}^{-1})\Bigr)
      =a(\gamma^{-1}\,\om_{\theta}'^{-1})\cdot
      x_{\theta}^+(-1)
\end{eqnarray*}
where $\om_{\theta}=\om_{i_1}\,\cdots\,
\om_{i_{h-1}},\,\om_{\theta}'=\om'_{i_1}\,\cdots\, \om'_{i_{h-1}}$, and
$$a=\left\{\begin{array}{cl} 1,
&  \textit{ for type \ $A^{(1)}$;}\\
(rs)^{n-2},
&  \textit{ for type \ $D^{(1)}$;}\\
(rs)^{4},&  \textit{ for type \ $E_6^{(1)}$.}\\
\end{array}\right.$$
\end{theo}

\medskip

\begin{remark}
Let $E_i,\,F_i$ ($i\in I_0$) and $\om_0$, $\om_0'$
denote the images of $e_i,\,f_i$ ($i\in I_0$) and $\om_0$, $\om_0'$
in the algebra ${\mathcal U}_{r,s}(\widehat{\frak{g}})$,
respectively. Denote by $\mathcal U_{r,s}'(\widehat{\frak{g}})$ the subalgebra
of $\mathcal{U}_{r,s}(\widehat{\frak{g}})$ generated by
$E_i,\,F_i,\,\om_i^{\pm1}$, $\om_i'^{\pm1}$ ($i\in I_0$),
$\gamma^{\pm\frac{1}2},\, \gamma'^{\,\pm\frac{1}2}$, $D^{\pm1}$,
$D'^{\pm1}$, that is,
$$
{\mathcal U\,'}_{r,s}(\widehat{\frak {g}}):=\left.\left\langle\,
E_i,\, F_i,\, \om_i^{\pm1},\, {\om}_i'^{\,\pm1}\;,\,
\gamma^{\pm\frac{1}2},\, \gamma'^{\,\pm\frac{1}2},\, D^{\pm1},\,
D'^{\,\pm1}\; \right| \;i\in I_0\;\right\rangle.
$$

Thereby, to prove the Drinfeld isomorphism theorem (Theorem 3.9) is
equivalent to prove the following three Theorems:

{\bf Theorem $\mathcal{A}$.}   {\it \quad
$\Psi:\,U_{r,s}(\hat{\frak {g}}) \longrightarrow \mathcal
U_{r,s}'(\hat{\frak {g}})$ is an epimorphism.}

{\bf Theorem $\mathcal{B}$.} {\it \quad  $\mathcal
U_{r,s}'(\hat{\frak {g}})={\mathcal U}_{r,s}(\hat{\frak {g}})$.}

{\bf Theorem $\mathcal{C}$.} {\it \,There exists a surjective $\Phi:\,\mathcal U_{r,s}'(\hat{\frak {g}}) \longrightarrow U_{r,s}(\hat{\frak {g}})$ such that $\Psi\Phi=\Phi\Psi=1$.}
\end{remark}

\section{Proof of Drinfeld
Isomorphism Theorem}
\medskip
For completeness, we check {\bf Theorem $\mathcal{A}$} for types $\mathrm{D}_n^{(1)}$ and $\mathrm{E}_6^{(1)}$ (similarly for types $\mathrm{E}_7^{(1)}$ and $\mathrm{E}_8^{(1)}$),
since we have proved it for type $\mathrm{A}_n^{(1)}$ in \cite{HRZ}.

\subsection{Proof of Theorem $\mathcal{A}$ for $U_{r,s}(\mathrm{D}_n^{(1)})$}

We have to check relations $(X1)$---$(X5)$ in Definition 2.2 for type $\mathrm{D}_n^{(1)}$.
It is easy to verify relations $(X1)$---$(X3)$, which are similar to the case of
$\mathrm{A}_{n-1}^{(1)}$ (\cite{HRZ}).

\medskip
To check relation $(X4)$, we first consider that when
$i\neq 0$,
$$[\,E_0,\,F_i\,]
=[\,x_{\theta}^-(1)
\,(\gamma'^{-1}\,\om_{\theta}^{-1}),\, x_i^-{(0)}\,]=-
[\,x_i^-{(0)},\,x_{\theta}^-(1) \,]_{\langle
\omega'_i,~ \omega_0 \rangle^{-1}} (\gamma'^{-1}\om_{\theta}^{-1}).
$$

Thus, relation $(X4)$ follows from immediately from the
following Lemma 4.1 in the case of $i\neq 0$.

\begin{lemm}  { \label{d:17} Let $i\in \{1,\cdots, n\}$, then we have
$$[\,x_i^-{(0)},\, x_{\theta}^-(1)\,]
_{\langle \omega'_i,~ \omega_0 \rangle^{-1}}=0.$$}
\end{lemm}

To show Lemma 4.1, the following Lemmas will play a crucial
role which will be proved in the appendix.

For our purpose, we need some notations: For $1\leqslant i<j \leqslant n-1$,
 $$
 x_{\alpha_{1,i}}^-(1)=[\,x_{i}^-(0),\,\cdots,\,x_{2}^-(0),\,x_{1}^-(1) \,]_{(s,\,\cdots,\,s)},
 $$
 $$
 x_{\beta_{i,j}}^-(1)=[\,x_i^-(0),\,\cdots,\,x_n^-(0),\,
x_{n-2}^-(0),\,\cdots,\,x_{j+1}^-(0),\,x_j^-(1) \,]_{(s,\,\cdots,\,s,\,
r^{-1},\,\cdots,\,r^{-1})}.
$$
Consequently, we get $x_{\theta}^-(1)=x_{\beta_{1,2}}^-(1)$.

\begin{lemm}  {\it \label{d:18} $[\,x_{i-1}^-(0),\,x_{\beta_{i-1,i+1}}^-(1)\,]_{s^{-1}}=0$, for $1< i < n$.}
\end{lemm}

\begin{lemm}  {\it \label{d:19}
$[\,x_{i}^-(0),\,x_{\beta_{i,i+1}}^-(1)\,]_{(rs)^{-1}}=0$, for $1\leqslant i \leqslant n-1$.}
\end{lemm}

\begin{lemm}  {\it \label{d:20}
 $[\,x_2^-{(0)},\, x_{\beta_{1,4}}^-(1)\,]=0$.}
\end{lemm}

\begin{lemm} {\it \label{d:add}
 $[\,x_i^-{(0)},\, x_{\beta_{1,i+2}}^-(1)\,]=0$, for $3\leqslant i \leqslant n-2$.}
\end{lemm}

The following Lemmas can be verified directly.

\begin{lemm}  {\label{d:21}
 $[\,x_i^-{(0)},\, x_{\beta_{1,i}}^-(1)\,]_{s^{-1}}=0$, for $3\leqslant i \leqslant n-1$.}
\end{lemm}

\begin{lemm}  {\label{d:22}
 $[\,x_n^-{(0)},\, x_n^-(0),\,x_{\alpha_{1,n}}^-(1)\,]_{(rs^2,\,r^2s)}=0$.}
\end{lemm}

Now using the above Lemmas, we turn to prove Lemma 4.1.
\medskip

{\noindent\bf Proof of Lemma 4.1.}  (I) When $i=1$,
$\langle\om_1',\om_0\rangle=rs$ and  $\la
\om_1',\om_\theta\ra=(rs)^{-1}$. It follows from
Lemma \ref{d:19} for the case of $i=1$.
\smallskip

(II) \ When $i=2$, $\langle\om_{2}',\om_0\rangle=s$, that is, $\la
\om_{2}',\om_\theta\ra=s^{-1}$. Let us first consider
\begin{equation*}
\begin{split}
x_{\theta}^-(1)&=x_{\beta_{1,2}}^-(1)
\qquad\qquad\qquad\qquad\qquad\qquad\ \;\, {\hbox{(by definition)}}\\
&=[\,x_{2}^-(0),\,
x_3^-(0),\,x_{\beta_{1,4}}^-(1)\,]_{(r^{-1},r^{-1})}
\qquad {\hbox{(by (\ref{b:1}))}}\\
&=[\,[\,x_{2}^-(0),\, x_{3}^-(0)\,]_{r^{-1}},\,x_{\beta_{1,4}}^-(1)\,]_{r^{-1}}\\
&\quad+r^{-1}[\,x_{3}^-(0),\,
\underbrace{[\,x_2^-(0),\,x_{\beta_{1,4}}^-(1)\,]}\,]
\qquad\, {\hbox{(=0 by Lemma \ref{d:20})}}\\
&=[\,[\,x_{2}^-(0),\,
x_{3}^-(0)\,]_{r^{-1}},\,x_{\beta_{1,4}}^-(1)\,]_{r^{-1}}.
\end{split}
\end{equation*}
Our previous result leads us to get
\begin{equation*}
\begin{split}
[\,x_2^-&{(0)},\, x_{\theta}^-(1)\,]_{s^{-1}}
\qquad\qquad\qquad\hskip4cm \;\;{\hbox{(by definition)}}\\
&=[\,x_2^-(0),\,[\,x_{2}^-(0),\,
x_{3}^-(0)\,]_{r^{-1}},\,x_{\beta_{1,4}}^-(1)\,]_{(r^{-1},\,s^{-1})}
\qquad\quad\; {\hbox{(by (\ref{b:1}))}}\\
&=[\,\underbrace{[\,x_2^-(0),\,x_{2}^-(0),\,x_3^-(0)\,]_{(r^{-1},\,s^{-1})}}
,\,x_{\beta_{1,4}}^-(1)\,]_{r^{-1}}\;\qquad\quad {\hbox{(=0 by (D9$_2$))}}\\
&\quad+s^{-1}[\,[\,x_2^-(0),\,x_3^-(0)\,]_{r^{-1}},\,
\underbrace{[\,x_{2}^-(0),\,x_{\beta_{1,4}}^-(1)\,]}
\,]_{r^{-1}s}\qquad {\hbox{(=0 by Lemma \ref{d:20})}}\\
&=0.
\end{split}
\end{equation*}

(III) When $3\leqslant i \leqslant n-1$,
$\langle\om_{i}',\om_0\rangle=1$, that is to say, $\la
\om_{i}',\om_\theta\ra=1$. We may use (\ref{b:1}) and (D9$_1$) to
show that
\begin{equation*}
\begin{split}
[\,x_i^-&{(0)},\, x_{\theta}^-(1)\,]\qquad\hskip4cm {\hbox{(by definition)}}\\
&=[\,x_i^-(0),\,[\,x_{2}^-(0),\,\cdots, x_{i-2}^-(0),\,
x_{\beta_{1,i-1}}^-(1)\,]_{(r^{-1},\cdots,r^{-1})}\,]\\
&=[\,x_{2}^-(0),\,\cdots, x_{i-2}^-(0),\,
\underbrace{[\,x_i^-(0),\,x_{\beta_{1,i-1}}^-(1)\,]}
\,]_{(r^{-1},\cdots,r^{-1})}.
\end{split}
\end{equation*}
For this purpose, it suffices to check that
$[\,x_i^-(0),\,x_{\beta_{1,i-1}}^-(1)\,]=0$. It is now
straightforward to verify that
\begin{equation*}
\begin{split}
[\,x_i^-&(0),\,x_{\beta_{1,i-1}}^-(1)\,]_{r^{-1}s}
\qquad\qquad\hskip4.48cm {\hbox{(by definition)}}\\
&=[\,x_i^-(0),\,[\,x_{i-1}^-(0),\,x_i^-(0),\,x_{\beta_{1,i+1}}^-(1)
\,]_{(r^{-1},\,r^{-1})}\,]_{r^{-1}s}
\qquad\,{\hbox{(by (\ref{b:1}))}}\\
&=[\,x_i^-(0),\,[\,x_{i-1}^-(0),\,x_i^-(0)\,]_{r^{-1}}
,\,x_{\beta_{1,i+1}}^-(1)\,]_{(r^{-1},\,r^{-1}s)}
\qquad\,{\hbox{(by (\ref{b:1}))}}\\
&\quad+r^{-1}[\,x_i^-(0),\,x_i^-(0),\,
\underbrace{[\,x_{i-1}^-(0),\,x_{\beta_{1,i+1}}^-(1)\,]}\,]_{(1,\,r^{-1}s)}
\qquad{\hbox{(=0 by Lemma \ref{d:add})}}\\
&=[\,\underbrace{[\,x_i^-(0),\,[\,x_{i-1}^-(0),\,x_i^-(0)\,]_{r^{-1}}
\,]_{s}}
,\,x_{\beta_{1,i+1}}^-(1)\,]_{r^{-2}}\qquad\quad\ \, {\hbox{(=0 by (D9$_2$))}}\\
&\quad+s[\,[\,x_{i-1}^-(0),\,x_i^-(0)\,]_{r^{-1}}
,\,[\,x_{i}^-(0),\,x_{\beta_{1,i+1}}^-(1)\,]_{r^{-1}}
\,]_{(rs)^{-1}}\qquad {\hbox{(by (\ref{b:2}))}}\\
&=s[\,x_{i-1}^-(0),\,
\underbrace{[\,x_i^-(0),\,x_{\beta_{1,i}}^-(1)\,]_{s^{-1}}}
\,]_{r^{-2}}\qquad\qquad\qquad\qquad\; {\hbox{(=0 by Lemma \ref{d:21})}}\\
&\quad+[\,[\,x_{i-1}^-(0),\,x_{\beta_{1,i}}^-(1)\,]_{r^{-1}}
,\,x_{i}^-(0)\,]_{r^{-1}s}\qquad\qquad\qquad\quad\ \, {\hbox{(by definition)}}\\
&=[\,x_{\beta_{1,i-1}}^-(1),\,x_{i}^-(0)\,]_{r^{-1}s},
\end{split}
\end{equation*}
which implies that $(1+r^{-1}s)[\,x_i^-(0),\,
x_{\beta_{1,i-1}}^-(1)\,]=0$. That is to say, $r\ne -s$,
$[\,x_i^-(0),\,x_{\beta_{1,i-1}}^-(1)\,]=0$.
\smallskip

(IV) When $i=n$, $\langle\om_{n}',\om_0\rangle=(rs)^{-2}$, that
is to say $\la \om_{n}',\om_\theta\ra=(rs)^2$. By using (\ref{b:1})
and (D9$_1$), we see that
\begin{equation*}
\begin{split}
[\,x_n^-&{(0)},\, x_{\theta}^-(1)\,]_{(rs)^2} \hskip5cm{\hbox{(by definition)}}\\
&=[\,x_n^-(0),\,[\,x_{2}^-(0),\,\cdots, x_{n-3}^-(0),\,
x_{\beta_{1,n-2}}^-(1)\,]_{(r^{-1},\cdots,r^{-1})}\,]_{(rs)^2}\\
&=[\,x_{2}^-(0),\,\cdots, x_{n-3}^-(0),\,
\underbrace{[\,x_n^-(0),\,x_{\beta_{1,n-2}}^-(1)\,]_{(rs)^2}}
\,]_{(r^{-1},\cdots,r^{-1})}.
\end{split}
\end{equation*}
Hence, if we have verified that
$[\,x_n^-(0),\,x_{\beta_{1,n-2}}^-(1)\,]_{(rs)^2}=0$, then we would get
the desired conclusion. We actually have
\begin{equation*}
\begin{split}
&x_{\beta_{1,n-2}}^-(1)=[\,x_{n-2}^-{(0)},\,x_{n-1}^-(0),\,x_n^-(0),\,
 x_{\alpha_{1,n-1}}^-(1)\,]_{(s,\,r^{-1},\,r^{-1})}\qquad\qquad\, {\hbox{(by (\ref{b:1}))}}\\
&\quad=[\,x_{n-2}^-(0),\,\underbrace{[\,x_{n-1}^-(0),\,
x_{n}^-(0)\,]_{(rs)^{-1}}},\,
x_{\alpha_{1,n-1}}^-(1)\,]_{(s^{2},r^{-1})}
\qquad\qquad {\hbox{(=0 by (D9$_1$))}}\\
&\qquad+(rs)^{-1}[\,x_{n-2}^-(0),\, x_{n}^-(0),\,
\underbrace{[\,x_{n-1}^-(0),\,x_{\alpha_{1,n-1}}^-(1)\,]_{s}}
\,]_{(rs^{2},\cdots,r^{-1})}\qquad\; {\hbox{(by (\ref{b:1}))}}\\
&\quad=(rs)^{-1}[\,[\,x_{n-2}^-(0),\, x_{n}^-(0)\,]_{r^{-1}},\,
x_{\alpha_{1,n}}^-(1)\,]_{rs^{2}}\\
&\qquad+r^{-2}s^{-1}[\,x_{n}^-(0),\,\underbrace{[\,
x_{n-2}^-(0),\,x_{\alpha_{1,n}}^-(1)\,]}
\,]_{(rs)^{2}}\qquad\hskip1cm {\hbox{(=0 by (\ref{b:1}) \& (D9$_1$))}}\\
&\quad=(rs)^{-1}[\,[\,x_{n-2}^-(0),\, x_{n}^-(0)\,]_{r^{-1}},\,
x_{\alpha_{1,n}}^-(1)\,]_{rs^{2}}.
\end{split}
\end{equation*}
It can be proved from the above steps that
\begin{equation*}
x_{\beta_{1,n-2}}^-(1)=(rs)^{-1}[\,x_{n-2}^-(0),\, x_{n}^-(0),\,
x_{\alpha_{1,n}}^-(1)\,]_{(rs^{2},\cdots,r^{-1})}.
\end{equation*}

Hence, we may easily check that
\begin{equation*}
\begin{split}
[\,x_n^-&(0),\,x_{\beta_{1,n-2}}^-(1)\,]_{rs^3}\\
&=(rs)^{-1}[\,x_n^-(0),\,[\,x_{n-2}^-(0),\,x_n^-(0)\,]_{r^{-1}}
,\,x_{\alpha_{1,n}}^-(1)
\,]_{(rs^{2},\,rs^{3})}\qquad\ \;  {\hbox{(by (\ref{b:1}))}}\\
&=(rs)^{-1}[\,\underbrace{[\,x_n^-(0),\,x_{n-2}^-(0),\,x_n^-(0)\,]_{(r^{-1},\,s)}}
,\,x_{\alpha_{1,n}}^-(1)\,]_{r^{2}s^4}\qquad\quad {\hbox{(=0 by (D9$_2$))}}\\
&\quad+r^{-1}[\,[\,x_{n-2}^-(0),\,x_n^-(0)\,]_{r^{-1}},\,
\underbrace{[\,x_{n}^-(0),\,x_{\alpha_{1\,n}}^-(1)\,]_{rs^2}}\,]_{rs}\qquad\quad\ \, {\hbox{(by (\ref{b:2}))}}\\
&=r^{-1}[\,x_{n-2}^-(0),\,\underbrace{[\,x_n^-(0),\,x_{n}^-(0),\,x_{\alpha_{1,n}}^-(1)\,]_{(rs^{2},\,r^2s)}
}\,]_{r^{-2}}\qquad {\hbox{(=0 by Lemma \ref{d:22})}}\\
&\quad+rs[\,\underbrace{[\,x_{n-2}^-(0),\,x_n^-(0),\,x_{\alpha_{1,n}}^-(1)\,]_{(rs^{2},\,r^{-1})}}
,\,x_{n}^-(0)\,]_{r^{-3}s^{-1}}\qquad {\hbox{(by definition)}}\\
&=(rs)^2[\,x_{\beta_{1,n-2}}^-(1),\,x_n^-(0)\,]_{r^{-3}s^{-1}}.
\end{split}
\end{equation*}
So, we have $(1+r^{-1}s)\,[\,x_n^-(0),\,
x_{\beta_{1,n-2}}^-(1)\,]_{(rs)^2}=0$. Since $r\ne -s$, then we
get $[\,x_n^-(0),\, x_{\beta_{1,n-2}}^-(1)\,]_{(rs)^2}=0$. This completes the proof of Lemma \ref{d:17}. \hfill\qed
\smallskip

We are now ready to prove relation $(X4)$ for $i=j=0$, that is,
\begin{prop} {\it
$[\,E_0, F_0\,]=
\frac{\gamma'^{-1}\om_{\theta}^{-1}-\gamma^{-1}{\om'_\theta}^{-1}}{r-s}$.}
\end{prop}

\begin{proof}  Note that the construction of $E_0$ and $F_0$, we check the statement step by step.
First, using $(\textrm{D1})$ and $(\textrm{D5})$, we have
\begin{equation*}
\begin{split}
\bigl[\,E_0, F_0\,\bigr]&=(rs)^{n-2}\bigl[\,x^-_{\theta}(1)\,
\gamma'^{-1}{\om_\theta}^{-1},\,\gamma^{-1}{\om'_\theta}^{-1}
x^+_{\theta}(-1)\,\bigr]\\
&=(rs)^{n-2}\bigl[\,x^-_{\theta}(1),\,x^+_{\theta}(-1)\,\bigr]\,
\cdot(\gamma^{-1}\gamma'^{-1}{\om_\theta}^{-1}{\om'_\theta}^{-1}).
\end{split}
\end{equation*}
We may now use the result of the case of $A_{n-1}^{(1)}$
\begin{equation*}
[\,x_{\alpha_{1,n-1}}^-(1),\,x_{\alpha_{1,n-1}}^+(-1)\,]=
\frac{\gamma\om_{\alpha_{1,n-1}}'-\gamma'\om_{\alpha_{1,n-1}}}{r-s}.
\end{equation*}
Applying the above result, it is now straightforward to verify that
\begin{equation*}
\begin{split}
&[\,x_{\beta_{1,n}}^-(1),\,x_{\beta_{1,n}}^+(-1)\,]\qquad\hskip4.9cm{\hbox{(by definition)}}\\
&\quad=[\,[\,x_{n}^-(0),\,x_{\alpha_{1,n-1}}^-(1)\,]_{s}
,[\,x_{\alpha_{1,n-1}}^+(-1),x_{n}^+(0)\,]_{r}\,]\qquad\quad\,{\hbox{(by (\ref{b:1}))}}\\
&\quad=[\,[\,[\,x_{n}^-(0),\,x_{\alpha_{1,n-1}}^+(-1)\,],x_{\alpha_{1,n-1}}^-(1)\,]_{s}
,x_{n}^+(0)\,]_{r}\qquad\quad\, {\hbox{(=0 by (\ref{b:1}) \& (D8))}}\\
&\qquad+[\,[\,x_{n}^-(0),\,[\,x_{\alpha_{1,n-1}}^-(1),\,x_{\alpha_{1,n-1}}^+(-1)\,]\,]_{s}
,\,x_{n}^+(0)\,]_{r}\qquad {\hbox{(by (D5) \& (D8))}}\\
&\qquad+[\,x_{\alpha_{1,n-1}}^+(-1),\,[\,
[\,x_{n}^-(0),\,x_{n}^+(0)\,],\,x_{\alpha_{1,n-1}}^-(1)\,]_{s}\,]_{r}
\qquad{\hbox{(by (D8) \& (D5))}}\\
&\qquad+[\,x_{\alpha_{1,n-1}}^+(-1),\,[\,x_{n}^-(0),\,
[\,x_{\alpha_{1,n-1}}^-(1),\,x_{n}^+(0)\,]\,]_{s}
\,]_{r}\qquad{\hbox{(=0 by (\ref{b:1}) \& (D8))}}\\
&\quad=\gamma\om'_{\alpha_{1,n-1}}\cdot \frac{\om'_n-\om_n}{r-s}
+\frac{\gamma\om'_{\alpha_{1,n-1}}-\gamma'\om_{\alpha_{1,n-1}}}
{r-s}\om_n\\
&\quad=\frac{\gamma\om'_{\beta_{1,n}}-\gamma'\om_{\beta_{1,n}}}{r-s}.
\end{split}
\end{equation*}
We have then by repeating the above step
\begin{equation*}
\begin{split}
&[\,x_{\beta_{1,n-1}}^-(1),x_{\beta_{1,n-1}}^+(-1)\,]\qquad\hskip5cm{\hbox{(by definition)}}\\
&\quad=[\,[\,x_{n-1}^-(0),\,x_{\beta_{1,n}}^-(1)\,]_{r^{-1}}
,\,[\,x_{\beta_{1,n}}^+(-1),\,x_{n-1}^+(0)\,]_{s^{-1}}\,]
\qquad\quad {\hbox{(by (\ref{b:1}))}}\\
&\quad=[\,[\,[\,x_{n-1}^-(0),\,x_{\beta_{1,n}}^+(-1)\,],\,x_{\beta_{1,n}}^-(1)\,]_{r^{-1}}
,\,x_{n-1}^+(0)\,]_{s^{-1}}\qquad\quad {\hbox{(=0 by (\ref{b:1}) \& (D8))}}\\
&\qquad+[\,[\,x_{n-1}^-(0),\,[\,x_{\beta_{1,n}}^-(1),\,x_{\beta_{1,n}}^+(-1)\,]\,]_{r^{-1}}
,\,x_{n-1}^+(0)\,]_{s^{-1}}\qquad\,{\hbox{(by (\ref{b:1}) \& (D8))}}\\
&\qquad+[\,x_{\beta_{1,n}}^+(-1),\,[\,
[\,x_{n-1}^-(0),\,x_{n-1}^+(0)\,],\,x_{\beta_{1,n}}^-(1)\,]_{r^{-1}}
\,]_{s^{-1}}\qquad\,{\hbox{(by (\ref{b:1}) \& (D8))}}\\
&\qquad+[\,x_{\beta_{1,n}}^+(-1),\,[\,x_{n-1}^-(0),\,
[\,x_{\beta_{1,n}}^-(1),\,x_{n-1}^+(0)\,]\,]_{r^{-1}}
\,]_{s^{-1}}\qquad\, {\hbox{(=0 by (\ref{b:1}) \& (D8))}}\\
&\quad=(rs)^{-1}\gamma\om'_{\beta_{1,n}}\cdot
\frac{\om'_{n-1}-\om_{n-1}}{r-s}
+(rs)^{-1}\frac{\gamma\om'_{\beta_{1,n}}-\gamma'\om_{\beta_{1,n}}}
{r-s}\om_{n-1}\\
&\quad=(rs)^{-1}\frac{\gamma\om'_{\beta_{1,n-1}}-\gamma'\om_{\beta_{1,n-1}}}{r-s}.
\end{split}
\end{equation*}
Furthermore, it follows from the above results
\begin{equation*}
\begin{split}
&[\,x_{\beta_{1,n-2}}^-(1),\,x_{\beta_{1,n-2}}^+(-1)\,]\qquad\ \;\hskip5cm{\hbox{(by definition)}}\\
&\quad=[\,[\,x_{n-2}^-(0),\,x_{\beta_{1,n-1}}^-(1)\,]_{r^{-1}}
,\,[\,x_{\beta_{1,n-1}}^+(-1),\,x_{n-2}^+(0)\,]_{s^{-1}}\,]\qquad{\hbox{(by (\ref{b:1}))}}\\
&\quad=[\,[\,[\,x_{n-2}^-(0),\,x_{\beta_{1,n-1}}^+(-1)\,],\,x_{\beta_{1,n-1}}^-(1)\,]_{r^{-1}}
,\,x_{n-2}^+(0)\,]_{s^{-1}}\qquad{\hbox{(=0 by (D8) \& (D5))}}\\
&\qquad+[\,[\,x_{n-2}^-(0),\,[\,x_{\beta_{1,n-1}}^-(1),\,x_{\beta_{1,n-1}}^+(-1)\,]
\,]_{r^{-1}}
,\,x_{n-2}^+(0)\,]_{s^{-1}}\quad{\hbox{(by (D5) \& (D8))}}\\
&\qquad+[\,x_{\beta_{1,n-1}}^+(-1),\,[\,
[\,x_{n-2}^-(0),\,x_{n-2}^+(0)\,],\,x_{\beta_{1,n-1}}^-(1)\,]_{r^{-1}}
\,]_{s^{-1}}\quad{\hbox{(by (D8) \& (D5))}}\\
&\qquad+[\,x_{\beta_{1,n-1}}^+(-1),\,[\,x_{n-2}^-(0),\,
[\,x_{\beta_{1,n-1}}^-(1),\,x_{n-2}^+(0)\,]\,]_{r^{-1}}
\,]_{s^{-1}}\quad{\hbox{(=0 by (D5) \& (D8))}}\\
&\quad=(rs)^{-2}\gamma\om'_{\beta_{1,n-1}}\cdot
\frac{\om'_{n-2}-\om_{n-2}} {r-s}
+(rs)^{-2}\frac{\gamma\om'_{\beta_{1,n-1}}-\gamma'\om_{\beta_{1,n-1}}}
{r-s}\om_{n-1}\\
&\quad=(rs)^{-2}\frac{\gamma\om'_{\beta_{1,n-2}}-\gamma'\om_{\beta_{1,n-2}}}
{r-s}.
\end{split}
\end{equation*}
By the same way, we get at last
\begin{equation*}
\begin{split}
[\,x_{\beta_{1,2}}^-(1),\,x_{\beta_{1,2}}^+(-1)\,]
=(rs)^{2-n}\frac{\gamma\om'_{\beta_{1,2}}-\gamma'\om_{\beta_{1,2}}}
{r-s}.
\end{split}
\end{equation*}
As a consequence, we obtain the required result
$$
\bigl[\,E_0,\, F_0\,\bigr]=\frac{\gamma'^{-1}\om_{\theta}^{-1}-\gamma^{-1}{\om'_\theta}^{-1}}{r-s}.
$$
\end{proof}

For the rest of this subsection, we will focus on checking the Serre
relations of $U_{r,s}(\mathrm{D}_n^{(1)})$.

\begin{lemm} {\it $(1)$ \ $E_nE_0=(rs)^2\,E_0E_n$,

$(2)$ \ $E_0E_2^2-(r+s)E_2E_0E_2+rsE_2^2E_0=0$,

$(3)$ \ $E_0^2E_2-(r+s)E_0E_2E_0+rsE_2E_0^2=0$,

$(4)$ \ $F_0F_n=(rs)^2\,F_nF_0$,

$(5)$ \ $F_{1}F_0^{2}-(r+s)\,F_{0}F_1F_{0}+rs\,F_0^2F_{1}=0$,

$(6)$ \ $F_2^2F_{0}-(r+s)F_2F_{0}F_2+rsF_0F_{2}^2=0$.}
\end{lemm}

\begin{proof} Relations $(4)$---$(6)$ follow from the action of $\tau$ on relations $(1)$---$(3)$.
To be precise, let us just consider the second and third relation.

(1) For the second equality, it is easy to see that
\begin{equation*}
\begin{split}
\bigl[&\,E_2, x^-_{\theta}(1)\,\bigr]\hskip4.6cm \hbox{(by (\ref{b:1}))}\\
&=[\,[\,x_2^{+}(0),\,x_2^-(0)\,],\,x_{\beta_{1,3}}^-(1)\,]_{r^{-1}}
\qquad\qquad(\hbox{by (D8) \& (D5)})\\
&\ +[\,x_2^-(0),\cdots,x_n^-(0),x_{n-2}^-(0),\cdots,x_3^-(0),
[\,x_2^+(0),x_2^-(0)\,],x_1^-(1)\,]_{(s,\cdots,s,r^{-1}
\cdots,r^{-1})}\\
&\hskip5.65cm\qquad(\hbox{=0 by $(\textrm{D8})$, $(\textrm{D5})$ \& $(\textrm{D9}_1)$})\\
&=(rs)^{-1}x_{\beta_{1,3}}^-(1)\,\om_2.
\end{split}
\end{equation*}
By the above observation, it can be proved in a straightforward
manner that
\begin{equation*}
\begin{split}
E_0&E_2^2-(r+s)E_2E_0E_2+rs\,E_2^2E_0\\
&=rs\Big(E_2^2
x^-_{\theta}(1)-(1+r^{-1}s)\,E_2x^-_{\theta}(1)E_2+r^{-1}s\,x^-_{\theta}(1)E_2^2\Big)
(\gamma'^{-1}\om_{\theta}^{-1})\\
&=rs\,\bigl[\,E_2,\, \underbrace{\bigl[\,E_2,\,x^-_{\theta}(1)\,
\bigr]}\,\bigr]_{r^{-1}s}\,(\gamma'^{-1}\om_{\theta}^{-1})\\
&=\underbrace{\bigl[\,E_2,\,x^-_{\beta_{1,3}}(1)\,
\bigr]}\om_2\,(\gamma'^{-1}\om_{\theta}^{-1})\qquad\qquad
(\hbox{=0 by (\ref{b:1}) \& $(\textrm{D8})$})\\
&=0.
\end{split}
\end{equation*}

(2) Note that the formula of $\,\bigl[\,E_2,
x^-_{\theta}(1)\,\bigr]\,$ obtained in (1), we actually have
\begin{equation*}
\begin{split}
&E_0^2E_2-(r+s)E_0E_2E_0+(rs)\,E_2E_0^2 \\
&\quad=(rs)\,\bigl[\,x^-_{\theta}(1), \underbrace{x^-_{\theta}(1),
E_2}\,\bigr]_{(1,\,rs^{-1})}(\gamma'^{-2} \om_{\theta}^{-2})\\
&\quad=-\,\bigl[\,x^-_{\theta}(1),\,
x_{\beta_{1,3}}^-(1)\om_2\,\bigr]_{rs^{-1}}\,(\gamma'^{-2}
\om_{\theta}^{-2})\\
&\quad=-\underbrace{\bigl[\,x^-_{\theta}(1),\,
x_{\beta_{1,3}}^-(1)\,\bigr]_{s^{-1}}}\,\om_2(\gamma'^{-2}
\om_{\theta}^{-2})
\qquad(\textrm{=0 by Lemma \ref{d:23} below})\\
&\quad=0.
\end{split}
\end{equation*}
\end{proof}

\begin{lemm} {\it  \label{d:23} $\bigl[\,x_{\beta_{1,3}}^-(1),
x^-_{\beta_{1,2}}(1)\,\bigr]_s=0$, \textit{for} $r\ne -s$}.
\end{lemm}

Proof. See the appendix.

\subsection{Proof of Theorem $\mathcal{A}$ for $U_{r,s}(\mathrm{E}_6^{(1)})$}

As usual, we need to verify some critical relations of Theorem
$\mathcal{A}$.

Note that the highest root of simple
Lie algebra $\mathrm{E}_6$ is
$$\theta=\alpha_{13456243542}\doteq \alpha_1+\alpha_{3}+\cdots+\alpha_{6}+\alpha_2 +\alpha_{4}+
\alpha_{3}+\alpha_5+\alpha_4+\alpha_2.$$
The maximal quantum root vectors $x_{\theta}^-(1)$ and $x_{\theta}^+(-1)$ are defined as follows
\begin{equation*}
\begin{split}
x_{\theta}^-(1)&=x_{\alpha_{13456243542}}^-(1)=
[\,x_2^-(0),\,x_{\alpha_{1345624354}}^-(1)\,]_{r^{-2}s^{-1}}=\cdots\\
&=[\,x_2^-(0),\,x_4^-(0),\,x_5^-(0),\,x_3^-(0),\,x_4^-(0),\,x_2^-(0)
,\,x_6^-(0),\cdots,\\
&\hskip2cm ,x_3^-(0),\,x_1^-(1) \,]_{(s,\,\cdots,\,s,\,
r^{-1},\,s,\,r^{-1},\,s,\,s,\,r^{-2}s^{-1})}.\\
x_{\theta}^+(-1)&=x_{\alpha_{13456243542}}^+(-1)=
[\,x_{\alpha_{1345624354}}^+(-1),\,x_2^+(0)\,]_{r^{-1}s^{-2}}\\
&=[\,x_1^+(-1),\,x_3^+(0),\,\cdots
,\,x_{6}^+(0),\,
x_{2}^+(0),\,x_{4}^+(0),\,x_3^+(0),\,x_5^+(0),\\
&\hskip2cm ,x_4^+(0),\,x_2^+(0)\,]_{\langle r,\,\cdots,\,r,\,
s^{-1},\,r,\,s^{-1},\,r,\,r,\,r^{-1}s^{-2}\rangle}.
\end{split}
\end{equation*}

Similarly,  relation $(X4)$ holds due to Lemma
\ref{d:24} below in the case of $i\neq 0$.

\begin{lemm} {\label{d:24}
  $[\,x_i^-{(0)},\,x_{\theta}^-(1)
\,]_{\la \om_1',\om_\theta\ra}=0$, for $i=1, 2,\cdots,6$.}
\end{lemm}

To verify Lemma \ref{d:24}, the following three Lemmas, which we will check in the appendix,
will play a crucial role. To
be precise, let $x_{i_1\,i_2\cdots\,i_n}^{\pm}(k)=x_{\alpha_{i_1,i_2,\cdots,i_n}}^{\pm}(k)$. We may easily check that

\begin{lemm} {\label{d:25}
$
[\,x_2^-(0),\, x_{134562435}^-(1)\,]_{(rs)^{-1}}=0$.}
\end{lemm}

More generally, we have
\begin{lemm} {\it  \label{d:26}
\begin{gather*}[\,x_4^-(0),\, x_{1345624}^-(1)\,]_{r}=0, \\
[\,x_3^-(0),\, x_{1345624354}^-(1)\,]_{(rs)^{-1}}=0.
\end{gather*}}
\end{lemm}

\begin{lemm}
{ \label{d:27}  $[\,x_1^-(0),\,
x_{1345624354}^-(1)\,]_{(rs)^{-1}}=0.$}
\end{lemm}

\noindent
{\bf Proof of Lemma \ref{d:24}.}  We may now use the previous Lemmas to show that

(I) \ When $i=1$,
$\langle\om_1',\om_0\rangle=rs$ and  $\la
\om_1',\om_\theta\ra=(rs)^{-1}$,
\begin{equation*}
\begin{split}
[\,x_1^-&(0),\, x_{\theta}^-(1)\,]_{(rs)^{-1}}
\hskip5.55cm(\hbox{by definition})\\
&=[\,x_1^-(0),\,
x_2^-(0),\,x_{1345624354}^-(1)\,]_{(r^{-2}s^{-1},\,(rs)^{-1})}
\qquad\ \ (\hbox{by (\ref{b:1})})\\
&=[\,\underbrace{[\,x_1^-(0),\, x_2^-(0)\,]},\,
x_{1345624354}^-(1)\,]_{r^{-3}s^{-2}}\qquad\qquad\quad\, \hbox{(=0 by $(\textrm{D9}_1)$)}\\
&\quad+[\,x_2^-(0),\,\underbrace{ [\,x_1^-(0),\,
x_{1345624354}^-(1)\,]_{(rs)^{-1}}}\,]_{r^{-2}s^{-1}}\qquad \hbox{(=0 by Lemma \ref{d:27})}\\
&=0.
\end{split}
\end{equation*}

 (II) \ When $i=2$,
$\langle\om_2',\om_0\rangle=rs^{2}$ and  $\la
\om_2',\om_\theta\ra=r^{-1}s^{-2}$,
\begin{equation*}
\begin{split}
&[\,x_2^-(0),\, x_{\theta}^-(1)\,]_{r^{-1}s^{-2}}
\hskip6cm\ \ \;(\hbox{by definition})\\
&\; =[\,x_2^-(0),\,
x_2^-(0),\,x_4^-(0),\,x_{134562435}^-(1)\,]_{(s\,r^{-2}s^{-1},\,r^{-1}s^{-2})}
\qquad\ \ \; (\hbox{by (\ref{b:1})})\\
&\; =[\,x_2^-(0),\,[\,x_2^-(0),\, x_4^-(0)\,]_{r^{-1}},\,
x_{134562435}^-(1)\,]_{(r^{-1},\,r^{-1}s^{-2})}\qquad\ \; \hbox{(by (\ref{b:1}))}\\
&\ +r^{-1}[\,x_2^-(0),\, x_4^-(0),\,\underbrace{[\,x_2^-(0),
x_{134562435}^-(1)\,]_{(rs)^{-1}}}\,]_{(rs,\, r^{-1}s^{-2})}\; \hbox{(=0 by Lemma \ref{d:25})}\\
&\; =[\,\underbrace{[\,x_2^-(0),\,x_2^-(0),\,
x_4^-(0)\,]_{(r^{-1},\,s^{-1})}},\,
x_{134562435}^-(1)\,]_{r^{-2}s^{-1}}\qquad\ \ \hbox{(=0 by $(\textrm{D9}_2)$)}\\
&\; +s^{-1}[\,[\,x_2^-(0),\, x_4^-(0)\,],\,\underbrace{[\,x_2^-(0),
x_{134562435}^-(1)\,]_{(rs)^{-1}}}\,]_{ r^{-1}s}\qquad \hbox{(=0 by Lemma \ref{d:25})}\\
&\; =0.
\end{split}
\end{equation*}

 (III) \ When $i=3$,
$\langle\om_3',\om_0\rangle=rs$ and  $\la
\om_3',\om_\theta\ra=(rs)^{-1}$. We may easily check that
\begin{equation*}
\begin{split}
[\,x_3^-&(0),\, x_{\theta}^-(1)\,]_{(rs)^{-1}}
\hskip5.6cm(\hbox{by definition})\\
&=[\,x_3^-(0),\,
x_2^-(0),\,x_{1345624354}^-(1)\,]_{(r^{-2}s^{-1},\,(rs)^{-1})}
\qquad\ \;\;(\hbox{by (\ref{b:1})})\\
&=[\,\underbrace{[\,x_3^-(0),\, x_2^-(0)\,]},\,
x_{1345624354}^-(1)\,]_{r^{-3}s^{-2}}\qquad\qquad\quad\ \hbox{(=0 by $(\textrm{D9}_1)$)}\\
&\quad+[\,x_2^-(0),\, \underbrace{[\,x_3^-(0),\,
x_{1345624354}^-(1)\,]_{(rs)^{-1}}}\,]_{r^{-2}s^{-1}}
\qquad\, \hbox{(=0 by Lemma \ref{d:26})}\\
&=0.
\end{split}
\end{equation*}

(IV) \ When $i=4$,
$\langle\om_4',\om_0\rangle=(rs)^{-1}$ and  $\la
\om_4',\om_\theta\ra=rs$. We have
\begin{equation*}
\begin{split}
&[\,x_4^-(0),\, x_{\theta}^-(1)\,]_{s^{2}}
\hskip6.5cm(\hbox{by definition})\\
&\ =[\,x_4^-(0),\,
x_2^-(0),\,x_4^-(0),\,x_{134562435}^-(1)\,]_{(s\,r^{-2}s^{-1},\,s^{2})}
\qquad\quad\ \,(\hbox{by (\ref{b:1})})\\
&\ =[\,x_4^-(0),\,[\,x_2^-(0),\, x_4^-(0)\,]_{r^{-1}},\,
x_{134562435}^-(1)\,]_{(r^{-1},\,s^{2})}\qquad\quad (\hbox{by (\ref{b:1})})\\
&\quad+r^{-1}[\,x_2^-(0),\, x_4^-(0),\,\underbrace{[\,x_2^-(0),
x_{134562435}^-(1)\,]_{(rs)^{-1}}}\,]_{(rs,\, s^{2})}\ \,\hbox{(=0 by Lemma \ref{d:25})}\\
&\ =[\,\underbrace{[\,x_4^-(0),\,x_2^-(0),\,
x_4^-(0)\,]_{(r^{-1},\,s)}},\,
x_{134562435}^-(1)\,]_{r^{-1}s}\qquad\quad (\hbox{=0 by $(\textrm{D9}_2)$})\\
&\quad+s[\,[\,x_2^-(0),\, x_4^-(0)\,]_{r^{-1}},\,[\,x_4^-(0),
x_{134562435}^-(1)\,]_{s}\,]_{ (rs)^{-1}}\ (\hbox{by definition \& (\ref{b:2})})\\
&\ =s\,[\,x_2^-(0),\, \underbrace{[\,x_4^-(0),
x_{1345624354}^-(1)\,]_r}\,]_{r^{-3}s^{-1}}\qquad\qquad\qquad\;\;(\hbox{=0 by Lemma \ref{d:26}})\\
&\quad+rs\,[\,[\,x_2^-(0),\,x_{1345624354}^-(1)\,]_{r^{-2}s^{-1}}
,\,x_4^-(0)\,]_{r^{-2}}\qquad\qquad\;\;(\hbox{by definition})\\
&\ =-r^{-1}s[\,x_4^-(0),\, x_{\theta}^-(0)\,]_{r^{2}},
\end{split}
\end{equation*}
which implies that $(1+r^{-1}s)\,[\,x_4^-(0),\,
x_{\theta}^-(0)\,]_{rs}=0$. That is to say, when $r\ne -s$,
$$[\,x_4^-(0),\, x_{\theta}^-(0)\,]_{rs}=0.$$

(V) The proof of the case $i=5$ or $6$ is similar to that of the
case $i=1$ or $3$, which are left to the readers. \hfill\qed
\medskip

We would like to point out that the proof of relation $[\,E_0,
F_0\,]=\frac{\omega_0-\omega'_0}{r-s}$ is the same as that of the case
of $D_n^{(1)}$. We now proceed to show the Serre relations for $E_6^{(1)}$.

\begin{prop}  { We have the following Serre relations

$(1)$ \ $E_0E_2^2-(rs)(r+s)E_2E_0E_2+(r s)^{3}\,E_2^2E_0=0$,

$(2)$ \ $E_0^2E_2-(rs)(r+s)E_0E_2E_0+(rs)^{3}\,E_2E_0^2=0$,

$(3)$ \ $F_2^2F_0-(rs)(r+s)F_2F_0F_2+(r s)^{3}\,F_0F_2^2=0$,

$(4)$ \ $F_2F_0^2-(rs)(r+s)F_0F_2F_0+(rs)^{3}\,F_0^2F_2=0$, }
\end{prop}

\begin{proof}\,  Here we only give the proof of the first $(r,s)$-Serre relation,
and the others are left to the readers.

Let us first to prove that
\begin{equation*}
\begin{split}
[\,E_2&, x_{\theta}^-(1)\,]\qquad\hskip5.1cm(\hbox{by definition \& (\ref{b:1})})\\
&=[\,[\,x_2^+(0),\,x_{2}^-(0)\,],\,
x_{1345624354}^-(1)\,]_{r^{-2}s^{-1}}\qquad(\hbox{by (D8), (D5)})\\
&\quad+ [\,x_2^-(0),\,x_4^-(0),\, x_5^-(0),\,
x_3^-(0),\,x_4^-(0),\,[\,x_2^+(0),\,x_{2}^-(0)\,],\,\\
&\hskip1.1cm ,x_{13456}^-(1)\,]_{(r^{-1},\,s,\,r^{-1},\,s,\,s,\,r^{-2}s^{-1})} \qquad(\hbox{=0 by $(D8)$, $(D5)$ \& $(\textrm{D9}_1)$})\\
&=(rs)^{-2}x_{1345624354}^-(1)\,\om_2
\end{split}
\end{equation*}
By the above result, we get that
\begin{equation*}
\begin{split}
E_0&E_2^2-(rs)(r+s)E_2E_0E_2+(r s)^{2}\,E_2^2E_0 \\
&=(r s)^{-2}\Big(E_2^2x_{\theta}^-(0)-
(1+r^{-1}s)E_2x_{\theta}^-(1)E_2+(r^{-1}s)\,x_{\theta}^-(1)E_2^2\Big)
\gamma\, \om_{\theta}^{-1}\\
&= (rs)^{-2}\,[\,E_2, \,E_2,\,x_{\theta}^-(1)\,]
_{(1,\,r^{-1}s)}\,\gamma \,\om_{\theta}^{-1}\\
&=(rs)^{-2}\,[\,E_2,\,x_{\theta}^-(1)\,]\,\om_2\\
&=[\,x_4^-(0),\, x_5^-(0),\,
x_3^-(0),\,x_4^-(0),\,[\,x_2^+(0),\,x_{2}^-(0)\,],\,
x_{13456}^-(1)\,]_{(r^{-1},\,s,\,r^{-1},\,s,\,s)}\\
&\hskip5cm (\hbox{=0 by $(D8)$, $(D5)$ \& $(\textrm{D9}_1)$})\\
&=0.
\end{split}
\end{equation*}
\end{proof}

\subsection{Proof of Theorem $\mathcal{B}$} \ This is
similar to that of the case $A_{n-1}^{(1)}$ \cite{HRZ}. We shall show
that the algebra ${\mathcal U}_{r,s}(\widehat{\frak{g}})$ is generated by
$E_i,\, F_i,\,\om_i^{\pm1},\,{\om}_i'^{\,\pm1},\,
\gamma^{\pm\frac{1}2}$, $\gamma'^{\,\pm\frac{1}2}$ ($i\in I_0$).
More explicitly, any generators of the algebra ${\mathcal U}_{r,s}(\widehat{\frak{g}})$
are in the subalgebra $\mathcal U'_{r,s}(\widehat{\frak{g}})$.
To do so, we also need the following two Lemmas, which can be similarly checked
like in \cite{HRZ}.
\begin{lemm}
\ $(1)$
\begin{equation*}
\begin{split}
&x_1^-(1)\\
&\quad=[\,E_2,E_3,\cdots,E_{n-2},E_{n},\cdots,E_2,
E_0\,]_{(s^{-1},\cdots,s^{-1},r,r,\cdots,r)}
\gamma'\om_1\in {\mathcal U}'_{r,s}(D_n^{(1)}),
\end{split}
\end{equation*}
then for any $i\in I$, $x_i^-(1)\in {\mathcal U}'_{r,s}(D_n^{(1)})$.

\smallskip
$(2)$
\begin{equation*}
\begin{split}
&x_1^+(-1)\\
&\quad=\tau\Bigl(\bigl[\,E_2,E_3,\cdots,E_{n-2},E_{n},\cdots,E_2,
E_0\,]_{(s^{-1},\cdots,s^{-1},r,r,\cdots,r)}
\gamma'\om_1\Bigr)\\
&\quad=\gamma\om'_1\,[\,F_0,F_2,\cdots,F_{n-2},
F_{n},\cdots,F_2, F_0\,]_{(r^{-1},\cdots,r^{-1},r^{-1}
,s,\cdots, s)}\in {\mathcal U}'_{r,s}(D_n^{(1)}),
\end{split}
\end{equation*}
then for any
$i\in I$, $x_i^+(-1)\in {\mathcal U}'_{r,s}(D_n^{(1)})$.\hfill$\Box$
\end{lemm}

\smallskip

\begin{lemm}
\ $(1)$
\begin{equation*}
\begin{split}
&x_1^-(1)\\
&\quad=[\,E_3,\cdots,E_6,E_2,E_4,E_3,E_{5},E_{4},E_2,
E_0\,]_{(r^{-1}s^{-2},r
,r,s^{-1},r,s^{-1},r,\cdots,r)}\gamma'\om_1\\
&\quad\in {\mathcal U}'_{r,s}(E_6^{(1)}),
\end{split}
\end{equation*}
then for any $i\in I$, $x_i^-(1)\in
{\mathcal U}'_{r,s}(E_6^{(1)})$.

\smallskip
$(2)$
\begin{equation*}
\begin{split}
&x_1^+(-1)\\
&\quad=\tau\Bigl(\bigl[\,E_3,\cdots, E_6, E_2, E_4, E_3, E_5, E_4, E_2,
E_0\,]_{(r^{-1}s^{-2},\,r
,\,r,\,s^{-1},\,r,\,s^{-1},\,r,\cdots,\,r)}\gamma'\om_1\Bigr)\\
&\quad=\gamma\om'_1\,
[\,F_0,F_2,F_4,F_5,F_3,F_{4},F_{2},F_6,\cdots,F_3\,]_{(s,\cdots
,s,r^{-1},s,r^{-1},s,s,r^{-2}s^{-1})}\\
&\quad\in {\mathcal
U}'_{r,s}(E_6^{(1)}),
\end{split}
\end{equation*}
 then for any $i\in I$, $x_i^+(-1)\in
{\mathcal U}'_{r,s}(E_6^{(1)})$.\hfill$\Box$
\end{lemm}
\smallskip

Furthermore, applying the above results and combining $(D8)$ with $(D6)$, we get
the following
\begin{lemm} \ $(1)$ \ $a_i(l)\in {\mathcal U}'_{r,s}(\hat{\frak{g}})$,
\ for $\;l\in\mathbb{Z}\backslash \{0\}$.

\smallskip
$(2)$ \
$x_i^{\pm}(k)\in {\mathcal U}'_{r,s}(\hat{\frak{g}})$, \ for
$\;k\in\mathbb{Z}$.\hfill$\Box$
\end{lemm}

Therefore, by induction, all generators are in the subalgebra ${\mathcal U}'_{r,s}(\widehat{\frak{g}})$.
So, this finishes the proof of Theorem $\mathcal{B}$.\hfill$\Box$

\subsection{Proof of Theorem $\mathcal{C}$} This subsection focuses on showing Theorem $\mathcal{C}$.

\smallskip
{\bf Theorem $\mathcal{C}$.}  {\it There exists a surjective $\Phi:\,{\mathcal U}'_{r,s}(\hat{\frak {g}}) \longrightarrow U_{r,s}(\hat{\frak {g}})$ such that $\Psi\Phi=\Phi\Psi=1$.}
\begin{proof}\, We define $\Psi$ on the generators as follows. For $i\in I_0$,
\begin{gather*}
\Psi(E_i)=e_i,\quad \Psi(F_i)=f_i,\quad \Psi(\omega_i)=\omega_i,\quad\Psi(\omega'_i)=\omega'_i,\\
\Psi(\gamma)=\gamma,\quad \Psi(\gamma')=\gamma',\quad \Psi(D)=D,\quad \Psi(D')=D'.
\end{gather*}
Consequently, it is not difficult to see that $\Psi\Phi=\Phi\Psi=1$.
\end{proof}

Up to now, we prove the Drinfeld Isomorphism Theorem for the two parameter quantum affine algebras.

\section{Vertex Representations}\,
In the last section, we turn to construct the level-one vertex representations of
two-parameter quantum affine algebras  $U_{r,s}(\widehat{\frak{g}})$ for types $X^{(1)}_n$ (where $X=ADE$). More precisely,
in our construction we can take $c=1$ in the Drinfeld relations in this section.

\subsection{Two-parameter quantum Heisenberg algebra}
 \ \begin{defi}\,
Two-parameter quantum Heisenberg algebra $U_{r,s}(\widehat{\frak{h}})$ is  the subalgebra
of $U_{r,s}(\widehat{\frak{g}})$ generated by $\{\,a_j(l),\,
\gamma^{\pm\frac1{2}},\,\gamma'^{\pm\frac1{2}} \mid
l\in\mathbb{Z}\backslash \{0\},\,j\in I\,\}$, satisfying the following relation, for $m,\, l \in \mathbb{Z}\backslash \{0\}$
\begin{equation*}
\ [\,a_i(m), a_j(l)\,]=\delta_{m+l, 0} \frac{
(rs)^{\frac{|m|}{2}}(r_is_i)^{-\frac{m a_{ij}}2}[\,m
a_{ij}\,]_i}{|m|}\frac{\gamma^{|m|}-{\gamma'}^{|m|}}{r_j-s_j}.
\end{equation*}
\end{defi}

We denote by
$U_{r,s}(\widehat{\frak h}^+)$ (resp. $U_{r,s}(\widehat{\frak
h}^-)$\,) the commutative subalgebra of
$U_{r,s}(\widehat{\frak{h}})$ generated by $a_j(l)$ (resp.
$a_j(-l)$) with $l\in\mathbb{Z}^{>0}$, $j\in I$. In fact, we have
$U_{r,s}(\widehat{\frak h}^-)=S(\widehat{\frak h}^-)$, where
$S(\widehat{\frak h}^-)$ is the symmetric algebra associated to
$\widehat{\frak h}^-$. Then $S(\widehat{\frak h}^-)$ is a
$U_{r,s}(\widehat{\frak{h}})$-module with the action defined by
\begin{gather*}
\gamma^{\pm\frac1{2}}\cdot v=r^{\pm\frac1{2}}\,v,
\qquad \gamma'^{\pm\frac1{2}}\cdot v=s^{\pm\frac1{2}}\,v,\\
a_i(-l)\cdot v=a_i(-l)\,v,\\
a_i(l)\cdot v=\sum_j\frac{
(rs)^{\frac{l}{2}}(r_is_i)^{-\frac{l a_{ij}}2}[\,l
a_{ij}\,]_i}{l}\cdot\frac{r^l-s^l}{r_j-s_j}\cdot\frac{d\, v}{d\, a_j(-l)}.
\end{gather*}
for any $v\in S(\widehat{\frak h}^-)$, $l\in\mathbb{Z}^{>0}$ and
$i\in I$.

\subsection{Fock space}
\ Let $Q=\bigoplus_{i\in I}\mathbb{Z}\al_i$ be the root lattice of
$\frak g$ with the Killing form $(\cdot\,|\,\cdot)$, one can form a
group algebra $\mathbb{K}[Q]$ with base elements of the form $e^\be$
($\be\in Q$), and the product
$$e^{\be}e^{\be'}=e^{\be+\be'},\qquad \be,\,\be'\in Q.$$
We define the Fock space as
$$
{\mathcal F}:=S(\widehat{\frak h}^-)\otimes\mathbb{K}[Q],
$$
and make it into a $U_{r,s}(\widetilde{\frak h})$-module (where
$U_{r,s}(\widetilde{\frak h})$ is generated by $\{\,a_i(\pm
l),\,\om_i^{\pm1},\,\om_i'^{\pm1},$
$\,\gamma^{\pm\frac1{2}},\,\gamma'^{\pm\frac1{2}}\mid i\in
I,\,l\in\mathbb{Z}^{>0}\,\}$) via extending the action of
$U_{r,s}(\widehat{\frak h}^-)$ to the Fock space ${\mathcal F}$. Let
$z$ be a complex variable and add the action of $\al_i(0)$ as
follows:
\begin{eqnarray*}
z^{\al_i(0)}(v\otimes e^\be)&=&z^{(\al_i\,|\,\be)}v\otimes e^\be,\\
e^\al(v\otimes e^\be)&=&v\otimes e^{\al+\be},\\
a_i(\pm l)\cdot (v\otimes e^\be)&=&(a_i(\pm l)\cdot v)\otimes e^\be,
\qquad l\in\mathbb{Z}\backslash \{0\},\\
\varepsilon_i(0)\cdot (v\otimes
e^\be)&=&(\varepsilon_i\,,\,\be)\,v\otimes e^\be,
\end{eqnarray*}
such that
\begin{gather*}
\om_i\cdot(v\otimes e^\be)
=\la \be,\, i\ra \,v\otimes e^\be,\\
 \om_i'\cdot (v\otimes
e^\be)=\la i,\, \be\ra^{-1}\,v\otimes e^\be.
\end{gather*}

\subsection{Vertex operators}\
 Let $\epsilon_0(\,,\,)$: $ Q\times Q\rightarrow
\mathbb K^*$
 be the cocycle such that
 \begin{eqnarray*}
&&\epsilon_0(\al, \be+\theta)=\epsilon_0(\al,\be)\epsilon_0(\al, \theta),\\
&&\epsilon_0(\al+\be, \theta)=\epsilon_0(\al, \theta)\epsilon_0(\be,
\theta),
\end{eqnarray*}
we construct such a cocycle directly by
$$\epsilon_0(\alpha_i,\alpha_j)=\left\{\begin{array}{cl} (-r_is_i)^{\frac{a_{ij}}{2}},
& \ i> j,\\
(rs)^{\frac{1}{2}},
& \ i= j,\\
1,& \ i<j.\\
\end{array}\right.$$

For $\alpha, \beta\in Q$, define $\mathbb{K}$-linear operators as
\begin{gather*}
\epsilon_\alpha(v\otimes e^\be)=
\epsilon_0(\alpha, \be)\,v\otimes e^\be,\\
D(r)(v\otimes e^\be)=r^{\be}v\otimes e^\be,\quad D(s)(v\otimes e^\be)=s^{\be}v\otimes e^\be.
\end{gather*}

We can now introduce the main vertex operators.
$$\begin{array}{rcl}
E_-^\pm(\al_i,z)&=&\exp\left(\pm\sum_{n=1}^\infty\dfrac{s^{\pm n/2}}{[n]}a_i(-n)z^n\right),\\
E_+^\pm(\al_i,z)&=&\exp\left(\mp\sum_{n=1}^\infty\dfrac{r^{\mp n/2}}{[n]}a_i(n)z^{-n}\right),\\
\end{array}
$$
where $[n]=\frac{r^n-s^n}{r-s}$, for $n\in\mathbb{Z}^{>0}$, $\al_i\in Q$.

\begin{theo} For the simply-laced cases, we have the vertex
representation $($of level-$1$$)$ $\pi$ of $U_{r,s}(\widehat{\frak g})$ on the Fock
space $\mathcal F$ as follows:
$$\begin{array}{rcl}
\gamma^{\pm1}&\mapsto& r^{\pm1},\qquad\ \; \gamma'^{\pm1}\mapsto s^{\pm1},\\
D&\mapsto& D(r),\qquad D'\mapsto D(s),\\
x_j^+(z)&\mapsto&X_j^+(z)=E_-^+(\al_j,
z)E_+^+(\al_j,z)e^{\al_j}z^{\al_j(0)+1}r^{\varepsilon_j(0)-\frac{1}{2}}\epsilon_{\al_j},\\
x_j^-(z)&\mapsto&X_j^-(z)=E_-^-(\al_j,
z)E_+^-(\al_j,z)e^{-\al_j}z^{-\al_j(0)+1}s^{\varepsilon_j(0)-\frac{1}{2}}\epsilon_{\al_j},\\
\om'_j(z)&\mapsto&\Phi_j(z)=\om'_j\exp\Big(-(r-s)\sum_{k>0}a_j(-k)z^k\Big)
,\\
\om_j(z)&\mapsto&\Psi_j(z)=\om_j \exp\Big((r-s)\sum_{k>0}a_j(k)z^{-k}\Big).
\end{array}$$
\end{theo}

 \vskip 0.5cm
\subsection{Proof of the theorem 5.2.}

We have to show that the operators $X_i^{\pm}(z)$, $\Phi_i(z)$ and $
\Psi_i(z)$ satisfy all the relations of Drinfeld's realization with
$\gamma=r,\, \gamma'=s$. More explicitly, we want to show that
$X_i^{\pm}(z),\, \Phi_i(z)$ and $ \Psi_i(z)$ satisfy relations
$(3.2)$--$(3.13)$  in Definition 3.1. It is clear that $(3.2)$--$(3.5)$
follow from the construction of vertex operators $\Phi_i(z)$ and $
\Psi_i(z)$. We are going to divide the proof into several steps.

It is easy to see that $\Psi_i,\, \Phi_i,\, e^{\al_j}$ have the
following commutative relations:
\begin{eqnarray}
&&\Phi_i(z)e^{\pm \alpha_j}=\la i,\,j\ra^{\mp1} e^{\pm
\alpha_j}\Phi_i(z),\\
&&\Psi_i(z)e^{\pm \alpha_j}=\la j,\, i\ra^{\pm1}e^{\pm
\alpha_j}\Psi_i(z).
\end{eqnarray}

Note that the action of $c$ is the identity, since here we construct a level-one representation.
$(3.6)$ follows from the following lemma.
\begin{lemm} For $i, j\in I$, we have
\begin{eqnarray}
g_{ij}\Big(\frac{z}{w}(rs)^{\frac{1}{2}}r\Big)\Phi_i(z)\Psi_j(w)
=g_{ij}\Big(\frac{z}{w}(rs)^{\frac{1}{2}}s\Big)\Psi_j(w)\Phi_i(z).
\end{eqnarray}

\end{lemm}
\noindent{\bf Proof: }When $a_{ij}=0$, the proof is trivial, so we
only check the relation for the case of $a_{ij}\neq 0$ (i.e.,
$a_{ij}=-1$ and $a_{ii}=2$).
\begin{eqnarray*}
\begin{split}
\Phi_i(z)&\Psi_j(w)\\
&=\om'_i\exp\Big(-(r-s) \sum_{k>0} a_i(-k)z^k\Big)
\cdot \om_j\exp\Big((r-s)\sum_{k>0}a_j(k)w^{-k}\Big)\\
&=\Psi_j(w)\Phi_i(z)\exp\Big(-(r-s)^2\sum_{k>0}[\,a_i(-k),\,
a_j(k)\,](\frac{z}{w})^k\Big)\\
&=\Psi_j(w)\Phi_i(z)\exp\Big(-\sum_{k>0}(rs)^{\frac{k}2} \frac{\lg
i,\,i\rg^{-\frac{a_{ij}k}{2}}-\lg i,\,i\rg^{\frac{a_{ij}k}{2}}}{k}
\cdot[k](\frac{z}{w})^k\Big)\\
&=\Psi_j(w)\Phi_i(z) \left\{\begin{array}{cl} \Big(\frac{\lg i,\,i
\rg^{-\frac{1}{2}} (\frac{z}{w}(rs)^{\frac{1}{2}}r)-1}{\lg i,\,i
\rg^{\frac{1}{2}}(\frac{z}{w}(rs)^{\frac{1}{2}}r)-1}\Big)^{-1}
\Big(\frac{\lg i,\,i \rg^{-\frac{1}{2}}
(\frac{z}{w}(rs)^{\frac{1}{2}}s)-1}{\lg i,\,i \rg^{\frac{1}{2}}
(\frac{z}{w}(rs)^{\frac{1}{2}}s)-1}\Big),
& \ i\neq j,\vspace{3mm}\\
\Big(\frac{\lg i,\,i \rg(\frac{z}{w}(rs)^{\frac{1}{2}}r)-1}{\lg
i,\,i \rg(\frac{z}{w}(rs)^{\frac{1}{2}}r)-1}\Big)^{-1}
\Big(\frac{\lg i,\,i \rg (\frac{z}{w}(rs)^{\frac{1}{2}}s)-1}{\lg
i,\,i \rg(\frac{z}{w}(rs)^{\frac{1}{2}}s)-1}\Big),
& \ i= j,
\end{array}\right.
\\
&=\Psi_j(w)\Phi_i(z)g_{ij}\Big(\frac{z}{w}(rs)^{\frac{1}{2}}r\Big)^{-1}g_{ij}
\Big(\frac{z}{w}(rs)^{\frac{1}{2}}s\Big),
\end{split}
\end{eqnarray*}
where we used the formal identity
$log(1-z)=-\sum_{n=1}^{\infty}\frac{z^{n}}{n}$.
\hfill$\Box$

Furthermore, the following Lemma 5.4 follows from the
relations $(3.7)$ and $(3.8)$.
\begin{lemm} For $i, j\in I$, we have
\begin{eqnarray}
\begin{split}
\Phi_i(z)X_j^{\pm}(w)\Phi_i(z)^{-1}&=g_{ij}\Big(\frac{z}{w}(rs)^{\frac{1}{2}}r^{\mp
\frac{1}{2}}\Big)^{\pm1}X_j^{\pm}(w),\\
\Psi_i(z)X_j^{\pm}(w)\Psi_i(z)^{-1}&=g_{ji}\Big(\frac{w}{z}(rs)^{\frac{1}{2}}s^{\pm
\frac{1}{2}}\Big)^{\mp1}X_j^{\pm}(w).
\end{split}
\end{eqnarray}
\end{lemm}
\noindent{\bf Proof: }  Naturally we only check $(5.2)$ and $(5.3)$ for
the case of  $a_{ij}\neq 0$.  The vertex operator is a product of two exponential
operators and a middle term operator. So we first consider
\begin{eqnarray*}
\begin{split}
\Phi_i(z)&E_+^{\pm}(\alpha_j,\, w)\\
&=\exp\Big(-(r-s)\sum_{k>0}a_i(-k)z^k\Big)r_i^{\varepsilon_{i+1}(0)}s_i^{\varepsilon_i(0)}\exp \Big(\mp\sum_{k>0}\frac{r^{\mp \frac{k}{2}}}{[k]}a_j(k)w^{-k}\Big)\\
&=E_+^{\pm}(\alpha_j,\,w)\Phi_i(z)\exp\Big(\pm(r-s))\sum_{k>0}\frac{r^{\mp\frac{k}{2}}}{[k]}[\,a_i(-k),\,
a_j(k)\,](\frac{z}{w})^k\Big)\\
&=E_+^{\pm}(\alpha_j,\, w)\Phi_i(z) \exp
\Big(\pm\sum_{k>0}(rs)^{\frac{k}2} r^{\mp \frac{k}{2}}\frac{\lg
i,\,i\rg^{-\frac{a_{ij}k}{2}}-\lg i,\,i\rg^{\frac{a_{ij}k}{2}}}{k}
(\frac{z}{w})^k\Big)
\\
&=E_+^{\pm}(\alpha_j,\, w)\Phi_i(z)\left\{\begin{array}{cl}
\Big(\frac{\lg i,\,i\rg^{-\frac{1}{2}}
(\frac{z}{w}(rs)^{\frac{1}{2}}r^{\mp \frac{1}{2}})-1}{\lg
i,\,i\rg^{\frac{1}{2}} (\frac{z}{w}(rs)^{\frac{1}{2}}r^{\mp
\frac{1}{2}})-1}\Big)^{\pm1}, & \ i\neq j,
\vspace{3mm}\\
\Big(\frac{\lg i,\,i \rg(\frac{z}{w}(rs)^{\frac{1}{2}}r^{\mp
\frac{1}{2}})-1}{\lg i,\,i \rg(\frac{z}{w}(rs)^{\frac{1}{2}}r^{\mp
\frac{1}{2}})-1}\Big)^{\pm1},
& \ i= j,\\
\end{array}\right.
\\
&=E_+^{\pm}(\alpha_j,\,w)\Phi_i(z)g_{ij}\Big(\frac{z}{w}(rs)^{\frac{1}{2}}r^{\mp
\frac{1}{2}}\Big)^{\pm1} \la i, j\ra^{\pm1}.
\end{split}
\end{eqnarray*}
We would proceed in the same way with the following result,
\begin{eqnarray*}
\begin{split}
\Psi_i(z)&E_-^{\pm}(\alpha_j,\, w)\\
&=\exp\Big((r-s)\sum_{k>0}a_i(k)z^{-k}\Big)r_i^{\varepsilon_i(0)}s_i^{\varepsilon_{i+1}(0)}
 \exp \Big(\pm\sum_{k>0}\frac{s^{\pm \frac{k}{2}}}{[k]}a_j(-k)w^{k}\Big)\\
&=E_-^{\pm}(\alpha_j,\,w)\Psi_i(z)\exp\Big(\pm(r-s)\sum_{k>0}\frac{s^{\pm\frac{k}{2}}}{[k]}[\,a_i(k),\,
a_j(-k)\,](\frac{w}{z})^k\Big)\\
&=E_-^{\pm}(\alpha_j,\, w)\Psi_i(z) \exp
\Big(\pm\sum_{k>0}(rs)^{\frac{k}2} s^{\pm \frac{k}{2}}\frac{\lg
i,\,i\rg^{\frac{a_{ij}k}{2}}-\lg
i,\,i\rg^{-\frac{a_{ij}k}{2}}}{k}(\frac{z}{w})^k\Big)\\
&=E_-^{\pm}(\alpha_j,\, w)\Psi_i(z)\left\{\begin{array}{cl}
\Big(\frac{\lg i,\,i\rg^{-\frac{1}{2}}
(\frac{z}{w}(rs)^{\frac{1}{2}}s^{\pm \frac{1}{2}})-1}{\lg
i,\,i\rg^{\frac{1}{2}} (\frac{z}{w}(rs)^{\frac{1}{2}}s^{\pm
\frac{1}{2}})-1}\Big)^{\mp1}, & \ i\neq j,
\vspace{3mm}\\
\Big(\frac{\lg i,\,i \rg(\frac{z}{w}(rs)^{\frac{1}{2}}s^{\pm
\frac{1}{2}})-1}{\lg i,\,i \rg(\frac{z}{w}(rs)^{\frac{1}{2}}s^{\pm
\frac{1}{2}})-1}\Big)^{\mp1},
& \ i=j,
\end{array}\right.
\\
&=E_-^{\pm}(\alpha_j,\,w)\Psi_i(z)g_{ji}\Big(\frac{z}{w}(rs^{-1})^{\mp
\frac{1}{4}}\Big)^{\pm1} \la j, i\ra^{\mp1}.
\end{split}
\end{eqnarray*}

Applying $(5.1)$ and $(5.2)$, we would arrive at the required results.
Thus the proof of the lemma is complete.  \hfill $\Box$
\medskip

Before checking the relations $(3.9)$ and $(3.10)$ and the quantum
Serre-relations, we need to introduce a useful notion---normal
ordering, which plays an important role in the theory of ordinary
vertex operator calculus. Define
\begin{gather*}
:\alpha_i(n)\alpha_j(-n):\;=\;:\alpha_j(-n)\alpha_i(n):\;=\alpha_j(-n)\alpha_i(n),\\
:\alpha_i(0)a_j:\;=\;:a_j\alpha_i(0):\;=\frac{1}{2}(\alpha_i(0)a_j+a_j\alpha_i(0)).
\end{gather*}
We can extend the notion to the vertex operators. For example, we
define
\begin{eqnarray*}
\begin{split}
:X_i^{\pm}(z)X_j^{\pm}(w):&=E_-^{\pm}(\alpha_i,\, z)E_-^{\pm}(\alpha_j,\, w)E_+^{\pm}(\alpha_i,\,
z)E_+^{\pm}(\alpha_j,\,w)\\
&\quad\cdot e^{\pm (\alpha_i+\alpha_j)}z^{\pm
\alpha_i(0)+1}w^{\pm\alpha_j(0)+1}\epsilon_{\pm\alpha_i}\epsilon_{\pm\alpha_j},\\
:X_i^{\pm}(z)X_j^{\mp}(w):
&=E_-^{\pm}(\alpha_i,\, z)E_-^{\mp}(\alpha_j,\, w)E_+^{\pm}(\alpha_i,\,
z)E_+^{\mp}(\alpha_j,\,w)\\
&\quad\cdot e^{\pm (\alpha_i-\alpha_j)}z^{\pm
\alpha_i(0)+1}w^{\mp\alpha_j(0)+1}\epsilon_{\pm\alpha_i}\epsilon_{\mp\alpha_j}.
\end{split}
\end{eqnarray*}

Therefore, the following formulas hold
\begin{gather*}
:X_i^{\pm}(z)X_j^{\pm}(w):\;=\;:X_j^{\pm}(w)X_i^{\pm}(z):\,,\\
:X_i^{\pm}(z)X_j^{\mp}(w):\;=\;:X_j^{\pm}(w)X_i^{\mp}(z):.
\end{gather*}

Using the above notation, we give the operator product expansions as follows
\begin{lemm} For $i\neq j\in I$ such that $a_{ij}\neq 0$, we have the
following operator product expansions.
\begin{eqnarray*}
\begin{split}
X_i^{\pm}(z)X_j^{\mp}(w)&=:X_i^{\pm}(z)X_j^{\mp}(w):
\big((\frac{z}{w})^{\frac{1}{2}}-(rs)^{\mp\frac{1}{2}}(\frac{w}{z})^{\frac{1}{2}}\big)
\epsilon_{0}(\alpha_i, \alpha_j)^{\mp1} \\
X_i^{\pm}(z)X_j^{\pm}(w)&=:X_i^{\pm}(z)X_j^{\pm}(w):
\big(1-(r^{-1}s)^{\pm\frac{1}{2}}(\frac{w}{z})\big)^{-1}(\frac{w}{z})^{\frac{1}{2}}
\epsilon_{0}(\alpha_i, \alpha_j)^{\pm1}\\
X_i^{\pm}(z)X_i^{\mp}(w)&=:X_i^{\pm}(z)X_i^{\mp}(w):
\big(1-(rs)^{\frac{1}{2}}(rs)^{\mp\frac{1}{2}}r^{-1}\frac{w}{z}\big)^{-1}\\
&\hskip2cm\big(1-(rs)^{\frac{1}{2}}(rs)^{\mp\frac{1}{2}}s^{-1}\frac{w}{z}\big)^{-1}\frac{w}{z}
\epsilon_{0}(\alpha_i,\alpha_i)^{\mp1}  \\
 X_i^{\pm}(z)X_i^{\pm}(w)&=:X_i^{\pm}(z)X_i^{\pm}(w):
\big(1-(r^{-1}s)^{\pm\frac{1}{2}}(rs)^{\frac{1}{2}}r^{-1}\frac{w}{z}
\big)\\
&\hskip2cm\big(1-(r^{-1}s)^{\pm\frac{1}{2}}(rs)^{\frac{1}{2}}s^{-1}\frac{w}{z}
\big)\frac{z}{w}\epsilon_{0}(\alpha_i, \alpha_i)^{\pm1},
\end{split}
\end{eqnarray*}
where the two factors are called contraction factors, denoted as
$\underline{X_i^{\pm}(z)X_j^{\mp}(w)}$ and
$\underline{X_i^{\pm}(z)X_j^{\pm}(w)}$ respectively in the sequel.
\end{lemm}
\noindent{\bf Proof.}  By the definition of normal ordering, the first and third formula follow from the following
equalities.
\begin{eqnarray*}
\begin{split}
E_+^{\pm}(&\alpha_i,\, z)E_-^{\mp}(\alpha_j,\, w)\\
&=\exp\Big(\mp\sum_{k>0}\frac{r^{\mp\frac{k}{2}}}{[k]}a_i(k)z^{-k}\Big)
\exp \Big(\mp\sum_{k>0}\frac{s^{\mp \frac{k}{2}}}{[k]}a_j(-k)w^{k}\Big)\\
&=E_-^{\pm}(\alpha_j,\,w)E_+^{\mp}(\alpha_i,\,z)
\exp\Big(\sum_{k>0}\frac{(rs)^{\mp\frac{k}{2}}}{[k]^2}[\,a_i(k),\,
a_j(-k)\,](\frac{w}{z})^k\Big)\\
&=E_-^{\pm}(\alpha_j,\,w)E_+^{\mp}(\alpha_i,\,z)
\exp\Big(\sum_{k>0}\frac{(rs)^{\mp\frac{k}{2}}}{[k]}
\frac{(rs)^{\frac{k}2} \big(\lg i,\,i\rg^{\frac{a_{ij}k}{2}}-\lg
i,\,i\rg^{-\frac{a_{ij}k}{2}}\big)}{k(r-s)}(\frac{w}{z})^k\Big),\\
&=E_-^{\pm}(\alpha_j,\,w)E_+^{\mp}(\alpha_i,\,z)\\
&\hskip0.8cm \cdot \left\{\begin{array}{ll}
\big(1-(rs)^{\mp\frac{1}{2}}\frac{w}{z})\big),
& \ i\neq j,\vspace{2mm}\\
\big(1-(rs)^{\mp\frac{1}{2}}(r^{-1}s)^{\frac{1}{2}}\frac{w}{z}\big)^{-1}
\big(1-(rs)^{\mp\frac{1}{2}}(rs^{-1})^{\frac{1}{2}}\frac{w}{z}\big)^{-1},
& ~i=j,\\
\end{array}\right.
\\
\end{split}
\end{eqnarray*}
It remains to deal with the rest two formulas, which hold from the following results.
\begin{eqnarray*}
\begin{split}
E_+^{\pm}&(\alpha_i,\, z)E_-^{\pm}(\alpha_j,\, w)\\
&=\exp\Big(\mp\sum_{k>0}\frac{r^{\mp\frac{k}{2}}}{[k]}a_i(k)z^{-k}\Big)
\exp \Big(\pm\sum_{k>0}\frac{s^{\pm \frac{k}{2}}}{[k]}a_j(-k)w^{k}\Big)\\
&=E_-^{\pm}(\alpha_j,\,w)E_+^{\pm}(\alpha_i,\,z)
\exp\Big(-\sum_{k>0}\frac{(r^{-1}s)^{\pm\frac{k}{2}}}{[k]^2}[\,a_i(k),\,
a_j(-k)\,](\frac{w}{z})^k\Big)\\
&=E_-^{\pm}(\alpha_j,\,w)E_+^{\mp}(\alpha_i,\,z)
\exp\Big(-\sum_{k>0}\frac{(r^{-1}s)^{\pm\frac{k}{2}}}{[k]}
\frac{(rs)^{\frac{k}2} \big(\lg i,\,i\rg^{\frac{a_{ij}k}{2}}-\lg
i,\,i\rg^{-\frac{a_{ij}k}{2}}\big)}{k(r-s)}(\frac{w}{z})^k\Big)\\
&=E_-^{\pm}(\alpha_j,\,w)E_+^{\pm}(\alpha_i,\,z)\\
&\hskip0.8cm \cdot \left\{\begin{array}{ll}
\big(1-(r^{-1}s)^{\pm\frac{1}{2}}\frac{w}{z} \big)^{-1}, & \ i\neq
j,\vspace{2mm}\\
\big(1-(r^{-1}s)^{\pm\frac{1}{2}}(r^{-1}s)^{\frac{1}{2}}\frac{w}{z}
\big)
\big(1-(r^{-1}s)^{\pm\frac{1}{2}}(rs^{-1})^{\frac{1}{2}}\frac{w}{z}
\big),
& ~i=j,\\
\end{array}\right.
\end{split}
\end{eqnarray*}
where we have used the formula: $log(1-x)=-\sum_{n>0}\frac{x^n}{n}$.
\hfill $\Box$
\medskip

Now we turn to check the relation $(3.10)$.

\begin{lemm} The vertex operators satisfy the following
\begin{eqnarray*}
F_{ij}^{\pm}(z,\,w)\,X_i^\pm(z)\,X_j^{\pm}(w)=G_{ij}^{\pm}(z,\,w)\,X_j^{\pm}(w)\,X_i^{\pm}(z),
\end{eqnarray*}
where $F_{ij}^{\pm}(z,\,w)$ and $G_{ij}^{\pm}(z,\,w)$ are defined as
definition 3.1.
\end{lemm}
\noindent{\bf Proof.} Similarly, if $a_{ij}=0$, it is trivial, So we only show the Lemma for $a_{ij}\neq 0$.
First we notice that

\begin{eqnarray*}
\frac{G_{ij}^{\pm}(z,\,w)}{F_{ij}^{\pm}(z,\,w)} &=&\frac{\lg
j,i\rg^{\pm1}z-(\lg j,i\rg\lg i,j\rg^{-1})^{\pm\frac1{2}}w}
{z-(\lg i,j\rg\lg j,i\rg)^{\pm\frac1{2}}w}\\
&=&{\left\{\begin{array}{ll}
(rs)^{\mp\frac{1}{2}}\frac{(r^{-1}s)^{\pm\frac{1}{2}}z-w}
{z-(r^{-1}s)^{\pm\frac{1}{2}}w},
& \ i> j;\vspace{2mm}\\
(rs)^{\pm\frac{1}{2}}\frac{(r^{-1}s)^{\pm\frac{1}{2}}z-w}
{z-(r^{-1}s)^{\pm\frac{1}{2}}w},
& ~i<j,\vspace{2mm}\\
\frac{(rs^{-1})^{\pm\frac{1}{2}}z-w}
{z-(rs^{-1})^{\pm\frac{1}{2}}w}, & \ i=j,
\end{array}\right.}\\
&=&\frac{\underline{X_i^\pm(z)\,X_j^{\pm}(w)}}
{\underline{X_j^{\pm}(w)\,X_i^{\pm}(z)}}
=\frac{X_i^\pm(z)\,X_j^{\pm}(w)}{X_j^{\pm}(w)\,X_i^{\pm}(z)}.
\end{eqnarray*}
On the other hand,  we have $$:X_i^\pm(z)\,X_j^{\pm}(w):=:X_j^\pm(w)\,X_i^{\pm}(z):$$
 Thus we get the required statement immediately. \hfill $\Box$
\medskip

Furthermore,  before verifying the relation $(3.9)$, we need the following lemma.

\begin{lemm} We claim that
$$:X_i^+(z)X_i^-(zr):
=\Phi_i(zs^{-\frac{1}{2}})(rs)^{-\frac{1}{2}},$$
$$:X_i^+(ws^{-1})X_i^-(w):=\Psi_i(wr^{\frac{1}{2}})(rs)^{-\frac{1}{2}},$$
which can be easily verified directly.
\end{lemm}

\noindent{\bf Proof.}  Continuing to use the above notations, we get immediately,
\begin{eqnarray*}
\begin{split}
:X_i^+&(z)X_i^-(z r):\\
&=
\exp\left(\sum_{n=1}^\infty\dfrac{s^{
n/2}}{[n]}a_i(-n)(z)^n\right)
\exp\left(-\sum_{n=1}^\infty\dfrac{s^{-
n/2}}{[n]}a_i(-n)(z r)^n\right)\\
&\exp\left(-\sum_{n=1}^\infty\dfrac{r^{-
n/2}}{[n]}a_i(n)(z)^{-n}\right) \exp\left(\sum_{n=1}^\infty\dfrac{r^
{n/2}}{[n]}a_i(n)(z
r)^{-n}\right)r^{-\alpha_i(0)} (rs)^{\varepsilon_i-\frac{1}{2}}\\
&=\exp\Big(-(r-s)\sum_{n>0}a_i(-n)(zs^{-\frac{1}{2}})^k\Big)
r^{\varepsilon_{i+1}(0)}s^{\varepsilon_i(0)}(rs)^{-\frac{1}{2}}\\
&=\Phi_i(zs^{-\frac{1}{2}})(rs)^{-\frac{1}{2}}.
\end{split}
\end{eqnarray*}
We would proceed in the same way with the second formula.
\begin{eqnarray*}
\begin{split}
:\,&X_i^+(ws^{-1})X_i^-(w):\\
&=\exp\left(\sum_{n=1}^\infty\dfrac{s^{
n/2}}{[n]}a_i(-n)(ws^{-1})^n\right)\exp\left(-\sum_{n=1}^\infty\dfrac{s^{-
n/2}}{[n]}a_i(-n)(w)^n\right)\\
&\exp\left(-\sum_{n=1}^\infty\dfrac{r^{-
n/2}}{[n]}a_i(n)(ws^{-1})^{-n}\right)
\exp\left(\sum_{n=1}^\infty\dfrac{r^ {n/2}}{[n]}a_i(n)(w
)^{-n}\right)s^{-\alpha_i(0)}(rs)^{\varepsilon_i-\frac{1}{2}} \\
&=\exp\Big((r-s)\sum_{n>0}a_i(n)(wr^{-\frac{1}{2}})^{-k}\Big)
r^{\varepsilon_i(0)}s^{\varepsilon_{i+1}(0)}(rs)^{-\frac{1}{2}}\\
&=\Psi_i(wr^{-\frac{1}{2}})(rs)^{-\frac{1}{2}}.
\end{split}
\end{eqnarray*}

Now, we are ready to check the following proposition, which will yield the relation $(3.9)$.

\begin{prop}  \, For $i, j \in I$, one has
$$[\,X_i^+(z),
X_j^-(w)\,]=\frac{\delta_{ij}}{r_i-s_i}\Big(\delta(zw^{-1}s)\psi_i(wr^{\frac{1}2})
-\delta(zw^{-1}r)\phi_i(zs^{-\frac1{2}})\Big).$$
\end{prop}
\noindent{\bf Proof.}\, Firstly, observe that
$$:X_j^-(w)X_i^+(z):\;=\;:X_i^+(z)X_j^-(w):\;.$$
If $i> j$, it is easy to see that
\begin{eqnarray*}
\begin{split}
[\,&X_i^+(z), X_j^-(w)\,]\\
&=:X_i^+(z)X_j^-(w):
\big(\underline{X_i^+(z)X_j^-(w)}-\underline{X_j^-(w)X_i^+(z)}\big)\\
&=:X_i^+(z)X_j^-(w):
\Big(-\big((\frac{z}{w})^{\frac{1}{2}}-(rs)^{-\frac{1}{2}}(\frac{w}{z})^{\frac{1}{2}}\big)(rs)^{\frac{1}{2}}
-\big((\frac{w}{z})^{\frac{1}{2}}-(rs)^{\frac{1}{2}}(\frac{z}{w})^{\frac{1}{2}}\big)\Big)\\
&=0.
\end{split}
\end{eqnarray*}
Similarly,  we get for $i< j$,
\begin{eqnarray*}
\begin{split}
[\,&X_i^+(z), X_j^-(w)\,]\\
&=:X_i^+(z)X_j^-(w):
\big(\underline{X_i^+(z)X_j^-(w)}-\underline{X_j^-(w)X_i^+(z)}\big)\\
&=:X_i^+(z)X_j^-(w):
\Big(\big((\frac{z}{w})^{\frac{1}{2}}-(rs)^{-\frac{1}{2}}(\frac{w}{z})^{\frac{1}{2}}\big)
+(rs)^{-\frac{1}{2}}\big((\frac{w}{z})^{\frac{1}{2}}-(rs)^{\frac{1}{2}}(\frac{z}{w})^{\frac{1}{2}}\big)\Big)\\
&=0.
\end{split}
\end{eqnarray*}

It suffices to show  the case of $i=j$. Firstly, we get directly,
\begin{eqnarray*}
\underline{X_i^{+}(z)X_i^{-}(w)}
&=&(rs)^{-\frac{1}{2}}\frac{1}{\big(1-r^{-1}\frac{w}{z}\big)}
\frac{1}{\big(1-s^{-1}\frac{w}{z}\big)}\frac{w}{z}\\
&=&\frac{(rs)^{-\frac{1}{2}}}{r-s}\Big(\frac{r\frac{w}{z}}{1-s^{-1}\frac{w}{z}}
-\frac{s\frac{w}{z}}{1-r^{-1}\frac{w}{z}}\Big).
\end{eqnarray*}
On the other hand, it is clear to see that
\begin{eqnarray*} \underline{X_i^{-}(w)X_i^{+}(z)}
&=&\frac{(rs)^{\frac{1}{2}}}{r-s}\Big(\frac{r\frac{z}{w}}{1-r\frac{z}{w}}
-\frac{s\frac{z}{w}}{1-s\frac{z}{w}}\Big).
\end{eqnarray*}

As a consequence, the bracket becomes
\begin{eqnarray*}
\begin{split}
[\,&X_i^+(z), X_i^-(w)\,]\\
&=:X_i^+(z)X_i^-(w):\frac{(rs)^{-\frac{1}{2}}}{r-s}\Big(
\big(\frac{r\frac{w}{z}}{1-s^{-1}\frac{w}{z}}
-\frac{s\frac{w}{z}}{1-r^{-1}\frac{w}{z}}\big)
-\big(\frac{r^2s\frac{z}{w}}{1-r\frac{z}{w}}
-\frac{rs^2\frac{z}{w}}{1-s\frac{z}{w}}\big)\Big)\\
&=:X_i^+(z)X_i^-(w):\frac{(rs)^{-\frac{1}{2}}}{r-s} \big(\sum_{n\geq 0}s^{-n}r(\frac{w}{z})^{n+1}\\
&\hskip2cm-\sum_{n\geq 0}r^{-n}s(\frac{w}{z})^{n+1}-\sum_{n\geq
0}r^{n+2}s(\frac{z}{w})^{n+1}+\sum_{n\geq
0}rs^{n+2}(\frac{z}{w})^{n+1}\big)\\
&=:X_i^+(z)X_i^-(w):\frac{(rs)^{-\frac{1}{2}}}{r-s}
\Big(\sum_{n\geq
1}(rs)\{\big(s^{-n}(\frac{z}{w})^{-n}+s^n(\frac{z}{w})^{n}\big)\\
&\hskip6.6cm-\big(r^{-n}(\frac{z}{w})^{-n}+r^n(\frac{z}{w})^{n}\big)\}\Big)\\
&=:X_i^+(z)X_i^-(w):(rs)^{\frac{1}{2}}\frac{(\delta(\frac{z}{w}s)-\delta(\frac{z}{w}r))}{r-s}.
\end{split}
\end{eqnarray*}
Applying the above Lemma 5.7, then we arrive at
\begin{eqnarray*}
\begin{split}
[\,X_i^+(z), X_i^-(w)\,]
&=:X_i^+(z)X_i^-(w):(rs)^{\frac{1}{2}}\frac{(\delta(\frac{z}{w}s)-\delta(\frac{z}{w}r))}{r-s}\\
&=\frac{1}{r-s}\Big(\Psi_i(wr^{\frac{1}{2}})\delta(\frac{zs}{w})
-\Phi_i(zs^{-\frac{1}{2}})\delta(\frac{zr}{w})\Big),
\end{split}
\end{eqnarray*}
where we have used the property of $\delta$-function:
$$f(z_1,\, z_2)\delta(\frac{z_1}{z_2})=f(z_1,\, z_1)\delta(\frac{z_1}{z_2})
=f(z_2,\, z_2)\delta(\frac{z_1}{z_2}).$$ \hfill $\Box$

Lastly, we are left to show the quantum Serre-relations $(3.11)$---$(3.13)$.\\

For simplicity, we only check
the $``+"$ case of $a_{ij}=-1, \, 1\leq i<j <n$, i.e.,
\begin{eqnarray}
&&X_i^{+}(z_1)
X_i^{+}(z_2)X_j^{+}(w)-(r+s)\,X_i^{+}(z_1)X_j^{+}(w)X_i^{+}(z_2)\\
&&\qquad+\,(rs) X_j^{+}(w)X_i^{+}(z_1)X_i^{+}(z_2)+\{z_1\leftrightarrow
z_2\}=0.\nonumber
\end{eqnarray}

The others can be obtained similarly.

Let us review the following formulas for further reference.
\begin{gather*}
X_j^{+}(w)X_i^{+}(z)=:X_j^{+}(w)X_i^{+}(z):(\frac{z}{w})^{\frac{1}{2}}
\big(1-(r^{-1}s)^{\frac{1}{2}}(\frac{z}{w})\big)^{-1}\epsilon_0(\alpha_j,
\alpha_i),\\
X_i^{+}(z_1)X_i^{+}(z_2)=:X_i^{+}(z_1)X_i^{+}(z_2):(\frac{z_2}{z_1})^{-1}
\big(1-\frac{z_2}{z_1} \big) \big(1-(r^{-1}s)\frac{z_2}{z_1}
\big)\epsilon_0(\alpha_i, \alpha_i).
\end{gather*}
By the properties of normal ordering,  we then have
\begin{eqnarray*}
\begin{split}
X_j^{+}&(w)X_i^{+}(z_1)X_i^{+}(z_2)\\
&=:X_j^{+}(w)X_i^{+}(z_1)X_i^{+}(z_2):(\frac{z_1}{w})(\frac{z_2}{z_1})^{-\frac{1}{2}}
\frac{(rs)^{-1/2}\big(1-\frac{z_2}{z_1} \big)
\big(1-(r^{-1}s)\frac{z_2}{z_1}
\big)}{\big(1-(r^{-1}s)^{\frac{1}{2}}(\frac{z_1}{w})\big)\big(1-(r^{-1}s)^{\frac{1}{2}}(\frac{z_2}{w})\big)},
\end{split}
\end{eqnarray*}
\begin{eqnarray*}
\begin{split}
X_i^{+}&(z_1)X_j^{+}(w)X_i^{+}(z_2)\\
&=:X_i^{+}(z_1)X_j^{+}(w)X_i^{+}(z_2):(\frac{z_2}{z_1})^{-\frac{1}{2}}
\frac{\big(1-\frac{z_2}{z_1} \big) \big(1-(r^{-1}s)\frac{z_2}{z_1}
\big)}{\big(1-(r^{-1}s)^{\frac{1}{2}}(\frac{w}{z_1})\big)\big(1-(r^{-1}s)^{\frac{1}{2}}(\frac{z_2}{w})\big)},
\end{split}
\end{eqnarray*}
\begin{eqnarray*}
\begin{split}
X_i^{+}&(z_1)X_i^{+}(z_2)X_j^{+}(w)\\
&=:X_i^{+}(z_1)X_i^{+}(z_2)X_j^{+}(w):(\frac{w}{z_1})(\frac{z_2}{z_1})^{-\frac{1}{2}}
\frac{(rs)^{1/2}\big(1-\frac{z_2}{z_1} \big)
\big(1-(r^{-1}s)\frac{z_2}{z_1}
\big)}{\big(1-(r^{-1}s)^{\frac{1}{2}}(\frac{w}{z_1})\big)\big(1-(r^{-1}s)^{\frac{1}{2}}(\frac{w}{z_2})\big)}.
\end{split}
\end{eqnarray*}

Note that
$$
:X_j^{\pm}(w)X_i^{\pm}(z_1)X_i^{\pm}(z_2):
\;=\;:X_i^{\pm}(z_1)X_j^{\pm}(w)X_i^{\pm}(z_2):
\;=\;:X_i^{\pm}(z_1)X_i^{\pm}(z_2)X_j^{\pm}(w):\;.
$$

Let us proceed to show this formal series identity.  Both as formal
series and fractions,  we can express the contraction factors as
follows.
\begin{eqnarray*}
\begin{split}
& \frac{\big(z_1-z_2\big) \big(1-(r^{-1}s)\frac{z_2}{z_1}
\big)}{\big(1-(r^{-1}s)^{\frac{1}{2}}(\frac{z_1}{w})\big)\big(1-(r^{-1}s)^{\frac{1}{2}}(\frac{z_2}{w})\big)}\\
&\quad=w\big(1-(r^{-1}s)\frac{z_2}{z_1}
\big)\Big(\frac{(r^{-1}s)^{-\frac{1}{2}}}{1-(r^{-1}s)^{\frac{1}{2}}\frac{z_1}{w}}
-\frac{(r^{-1}s)^{-\frac{1}{2}}}{1-(r^{-1}s)^{\frac{1}{2}}\frac{z_2}{w}}\Big),
\end{split}
\end{eqnarray*}
\begin{eqnarray*}
\begin{split}
& \frac{\big(z_1-z_2\big) \big(1-(r^{-1}s)\frac{z_2}{z_1}
\big)}{\big(1-(r^{-1}s)^{\frac{1}{2}}(\frac{w}{z_1})\big)\big(1-(r^{-1}s)^{\frac{1}{2}}(\frac{z_2}{w})\big)}\\
&\quad=w(z_1-z_2)\Big(\frac{(r^{-1}s)^{\frac{1}{2}}}{z_1(1-(r^{-1}s)^{\frac{1}{2}}\frac{w}{z_1})}
+\frac{1}{w(1-(r^{-1}s)^{\frac{1}{2}}\frac{z_2}{w})}\Big),
\end{split}
\end{eqnarray*}
\begin{eqnarray*}
\begin{split}
& \frac{\big(z_1-z_2\big) \big(1-(r^{-1}s)\frac{z_2}{z_1}
\big)}{\big(1-(r^{-1}s)^{\frac{1}{2}}(\frac{w}{z_1})\big)\big(1-(r^{-1}s)^{\frac{1}{2}}(\frac{w}{z_2})\big)}\\
&\quad=\big(1-(r^{-1}s)\frac{z_2}{z_1}
\big)\Big(\frac{z_1}{1-(r^{-1}s)^{\frac{1}{2}}\frac{w}{z_2}}-\frac{z_2}{1-(r^{-1}s)^{\frac{1}{2}}\frac{w}{z_1}}\Big).
\end{split}
\end{eqnarray*}

Changing the positions of $z_1$ and $z_2$, we obtain the other three
expressions for the part $\{z_1\leftrightarrow z_2\}$ in the quantum
Serre relation. Substituting the above expressions into the right hand
side of $(5.6)$ and pulling out the common factors, we
obtain that
\begin{eqnarray*}
&&(z_1z_2)^{-\frac{1}{2}}\Big\{ (z_1-(r^{-1}s)z_2)
\Big(\frac{(r^{-1}s)^{-\frac{1}{2}}}{1-(r^{-1}s)^{\frac{1}{2}}\frac{z_1}{w}}
-\frac{(r^{-1}s)^{-\frac{1}{2}}}{1-(r^{-1}s)^{\frac{1}{2}}\frac{z_2}{w}}\Big)\\
&& ~~~+(r+s)(rs)^{-\frac{1}{2}}w(z_1-z_2)
\Big(\frac{(r^{-1}s)^{\frac{1}{2}}}{z_1(1-(r^{-1}s)^{\frac{1}{2}}\frac{w}{z_1})}
+\frac{1}{w(1-(r^{-1}s)^{\frac{1}{2}}\frac{z_2}{w})}\Big)\\
&&~~~+\frac{w}{z_2}(1-(r^{-1}s)\frac{z_2}{z_1})\Big(\frac{z_1}{1-(r^{-1}s)^{\frac{1}{2}}\frac{w}{z_2}}
-\frac{z_2}{1-(r^{-1}s)^{\frac{1}{2}}\frac{w}{z_1}}\Big)+(z_1
\leftrightarrow z_2)\Big\},
\end{eqnarray*}
where $\frac{1}{1-(r^{-1}s)^{\frac{1}{2}}\frac{z_1}{w}}$ stands for
$\sum_{n=0}^{\infty}(r^{-1}s)^{\frac{n}{2}}(\frac{z_1}{w})^n$ and
similar fraction for other formal power series.

Collecting the factors
$\frac{1}{1-(r^{-1}s)^{\frac{1}{2}}\frac{z_1}{w}},
\,\frac{1}{1-(r^{-1}s)^{\frac{1}{2}}\frac{z_2}{w}},\,
\frac{1}{1-(r^{-1}s)^{\frac{1}{2}}\frac{w}{z_1}},\,
\frac{1}{1-(r^{-1}s)^{\frac{1}{2}}\frac{w}{z_2}}$, we get
\begin{eqnarray*}
&&\frac{1}{1-(r^{-1}s)^{\frac{1}{2}}\frac{z_1}{w}} \big(
(r^{-1}s)^{-\frac{1}{2}}(z_1-(r^{-1}s)z_2)
-(r^{-1}s)^{-\frac{1}{2}}(z_2-(r^{-1}s)z_1)+(r{+}s)\\
&&(rs)^{-\frac{1}{2}}(z_2-z_1)\big)
+\frac{1}{1-(r^{-1}s)^{\frac{1}{2}}\frac{z_2}{w}} \big(
-(r^{-1}s)^{-\frac{1}{2}}(z_1-(r^{-1}s)z_2)+(r^{-1}s)^{-\frac{1}{2}}\\
&&(z_2-(r^{-1}s)z_1)+(r{+}s)(rs)^{-\frac{1}{2}}(z_1-z_2)\big)
-\frac{1}{1-(r^{-1}s)^{\frac{1}{2}}\frac{w}{z_1}}\big(w(1-(r^{-1}s)\frac{z_2}{z_1})\\
&& +w(\frac{z_2}{z_1}-(r^{-1}s))
+(r{+}s)(rs)^{-\frac{1}{2}}(r^{-1}s)^{\frac{1}{2}}w(1-\frac{z_2}{z_1})\big)
+\frac{1}{1-(r^{-1}s)^{\frac{1}{2}}\frac{w}{z_2}} \\
&&\big( w(\frac{z_1}{z_2}-(r^{-1}s))-w(1-(r^{-1}s)\frac{z_1}{z_2})
+(r{+}s)(rs)^{-\frac{1}{2}}(r^{-1}s)^{\frac{1}{2}}w(1-\frac{z_1}{z_2})\big)\\
&&=0,
\end{eqnarray*}
where each term is zero.\hfill $\Box$

Consequently, we complete the proof of Theorem 5.2.


\section{Appendix: Proofs of Some Lemmas via Quantum Calculations}

Here, we would provide more details for some Lemmas' proofs, through which readers can see the quantum calculations of $(r,s)$-brackets how to work.

\medskip
{\noindent\bf Proof of Lemma \ref{d:18}.} It is easy to get that
\begin{equation*}
\begin{split}
x_{\beta_{i-1,i+1}}^-(1)
&=\big[\,x_{i+1}^-(0),\,\cdots,\,x_n^-(0),\,x_{n-2}^-(0),\,
\cdots,\, x_{i+1}^-(0),\,\\
&\quad\underbrace{[\,x_{i}^-(0),x_{i-1}^-(1)\,]_s}\,
\big]_{(s,\,\cdots,\,s,\,r^{-1},\,\cdots,\,r^{-1})}\qquad\qquad\quad {\textrm{(by (\ref{c:1}))}}\\
&=-(rs)^{\frac{1}{2}}\big[\,x_{i+1}^-(0),\,\cdots,\,x_n^-(0),\,x_{n-2}^-(0),\,
\cdots,\, x_{i+1}^-(0),\,\\
&\quad[\,x_{i-1}^-(0),x_{i}^-(1)\,]_{r^{-1}}\,
\big]_{(s,\,\cdots,\,s,\,r^{-1},\,\cdots,\,r^{-1})}.
\end{split}
\end{equation*}
Using the above result, one has
\begin{equation*}
\begin{split}
[\,x_{i-1}^-&(0),\,x_{\beta_{i-1,i+1}}^-(1)\,]_{s^{-1}}\\
&=-(rs)^{\frac{1}{2}}\Big[\,x_{i-1}^-(0),\,
\big[\,x_{i+1}^-(0),\,\cdots,\,x_n^-(0),\,x_{n-2}^-(0),\,
\cdots,\, x_{i+1}^-(0),\, \\
&\quad[\,x_{i-1}^-(0),\,x_{i}^-(1)\,]_{r^{-1}}\,\big]_{(s,\,\cdots,\,s,\,r^{-1},\,\cdots,\,r^{-1})}\,\Big]_{s^{-1}}
\qquad\qquad {\textrm{(by (\ref{b:1}) \& (D9$_1$))}}\\
&=-(rs)^{\frac{1}{2}}
\big[\,x_{i+1}^-(0),\,\cdots,\,x_n^-(0),\,x_{n-2}^-(0),\,
\cdots,\, x_{i+1}^-(0),\,\\
&\quad\underbrace{
[\,x_{i-1}^-(0),\,x_{i-1}^-(0),x_{i}^-(1)\,]_{(r^{-1},\,s^{-1})}}\,
\big]_{(s,\,\cdots,\,s,\,r^{-1},\,\cdots,\,r^{-1})}\qquad {\textrm{(=0 by (D9$_3$))}}\\
&=0.
\end{split}
\end{equation*}\hfill\qed
\medskip

{\noindent\bf Proof of Lemma \ref{d:19}.} Use induction on $i$.
For the case of $i=n{-}1$, by definition, it is easy to see that
\begin{eqnarray*}
[\,x_{n-1}^-(0),\,x_{\beta_{n-1,n}}^-(1)\,]_{(rs)^{-1}}
=[\,x_{n-1}^-(0),\, x_{n}^-(1)\,]_{(rs)^{-1}}=0.
\end{eqnarray*}

Suppose Lemma \ref{d:19} is true for the case of $i$, then for
the case of $i-1$, we note that
\begin{equation*}
\begin{split}
&x_{\beta_{i-1,i}}^-(1)\\
&=[\,x_{i}^-(0),\,\cdots,\, x_n^-(0),\,x_{n-2}^-(0),\,\cdots,\,
x_i^-(0),\,x_{i-1}^-(1)\,]_
{(s,\,\cdots,\,s,\,r^{-1},\,\cdots,\,r^{-1})}\, {\textrm{(by (\ref{c:1}))}}\\
&=r^{-1}s\,\Big[\,x_{i}^-(0),\,\cdots,\,
x_n^-(0),\,x_{n-2}^-(0),\,\cdots,\, [\,x_i^-(1),\,x_{i-1}^-(0)\,]_r\,\Big]_
{(s,\,\cdots,\,s,\,r^{-1},\,\cdots,\,r^{-1})}\\
&\hskip8.75cm\qquad{\textrm{(by (\ref{b:1}) \& $(\textrm{D9}_1)$)}}\\
&=\cdots \\
&=r^{-1}s\,\bigl[\,x_{i}^-(0),\,[\,x_{i+1}^-(0),\,\cdots,\,
x_n^-(0),\,x_{n-2}^-(0),\cdots,x_i^-(1)\,]_
{(s,\,\cdots,\,s,\,r^{-1},\,\cdots,\,r^{-1})},\\
&\hskip7.4cm
x_{i-1}^-(0)\,\bigr]_{(r,r^{-1})}\qquad {\textrm{(by definition)}}\\
&=r^{-1}s\,[\,x_{i}^-(0),\,x_{\beta_{i,i+1}}^-(1) ,\,
x_{i-1}^-(0)\,]_{(r,r^{-1})}\qquad\qquad\hskip3.15cm {\textrm{(by (\ref{b:1}))}}\\
&=r^{-1}s\,[\,\underbrace{
[\,x_{i}^-(0),\,x_{\beta_{i,i+1}}^-(1)\,]_{(rs)^{-1}}}
 ,\,x_{i-1}^-(0)\,]_{rs}\qquad\quad\ \, {\textrm{(=0 by inductive hypothesis)}}\\
&\quad +r^{-2}\,[\,x_{\beta_{i,i+1}}^-(1) ,\,
[\,x_i^-(0),\,x_{i-1}^-(0)\,]_s\,]_{r^2s}\\
&=r^{-2}\,[\,x_{\beta_{i,i+1}}^-(1) ,\,
[\,x_i^-(0),\,x_{i-1}^-(0)\,]_s\,]_{r^2s}.
\end{split}
\end{equation*}
Using the above identity, then we have
\begin{equation*}
\begin{split}
&[\,x_{i-1}^-(0),\,x_{\beta_{i-1,i}}^-(1)\,]_{r^{-2}}\\
&\quad=r^{-2}\Big[\,x_{i-1}^-(0),\,
\big[\,x_{\beta_{i,i+1}}^-(1),\,[\,x_i^-(0),\,x_{i-1}^-(0)\,]_s\,\big]_{r^2s}\,
\Big]_{r^{-2}}\qquad\qquad\quad {\textrm{(by (\ref{b:1}))}}\\
&\quad=r^{-2}\Big[\,[\,x_{i-1}^-(0),\,
x_{\beta_{i,i+1}}^-(1)\,]_{r^{-1}},\,[\,x_i^-(0),\,x_{i-1}^-(0)\,]_s\,\,
\Big]_{rs}\\
&\quad\quad +r^{-3}\Big[\, x_{\eta_{i,i+1}}^-(1),\,
\underbrace{[\,x_{i-1}^-(0),\,x_i^-(0),\,x_{i-1}^-(0)\,]_{(s,\,r^{-1})}}
\,\Big]_{(rs)^{2}}\qquad {\textrm{(=0 by $(\textrm{D9}_3)$)}}\\
&\quad=r^{-2}\Big[\,\underbrace{[\,x_{i-1}^-(0),\,
x_{\beta_{i,i+1}}^-(1)\,]_{r^{-1}}},\,[\,x_i^-(0),\,x_{i-1}^-(0)\,]_s\,\,
\Big]_{rs}.
\end{split}
\end{equation*}
At the same time, we can also get
\begin{equation*}
\begin{split}
[\,x_{i-1}^-&(0),\,
x_{\beta_{i,i+1}}^-(1)\,]_{r^{-1}}\\
&=[\,x_{i-1}^-(0),\,[\,x_{i+1}^-(0),\,\cdots,x_n^-(0),x_{n-2}^-(0),\cdots,
\\
&\hskip3.5cm x_{i+1}^-(0),x_i^-(1)\,]_
{(s,\,\cdots,\,s,\,r^{-1},\,\cdots,\,r^{-1})}\,]_{r^{-1}}\qquad {\textrm{(by $(\textrm{D9}_1)$)}}\\
&=[\,x_{i+1}^-(0),\,\cdots,x_n^-(0),x_{n-2}^-(0),\cdots,
x_{i+1}^-(0),\\
&\quad\underbrace{[\,x_{i-1}^-(0),\,x_i^-(1)\,]_{r^{-1}}}\,]_
{(s,\,\cdots,\,s,\,r^{-1},\,\cdots,\,r^{-1})}\,]_{r^{-1}}\qquad
{\textrm{(by (\ref{c:1}) and definition)}}\\
&=-(rs)^{-\frac{1}{2}}[\,x_{i+1}^-(0),\,\cdots,x_n^-(0),x_{n-2}^-(0),\cdots,
x_{i+1}^-(0),\\
&\quad[\,x_{i}^-(0),\,x_{i-1}^-(1)\,]_{s}\,]_
{(s,\,\cdots,\,s,\,
r^{-1},\,\cdots,\,r^{-1})}\,]_{r^{-1}}\qquad\hskip2.1cm {\textrm{(by definition)}}\\
&=-(rs)^{-\frac{1}{2}}x_{\beta_{i-1,i+1}}^-(1).
\end{split}
\end{equation*}
At last, we get immediately
\begin{equation*}
\begin{split}
[\,x_{i-1}^-&(0),\,x_{\beta_{i-1,i}}^-(1)\,]_{r^{-2}}\\
&=-r^{-\frac{5}{2}}s^{-\frac{1}{2}}[\,x_{\beta_{i-1,i+1}}^-(1),\,
[\,x_i^-(0),\,x_{i-1}^-(0)\,]_s\,]_{rs}\qquad\quad (\hbox{by (\ref{b:1})})\\
&=-r^{-\frac{5}{2}}s^{-\frac{1}{2}}[\,[\,x_{\beta_{i-1,i+1}}^-(1),\,
x_i^-(0)\,]_r,\,x_{i-1}^-(0)\,]_{s^{2}}\qquad\quad(\hbox{by definition})\\
&\quad -r^{-\frac{3}{2}}s^{-\frac{1}{2}}[\,x_i^-(0),\,\underbrace{[\,x_{\beta_{i-1,i+1}}^-(1),\,
x_{i-1}^-(0)\,]_s}\,]_{r^{-1}s}\qquad(\hbox{=0 by {Lemma \ref{d:18}}})\\
&=(rs)^{-1}[\,[\,x_{\beta_{i-1,i}}^-(1),\,
x_{i-1}^-(0)\,]_{s^{2}}.
\end{split}
\end{equation*}
Expanding the two sides of the above identity, one gets
$$(1+r^{-1}s)[\,x_{i-1}^-(0),\,x_{\beta_{i-1,i}}^-(1)\,]_{(rs)^{-1}}=0,$$
which implies that if $r\ne -s$, then
$[\,x_{i-1}^-(0),\,x_{\beta_{i-1,i}}^-(1)\,]_{(rs)^{-1}}=0$. Thus
we have checked Lemma \ref{d:19} for the case of $i-1$.
Consequently, Lemma \ref{d:19} has been proved by induction. \hfill\qed
\medskip

{\noindent \bf Proof of Lemma \ref{d:20}.}  Firstly, note that
\begin{equation*}
\begin{split}
&[\,x_2^-{(0)}, x_{\beta_{1,4}}^-(1)\,]\hskip5.17cm \hbox{(by definition)}\\
&\quad=[\,x_2^-{(0)},\,
[\,x_4^-(0),\cdots,x_n^-(0),\,x_{n-2}^-(0),\cdots,
x_4^-(0),\,x_{\alpha_{1,4}}^-(1)\,]_{(s,\cdots,s,\,r^{-1},\cdots,r^{-1})}\,]\\
&\hskip7cm \qquad \hbox{(by (\ref{b:1}) \& $(\textrm{D9}_1)$)}\\
&\quad=[\,x_4^-(0),\cdots,x_n^-(0),\,x_{n-2}^-(0),\cdots,
x_4^-(0),\underbrace{[\,x_2^-{(0)},x_{\alpha_{1,4}}^-(1)\,]}\,]_{(s,\cdots,s,\,r^{-1},\cdots,r^{-1})}.
\end{split}
\end{equation*}
 So it suffices to check the relation $[\,x_2^-{(0)},\,x_{\alpha_{1,4}}^-(1)\,]=0$.

 In fact, it is easy to see that
\begin{equation*}
\begin{split}
[\,x_2^-&{(0)},\,x_{\alpha_{1,4}}^-(1)\,]_{r^{-1}s}\hskip5.6cm \hbox{(by definition)}\\
&=[\,x_2^-{(0)},\,
[\,x_3^-(0),\,x_2^-(0),\,x_{1}^-(1)\,]_{(s,\,s)}\,]_{r^{-1}s}
\qquad\qquad\ \;\,\hbox{(by (\ref{b:1}))}\\
&=[\,x_2^-(0),\,[\,x_3^-(0),\,x_{2}^-(0)\,]_s,\,x_{1}^-(1)\,]\,]_{(s,\,r^{-1}s)}
\qquad\qquad\ \hbox{(by (\ref{b:1}))}\\
&\quad +s[\,x_2^-(0),\,x_{2}^-(0),\,\underbrace{[\,x_3^-(0),\,x_{1}^-(1)\,]}\,]_{(1,\,r^{-1}s)}
\qquad\quad\ \;\, \hbox{(=0 by $(\textrm{D9}_1)$)}\\
&=[\,\underbrace{[\,x_2^-(0),\,[\,x_3^-(0),\,x_{2}^-(0)\,]_s\,]_{r^{-1}}},\,x_{1}^-(1)\,]_{s^{2}}
\qquad\quad\quad\ \,\;\hbox{(=0 by $(\textrm{D9}_3)$)}\\
&\quad +r^{-1}[\,[\,x_3^-(0),\,x_{2}^-(0)\,]_s,\,[\,x_2^-(0),\,x_{1}^-(1)\,]_s\,]_{rs}
\qquad\quad\ \hbox{(by definition \& (\ref{b:2}))}\\
&=r^{-1}[\,x_3^-(0),\,\underbrace{[\,x_{2}^-(0),\,[\,x_2^-(0),\,x_{1}^-(1)\,]_s\,]_r}\,]_{s^2}
\qquad\qquad\; \hbox{(=0 by $(\textrm{D9}_2)$)}\\
&\quad +[\,[\,x_3^-(0),\,x_2^-(0),\,x_{1}^-(1)\,]_{(s,\,s)},\,x_{2}^-(0)\,]_{r^{-1}s}
\qquad\qquad \hbox{(by definition)}\\
&=[\,x_{\alpha_{1,4}}^-(1),\,x_{2}^-(0)\,]_{r^{-1}s}.
\end{split}
\end{equation*}
Then, we obtain
$(1+r^{-1}s)[\,x_2^-{(0)},\,x_{\alpha_{1,4}}^-(1)\,]=0$.
When $r\neq -s$, we arrive at our required conclusion
$[\,x_2^-{(0)},\,x_{\alpha_{1,4}}^-(1)\,]=0$.\hfill\qed
 \medskip

{\noindent \bf Proof of Lemma \ref{d:21}.} Repeatedly using
$(\ref{b:1})$, it is easy to get that
\begin{equation*}
\begin{split}
&[\,x_i^-{(0)},\, x_{\beta_{1,i}}^-(1)\,]_{s^{-1}}\hskip6.2cm \hbox{(by definition)}\\
&\quad=[\,x_i^-{(0)},\,
[\,x_i^-(0),\,x_{i+1}^-(0),\,x_{\beta_{1,{i+2}}}^-(1)\,]_{(r^{-1},\,r^{-1})}\,]_{s^{-1}}
\qquad\quad \hbox{(by (\ref{b:1}))}\\
&\quad=[\,x_i^-{(0)},\,
[\,x_i^-(0),\,x_{i+1}^-(0)\,]_{r^{-1}},\,x_{\beta_{1,{i+2}}}^-(1)\,]_{(r^{-1},\,s^{-1})}
\qquad\quad \hbox{(by (\ref{b:1}))}\\
&\quad\quad +r^{-1}[\,x_i^-{(0)},\,
x_{i+1}^-(0),\underbrace{[\,x_{i}^-(0),\,x_{\beta_{1,{i+2}}}^-(1)\,]}
\,]_{(1,\,s^{-1})}
\qquad\quad \hbox{(=0 by Lemma \ref{d:add})}\\
&\quad=[\,\underbrace{[\,x_i^-{(0)},\,
[\,x_i^-(0),\,x_{i+1}^-(0)\,]_{r^{-1}}\,]_{s^{-1}}}
,\,x_{\beta_{1,{i+2}}}^-(1)\,]_{r^{-1}}
\qquad\quad \hbox{(=0 by $(\textrm{D9}_2)$)}\\
&\quad\quad +s^{-1}[\,
[\,x_i^-(0),\,x_{i+1}^-(0)\,]_{r^{-1}},\underbrace{[\,x_i^-{(0)},x_{\beta_{1,{i+2}}}^-(1)\,]}
\,]_{r^{-1}s}
\qquad \hbox{(=0 by Lemma \ref{d:add})}\\
&\quad=0.
\end{split}
\end{equation*}\hfill\qed

 {\noindent \bf Proof of Lemma  \ref{d:22}.}   By direct calculation, one has
\begin{equation*}
\begin{split}
&[\,x_n^-{(0)},\,
x_n^-(0),\,x_{\alpha_{1,n}}^-(1)\,]_{(rs^2,\,r^2s)}
\hskip4cm\,\quad\ \hbox{(by definition)}\\
&\quad=[\,x_n^-{(0)},\,
x_n^-(0),\,x_{n-1}^-(0),\,x_{\alpha_{1,n-1}}^-(1)\,]_{(s,\,rs^2,\,r^2s)}
\qquad\qquad\ \,\ \hbox{(by (\ref{b:1}))}\\
&\quad=[\,x_n^-{(0)},\,
\underbrace{[\,x_n^-(0),\,x_{n-1}^-(0)\,]_{rs}},\,x_{\alpha_{1,n-1}}^-(1)\,]_{(s^2,\,r^2s)}
\qquad\qquad\ \hbox{(=0 by (D9$_1$))}\\
&\quad\quad +rs\,[\,x_n^-{(0)},\,x_{n-1}^-(0)
,\,[\,x_n^-(0),\,x_{\alpha_{1,n-1}}^-(1)\,]_s\,]_{(r^{-1},\,r^2s)}
\qquad\ \hbox{(by (\ref{b:1}))}\\
&\quad=rs\,[\,\underbrace{[\,x_n^-{(0)},\,x_{n-1}^-(0)\,]_{rs}}
 ,\,[\,x_n^-(0),\,x_{\alpha_{1,n-1}}^-(1)\,]_s\,]\qquad\qquad\quad \hbox{(=0 by (D9$_1$))}\\
&\quad\quad+(rs)^2[\,x_{n-1}^-(0)
,\,\underbrace{[\,x_n^-{(0)},\,x_n^-(0),\,x_{\alpha_{1,n-1}}^-(1)\,]_{(s,\,r)}}
\,]_{r^{-2}s^{-1}}
\qquad \hbox{(by (\ref{b:1}))}\\
 &\quad=(rs)^2\,[\,x_{n-1}^-(0) ,\,\underbrace{[\,x_n^-{(0)},\,x_n^-(0),\,x_{\alpha_{1,n-1}}^-(1)\,]_{(s,\,r)}}
\,]_{r^{-2}s^{-1}}.
\end{split}
\end{equation*}
 Therefore, it suffices to show the relation
 $[\,x_n^-{(0)},\,x_n^-(0),\,x_{\alpha_{1,n-1}}^-(1)\,]_{(s,\,r)}=0$.

 Indeed, it is easy to get the following conclusion
\begin{equation*}
\begin{split}
&[\,x_n^-{(0)},\,x_n^-(0),\,x_{\alpha_{1,n-1}}^-(1)\,]_{(s,\,r)}
 \hskip5cm \hbox{(by definition)}\\
 &\quad=[\,x_n^-{(0)},\,x_n^-(0),\,x_{n-2}^-(0),\,x_{\alpha_{1,n-2}}^-(1)\,]_{(s,\,s,\,r)}
 \qquad\qquad\qquad\quad\, \hbox{(by (\ref{b:1}))}\\
 &\quad=[\,x_n^-{(0)},\,[\,x_n^-(0),\,x_{n-2}^-(0)\,]_s,\,x_{\alpha_{1,n-2}}^-(1)\,]_{(s,\,r)}
 \qquad\qquad\qquad\ \;\hbox{(by (\ref{b:1}))}\\
 &\quad\quad +s\,[\,x_n^-{(0)},\,x_{n-2}^-(0),\,\underbrace{[\,x_{n}^-(0),\,x_{\alpha_{1,n-2}}^-(1)\,]}\,]_{(1,\,r)}
 \qquad\quad \hbox{(=0 by (\ref{b:1}) \& (D9$_1$))}\\
 &\quad=[\,\underbrace{[\,x_n^-{(0)},\,[\,x_n^-(0),\,x_{n-2}^-(0)\,]_s\,]_r}
 ,\,x_{\alpha_{1,n-2}}^-(1)\,]_{s}\qquad\qquad\qquad\ \; \hbox{(=0 by (D9$_2$))}\\
 &\quad\quad +r\,[\,[\,x_n^-(0),\,x_{n-2}^-(0)\,]_s,\,\underbrace{[\,x_n^-{(0)},\,x_{\alpha_{1,n-2}}^-(1)\,]}\,]_{r^{-1}s}
 \qquad \hbox{(=0 by (\ref{b:3}) \& (D9$_1$))}\\
&\quad=0.
\end{split}
\end{equation*}\hfill\qed

{\noindent \bf Proof of Lemma  \ref{d:23}.}   First, we shall
check that $$[\,x_1^-(1),\, x_{\beta_{1,3}}^-(1)\,]_{r^{-1}}=0.$$
In fact, we get directly
\begin{equation*}
\begin{split}
[\,x_1^-&{(1)},\,
x_{\beta_{1,3}}^-(1)\,]_{r^{-1}}
\qquad\;\ \hskip4.7cm\hbox{(by definition)}\\
&=[\,x_1^-{(1)},\,[\,
x_3^-(0),\,x_{\beta_{1,4}}^-(1)\,]_{r^{-1}}\,]_{r^{-1}}
\qquad\qquad\qquad\quad\;\ \hbox{(by (\ref{b:1}))}\\
&=[\,\underbrace{[\,x_1^-(1),\,x_3^-{(0)}\,]},\,x_{\beta_{1,4}}^-(1)\,]_{r^{-2}}
\qquad\qquad\qquad\qquad\quad \hbox{(=0 by (D9$_1$))}\\
&\quad +[\,x_3^-{(0)},\,\underbrace{[\,x_1^-(1),\,x_{\beta_{1,4}}^-(1)\,]_{r^{-1}}}\,]_{r^{-1}}
\qquad\qquad \hbox{(repeating by (\ref{b:1}) \& (D9$_1$))}\\
&=[\,x_3^-{(0)},\cdots,x_n^-(0),\,x_{n-2}^-(0),\cdots,\underbrace{[\,x_1^-(1),\,
x_3^-(0)\,]},\,\\
&\hskip1.2cm[\,x_2^-(0),\,x_1^-(1)\,]_s\,]_{(s,\cdots,s,\,r^{-1},\cdots,r^{-1})}\qquad\qquad\quad\, \hbox{(=0 by (D9$_1$))}\\
&\quad +[\,x_3^-{(0)},\cdots,x_n^-(0),\,x_{n-2}^-(0),\,\cdots,\,x_3^-(0),\,\\
&\quad\underbrace{[\,x_1^-(1),\,[\,x_2^-(0),\,x_1^-(1)\,]_s\,]_{r^{-1}}}
\,]_{(s,\cdots,s,\,r^{-1},\cdots,r^{-1})}\qquad\,\hbox{(=0 by (D9$_3$))}\\
&=0.
\end{split}
\end{equation*}
Using the above fact, we would like to show that
$$
[\,x_2^-(0),\,[\,x_1^-(1),\,
x_{\theta}^-(1)\,]_{r^{-2}}\,]_1=0.
$$

Indeed, notice that
\begin{equation*}
\begin{split}
[\,x_2^-&(0),\,[\,x_1^-(1),\,
x_{\theta}^-(1)\,]_{r^{-2}}\,]_{r^{-1}s}
 \qquad\qquad\qquad\qquad\qquad\quad\,\hbox{(by definition \& (\ref{b:1}))}\\
 &=[\,x_2^-{(0)},\,[\,x_1^-(1),\,x_2^-(0)\,]_{r^{-1}},\,x_{\beta_{1,3}}^-(1)\,]_{(r^{-2},\,r^{-1}s)}
 \qquad \hbox{(using (\ref{b:1}))}\\
 &\quad +[\,x_2^-{(0)},\,x_2^-(0),\,\underbrace{[\,x_1^-(1),\,x_{\beta_{1,3}}^-(1)\,]_{r^{-1}}}\,]_{(1,\,r^{-1}s)}
 \qquad \hbox{(=0 by the above fact)}\\
&=[\,\underbrace{[\,x_2^-{(0)},\,[\,x_1^-(1),\,x_2^-(0)\,]_{r^{-1}}\,]_{s}},\,x_{\beta_{1,3}}^-(1)\,]_{r^{-3}}
 \qquad\quad\;\, \hbox{(=0 by  (D9$_2$))}\\
 &\quad +s[\,[\,x_1^-{(1)},\,x_2^-(0)\,]_{r^{-1}},\,
 \underbrace{[\,x_2^-(0),\,x_{\beta_{1,3}}^-(1)\,]_{r^{-1}}}\,]_{r^{-2}s^{-1}}
 \quad \hbox{(by definition)}\\
&=s[\,x_1^-{(1)},\,\underbrace{[\,x_2^-(0),\,
 x_{\theta}^-(1)\,]_{s^{-1}}}\,]_{r^{-3}}
 \qquad\qquad\qquad\qquad\  \hbox{(=0 by Lemma 4.1)}\\
&\quad +[\,[\,x_1^-{(1)},\,
x_{\theta}^-(1)\,]_{r^{-2}},\,x_2^-(0)\,]_{r^{-1}s}\\
 &=[\,[\,x_1^-{(1)},\,
x_{\theta}^-(1)\,]_{r^{-2}},\,x_2^-(0)\,]_{r^{-1}s}.
\end{split}
\end{equation*}
The above result means that if $r\neq -s$, then
$$
[\,x_2^-(0),\,[\,x_1^-(1),\,
x_{\theta}^-(1)\,]_{r^{-2}}\,]_1=0,
$$
which will be used in the sequel.

Using the above result, we get easily
\begin{equation*}
\begin{split}
[\,[\,x_{2}^-&(0),\,x_1^-(1)\,]_s,\,x_{\theta}^-(1)\,]_{r^{-2}s^{-1}}
\hskip3cm\ \;\; \hbox{(using (\ref{b:1}))}\\
&=[\,x_{2}^-(0),\,[\,x_1^-(1),\,x_{\theta}^-(1)\,]_{r^{-2}}\,]_1
\qquad\qquad\quad\,\qquad \hbox{(=0 by the above fact)}\\
&\quad +r^{-2}[\,\underbrace{[\,x_{2}^-(0),\,x_{\theta}^-(1)\,]_{s^{-1}}},\,x_1^-(1)\,]_{r^2s}
\qquad\qquad \hbox{(=0 by Lemma 4.1)}\\
&=0.
\end{split}
\end{equation*}

Applying the above statement, we are ready to derive that
\begin{equation*}
\begin{split}
&[\,x_{\beta_{1,3}}^-(0),\,x_{\theta}^-(1)\,]_{s}
 \qquad \hskip5cm\hbox{(by definition \& (\ref{b:1}))}\\
 &\quad=[\,x_3^-{(0)},\,\underbrace{[\,x_{\beta_{1,4}}^-(1),\,x_{\theta}^-(1)\,]_{s}}\,]_s
 \qquad\qquad\ \hbox{(repeating by (\ref{b:1}) \& Lemma 4.1)}\\
 &\qquad +[\,\underbrace{[\,x_3^-{(0)},\,x_{\theta}^-(1)\,]_1},\,x_{\beta_{1,4}}^-(1)\,]_{s^{2}}
  \qquad\qquad\qquad\qquad\, \hbox{(=0 by Lemma 4.1)}\\
&\quad=\cdots \\
&\quad=[\,x_3^-{(0)},\cdots,x_{n}^-(0),\,x_{n-2}^-(0),\cdots, x_3^-(0),\\
&\quad\hskip0.2cm\underbrace{[\,[\,x_{2}^-(0),\,x_1^-(1)\,]_s,\,x_{\theta}^-(1)\,]_{r^{-2}s^{-1}}}
\,]_{(s,\cdots,s,r^{-1},\cdots,r^{-1})}\quad\, \hbox{(=0 by the above result)}\\
&\quad=0.
\end{split}
\end{equation*}

\noindent{\bf Proof of Lemma \ref{d:25}.} Firstly, we need to consider
\begin{equation*}
\begin{split}
&x_{13456}^-(1)\qquad\hskip8cm \hbox{(by definition)}\\
&\quad=[\,x_6^-(0),\,x_5^-(0),\,x_4^-(0)
x_3^-(0),\,x_1^-(1)\,]_{(s,\,s,\,s,\,s)} \qquad\qquad\qquad\quad\ \;\, \hbox{(by (\ref{b:1}))}\\
&\quad=[\,x_6^-(0),\,[\,x_5^-(0),\,x_4^-(0)\,]_s
,\,[\,x_3^-(0),\,x_1^-(1)\,]_s\,]_{(s,\,s)}\\
&\quad\quad +s [\,x_6^-(0),\,x_4^-(0)
,\,\underbrace{[\,x_5^-(0),\,[\,x_3^-(0),\,x_1^-(1)\,]_s\,]_{1}}\,]_{(1,\,s)}
\quad \hbox{(=0 by (\ref{b:1}) \& $(\textrm{D9}_1)$)}\\
&\quad=[\,x_6^-(0),\,[\,x_5^-(0),\,x_4^-(0)\,]_s
,\,[\,x_3^-(0),\,x_1^-(1)\,]_s\,]_{(s,\,s)}.
\end{split}
\end{equation*}
Using the above result, now we turn to check
\begin{equation*}
\begin{split}
&[\,x_4^-(0),\, x_{13456}^-(1)\,]_{r^{-1}s} \hskip7.3cm \hbox{(by (\ref{b:1}))}\\
&\ =[\,x_6^-(0),\,\underbrace{[\,x_4^-(0),\,x_5^-(0),\,x_4^-(0)\,]_{(s,\,r^{-1})}},\,
[\,x_3^-(0),\,x_1^-(1)\,]_s\,]_{(s^2,\,s)}\quad \hbox{(=0 by $(\textrm{D9}_3)$)}\\
&\quad + r^{-1} [\,x_6^-(0),\,[\,x_5^-(0),\,x_4^-(0)\,]_{s},\,
[\,x_4^-(0),\,x_3^-(0),\,x_1^-(1)\,]_{(s,\,s)}\,]_{(rs,\,s)}\quad\, \;\hbox{(by (\ref{b:2}))}\\
&\ =r^{-1}[\,x_6^-(0),\,x_5^-(0),\,\underbrace{[\,x_4^-(0),\,
x_4^-(0),\,x_3^-(0),\,x_1^-(1)\,]_{(s,\,s,\,r)}}\,]_{(s^2,\,s)}\quad \hbox{(=0 by $(\textrm{D9}_2)$)}\\
&\quad +[\,x_6^-(0),\,[\,x_5^-(0),\,x_4^-(0),\,x_3^-(0),\,x_1^-(1)\,]_{(s,\,s,\,s,)},\,
x_4^-(0),\,]_{(r^{-1}s,\,s)}\qquad \hbox{(by (\ref{b:1}))}\\
&\ =[\,[\,x_6^-(0),\,x_5^-(0),\,x_4^-(0),\,x_3^-(0),\,x_1^-(1)\,]_{(s,\,s,\,s,\,s,)},\,
x_4^-(0),\,]_{r^{-1}s}\\
&\quad +s[\,[\,x_5^-(0),\,x_4^-(0),\,x_3^-(0),\,x_1^-(1)\,]_{(s,\,s,\,s,)},\,
\underbrace{[\,x_6^-(0),\,x_4^-(0)\,]_{1}}\,]_{r^{-1}}\quad \hbox{(=0 by $(\textrm{D9}_1)$)}\\
&\ =-r^{-1}s[\,x_4^-(0),\,x_{13456}^-(1)\,]_{rs^{-1}},
\end{split}
\end{equation*}
this implies
$$
(1+r^{-1}s)\,[\,x_4^-(0),\,
x_{13456}^-(1)\,]_{1}=0.
$$

So, when $r\ne -s$, we have
$$
[\,x_4^-(0),\,
x_{13456}^-(1)\,]_{1}=0.
$$

To get our required conclusion, we also need to deal with
\begin{equation*}
\begin{split}
[\,x_2^-&(0),\, x_{134562}^-(1)\,]_{r^{-1}}\qquad\qquad\quad\ \hbox{(by definition)}\\
&=[\,x_6^-(0),\,x_5^-(0),\,[\,x_2^-(0),\,x_2^-(0),\,x_4^-(0),\,x_3^-(0),\,
x_1^-(1)\,]_{(s,\,s,\,r^{-1},\,s^{-1},)}\,]_{(s,\,s)}\\
&\hskip4cm\qquad \hbox{(by (\ref{b:1}))}\\
&=[\,x_6^-(0),\,x_5^-(0),\,[\,x_2^-(0),\,[\,x_2^-(0),\,x_4^-(0)\,]_{r^{-1}},\,
[\,x_3^-(0),\,x_1^-(1)\,]_s\,]_{(s,\,s^{-1})}\,]_{(s,\,s)}\\
&\hskip4cm\qquad \hbox{(by (\ref{b:1}))}\\
&=[\,x_6^-(0),\,x_5^-(0),\,[\,\underbrace{[\,x_2^-(0),\,x_2^-(0),\,x_4^-(0)\,]_{(r^{-1},\,s^{-1})}}
,\,[\,x_3^-(0),\,x_1^-(1)\,]_{s}\,]_s\,]_{(s,\,s)}\\
&\hskip4cm\qquad \hbox{(=0 by $(\textrm{D9}_3)$)}\\
&\quad +[\,x_6^-(0),\,x_5^-(0),\,[\,[\,x_2^-(0),\,x_4^-(0)\,]_{r^{-1}}
,\,\underbrace{[\,x_2^-(0),\,x_3^-(0),\,x_1^-(1)\,]_{(s,\,1)}}\,]_{s^2}\,]_{(s,\,s)}\\
&\hskip6cm\qquad
\hbox{(=0 by (\ref{b:1}) \& $(\textrm{D9}_1)$)}\\
&=0.
\end{split}
\end{equation*}
Combining the definition of quantum root vector with the above relation, one
has
\begin{equation*}
\begin{split}
&x_{134562}^-(1)\qquad\qquad\hskip4.5cm \hbox{(by definition)}\\
&\qquad=[\,x_4^-(0),\,x_2^-(0),\,x_{13456}^-(1)\,]_{(r^{-1},\,s)} \qquad\qquad \hbox{(by (\ref{b:1}))}\\
&\qquad=[\,[\,x_4^-(0),\,x_2^-(0)\,]_s,\,x_{13456}^-(1)\,]_{r^{-1}}\\
&\qquad\quad +[\,x_2^-(0),\,\underbrace{[\,x_4^-(0),\,x_{13456}^-(1)\,]_1}\,]_{(rs)^{-1}}
\qquad\, \hbox{(=0 by the above result)}\\
&\qquad=[\,[\,x_4^-(0),\,x_2^-(0)\,]_s,\,x_{13456}^-(1)\,]_{r^{-1}}
\end{split}
\end{equation*}
Therefore, by definition, (\ref{b:1}) and Serre relations, we arrive at
\begin{equation*}
\begin{split}
[\,x_2^-&(0),\, x_{134562435}^-(1)\,]_{(rs)^{-1}}\qquad\qquad\qquad\qquad \hbox{(by definition)}\\
&=[\,x_5^-(0),\,x_3^-(0),\,
\underbrace{[\,x_2^-(0),\,x_{1345624}^-(1)\,]_{(rs)^{-1}}}\,]_{(r^{-1},\,s)}.
\end{split}
\end{equation*}

So, to show Lemma \ref{d:25}, it suffices to check that
$[\,x_2^-(0),\,x_{1345624}^-(1)\,]_{(rs)^{-1}}=0$.
Actually,
\begin{equation*}
\begin{split}
[\,x_2^-&(0),\, x_{1345624}^-(1)\,]_{r^{-2}}\qquad\qquad\hskip5cm \hbox{(by (\ref{b:1}))}\\
&=[\,\underbrace{[\,x_2^-(0),\,x_4^-(0),\,x_2^-(0)\,]_{(s,\,r^{-1})}},\,x_{13456}^-(1)\,]_{r^{-2}}
\qquad\qquad\qquad \hbox{(=0 by $(\textrm{D9}_2)$)}\\
&\quad +r^{-1}[\,[\,x_4^-(0),\,x_2^-(0)\,]_s,\,[\,x_2^-(0),\,x_{13456}^-(1)\,]_{r^{-1}}\,]_{1}
\qquad\quad\quad\ \; \hbox{(by (\ref{b:2}))}\\
&=r^{-1}[\,x_4^-(0),\,\underbrace{[\,x_2^-(0),\,x_2^-(0),\,x_{13456}^-(1)\,]_{(r^{-1},\,s^{-1})}}\,]
_{s}\quad \hbox{(=0 by the above result)}\\
&\quad+(rs)^{-1}[\,[\,x_4^-(0),\,x_2^-(0),\,x_{13456}^-(1)\,]
_{(r^{-1},\,s)},\,x_2^-(0)\,]
_{s^2}\qquad\quad \hbox{(by definition)}\\
&=-r^{-1}s\,[\,x_2^-(0),\, x_{1345624}^-(1)\,]_{s^{-2}}.
\end{split}
\end{equation*}
Expanding the two sides of the above relation, we get immediately
$$
(1+r^{-1}s)[\,x_2^-(0),\, x_{1345624}^-(1)\,]_{(rs)^{-1}}=0.
$$

Hence, we obtain our required conclusion. \hfill\qed
\medskip

\noindent{\bf Proof of Lemma \ref{d:26}.} To get the first
relation, we notice that
$$
x_{1345624}^-(1)=[\,[\,x_4^-(0),\,x_2^-(0)\,]_s,\,x_{13456}^-(1)\,]_{r^{-1}}.
$$
Then we get directly
\begin{equation*}
\begin{split}
&[\,x_4^-(0),\, x_{1345624}^-(1)\,]_{r}
\qquad\qquad\qquad\qquad\qquad\qquad\quad \hbox{(by (\ref{b:1}))}\\
&\ =[\,\underbrace{[\,x_4^-(0),\,x_4^-(0),\,x_2^-(0)\,]_{(s,\,r)}}
,\,x_{13456}^-(1)\,]_{r^{-1}}
\qquad\, \hbox{(=0 by $(\textrm{D9}_2)$)}\\
&\ \ +r[\,[\,x_4^-(0),\,x_2^-(0)\,]_{s},\,\underbrace{[\,x_4^-(0)
,\,x_{13456}^-(1)\,]_1}\,]_{r^{-2}}
\ \hbox{(=0 by the proof of Lemma \ref{d:25})}\\
&\ =0.
\end{split}
\end{equation*}

To check the second relation, we obtain easily
\begin{equation*}
\begin{split}
&[\,x_3^-(0),\, x_{1345624354}^-(1)\,]_{s^{-2}}\qquad\qquad\qquad\qquad\qquad\qquad\quad\ \, \hbox{(by definition)}\\
&\quad=[\,x_3^-(0),\,x_4^-(0),\,x_{134562435}^-(1)\,]_{(s,\,s^{-2})}
\qquad\qquad\qquad\quad\, \hbox{(by (\ref{b:1}))}\\
&\quad=[\,[\,x_3^-(0),\,x_4^-(0)\,]_{s^{-1}},\,x_{134562435}^-(1)\,]_{1}
\qquad\qquad\qquad\quad\, \hbox{(by definition)}\\
&\quad\ +s^{-1}[\,x_4^-(0),\,\underbrace{[\,x_3^-(0)
,\,x_{134562435}^-(1)\,]_{s^{-1}}}\,]_{s^{2}}
\qquad\ \hbox{(=0 by the below conclusion)}\\
&\quad=[\,[\,x_3^-(0),\,x_4^-(0)\,]_{s^{-1}},\,
[\,x_3^-(0),\,x_5^-(0),\,x_{1345624}^-(1)\,]_{(s,\,r^{-1})}\,]_{1}
\qquad\quad\ \,  \hbox{(by (\ref{b:1}))}\\
&\quad=[\,\underbrace{[\,[\,x_3^-(0),\,x_4^-(0)\,]_{s^{-1}},\,x_3^-(0)\,]_{r}},\,
[\,x_5^-(0),\,x_{1345624}^-(1)\,]_{s}\,]_{r^{-2}}\quad\quad \hbox{(=0 by $(\textrm{D9}_3)$)}\\
&\quad\ +r[\,x_3^-(0),\,[\,[\,x_3^-(0),\,x_4^-(0)\,]_{s^{-1}},\,
[\,x_5^-(0),\,x_{1345624}^-(1)\,]_{s}\,]_{r^{-1}}\,]_{r^{-2}}\quad\quad \hbox{(by (\ref{b:2}))}\\
&\quad=r[\,x_3^-(0),\,[\,x_3^-(0),\, \underbrace{[\,x_4^-(0),\,x_5^-(0)
,\,x_{1345624}^-(1)\,]_{(s,\,1)}}\,]_{(rs)^{-1}}\,]_{r^{-2}}\\
&\quad\hskip5.3cm\hbox{(=0 by the below conclusion)}\\
&\quad\ +r[\,x_3^-(0),\,[\,[\,x_3^-(0),\,x_5^-(0)
,\,x_{1345624}^-(1)\,]_{(s,\,,\,r^{-1})}
,\,x_4^-(0)\,]_{s^{-1}}\,]_{r^{-2}}\\
&\quad\hskip5.3cm \hbox{(by definition)}\\
&\quad=r[\,x_3^-(0),\,[\,x_{134562435}^-(1)
,\,x_4^-(0)\,]_{s^{-1}}\,]_{r^{-2}}\\
&\quad=-rs^{-1}[\,x_3^-(0),\, x_{1345624354}^-(1)\,]_{r^{-2}}.
\end{split}
\end{equation*}
This implies
$$
(1+rs^{-1})[\,x_3^-(0),\,
x_{1345624354}^-(1)\,]_{(rs)^{-1}}=0.
$$
So, when $r\ne -s$,
$$
[\,x_3^-(0),\, x_{1345624354}^-(1)\,]_{(rs)^{-1}}=0.
$$

Up to now, the proof of Lemma \ref{d:26} is to show that
$$
[\,x_3^-(0)
,\,x_{134562435}^-(1)\,]_{s^{-1}}=0,
$$
and
$$
[\,x_4^-(0),\,x_5^-(0),\,x_{1345624}^-(1)\,]_{(s,\,1)}=0.
$$

Before giving the proof of the above two conclusions, we also need
the following claims, whose proofs are easy and left to the reader.
\begin{gather*}
[\,x_3^-(0),\,x_{134562}^-(1)\,]_1=0,\tag{1}\\
[\,x_5^-(0) ,\,x_{134562}^-(1)\,]_1=0,\tag{2}\\
[\,x_4^-(0),\,x_{1345624}^-(1)\,]_r=0.\tag{3}
\end{gather*}

Therefore, using the above claims, one has
\begin{equation*}
\begin{split}
[\,x_3^-&(0),\, x_{13456243}^-(1)\,]_{s^{-1}}
\qquad\qquad\qquad\qquad\qquad\quad \hbox{(by definition \& (\ref{b:1}))}\\
&=[\,x_3^-(0),\,[\,x_3^-(0),\,x_4^-(0)\,]_{r^{-1}},\,x_{134562}^-(1)\,]_{(s,\,s^{-1})}
\qquad\qquad \hbox{(by (\ref{b:1}))}\\
&\quad +r^{-1}[\,x_3^-(0),\,x_4^-(0),\,\underbrace{[\,x_3^-(0)
,\,x_{134562}^-(1)\,]_1}\,]_{rs}
\qquad\qquad\; \hbox{(=0 by (1))}\\
&=[\,\underbrace{[\,x_3^-(0),\,x_3^-(0),\,x_4^-(0)\,]_{(r^{-1},\,s^{-1})}}
,\,x_{134562}^-(1)\,]_{s}
\qquad\qquad \hbox{(=0 by $(\textrm{D9}_3)$)}\\
&\quad +s^{-1}[\,[\,x_3^-(0),\,x_4^-(0)\,]_{r^{-1}},\,
\underbrace{[\,x_3^-(0),\,x_{134562}^-(1)\,]_1}\,]_{s^{2}}
\qquad \,\hbox{(=0 by (1))}\\
&=0.
\end{split}
\end{equation*}

For the last conclusion, we also get obviously
\begin{equation*}
\begin{split}
[\,x_4^-&(0),\,x_5^-(0),\, x_{1345624}^-(1)\,]_{(s,\,r^{-1}s)}
\qquad\qquad\quad\ \, \hbox{(by definition \& (\ref{b:1}))}\\
&=[\,x_4^-(0),\,[\,x_5^-(0),\,x_4^-(0)\,]_s,\,x_{134562}^-(1)\,]_{(s,\,r^{-1}s)}
\qquad\quad\ \,\hbox{(by (\ref{b:1}))}\\
&\quad +r[\,x_4^-(0),\,x_4^-(0),\,\underbrace{[\,x_5^-(0)
,\,x_{134562}^-(1)\,]_1}\,]_{(1,\,r^{-1}s)}
\qquad \hbox{(=0 by (2))}\\
&=[\,\underbrace{[\,x_4^-(0),\,x_5^-(0),\,x_4^-(0)\,]_{(s,\,r^{-1})}}
,\,x_{134562}^-(1)\,]_{s^{2}}
\qquad\quad\  \hbox{(=0 by $(\textrm{D9}_2)$)}\\
&\quad +r^{-1}
[\,[\,x_5^-(0),\,x_4^-(0)\,]_{s},\,[\,x_4^-(0),\,x_{134562}^-(1)\,]_s\,]_{rs}
\qquad \hbox{(by (\ref{b:1}))}\\
&=r^{-1}
[\,x_5^-(0),\,\underbrace{[\,x_4^-(0),\,x_{1345624}^-(1)\,]_r}\,]_{s^2}
\qquad\qquad\qquad\quad \hbox{(=0 by (3))}\\
&\quad +[\,[\,x_5^-(0),\,x_{1345624}^-(1)\,]_s,\,x_4^-(0)\,]_{r^{-1}s}\\
&=-r^{-1}s\,[\,x_4^-(0),\,x_5^-(0),\,x_{1345624}^-(1)\,]_{(s,\,rs^{-1})}.
\end{split}
\end{equation*}
Expanding the two sides of the above relation,  we obtain
$$
(1+r^{-1}s)[\,x_4^-(0),\,x_5^-(0),\,
x_{1345624}^-(1)\,]_{(s,\,1)}=0.
$$
So,
$[\,x_4^-(0),\,x_5^-(0),\, x_{1345624}^-(1)\,]_{(s,\,1)}=0$ under the condition $r\ne -s$.
\hfill\qed
\medskip

\noindent {\bf Proof of Lemma \ref{d:27}.} For simplicity, we introduce some notations
$$
A\doteqdot[\,x_4^-(0),\,x_2^-(0),\,x_6^-(0),\,x_5^-(0),\,x_4^-(0),\,x_3^-(1)\,]
_{(s,\,s,\,s,\,r^{-1},\,s)},
$$
and
$$
B\doteqdot[\,x_2^-(0),\,x_6^-(0),\,x_5^-(0),\,x_4^-(1)\,]
_{(s,\,s,\,r^{-1})}.
$$

Note that
\begin{equation*}
\begin{split}
[\,x_1^-&(0),\, x_{1345624354}^-(1)\,]_{(rs)^{-1}}\,\qquad \qquad \hbox{(by (\ref{b:1}) \& $(\textrm{D9}_1)$)}\\
&=[\,x_4^-(0),\,x_5^-(0),\,\underbrace{[\,x_1^-(0),\,
x_{13456243}^-(1)\,]_{(rs)^{-1}}}\,]_{(s,\,s)}.
\end{split}
\end{equation*}
As a consequence, it suffices to verify that $[\,x_1^-(0),\,
x_{13456243}^-(1)\,]_{(rs)^{-1}}=0$.

In fact, we observe that
\begin{equation*}
\begin{split}
&x_{13456243}^-(1)\\
&\ =[\,x_3^-(0),\,x_4^-(0),\,x_2^-(0)
,\,x_6^-(0),\,x_5^-(0),\,x_4^-(0),\, x_3^-(0),\,x_1^-(1)\,]
_{(s,\,s,\,s,\,s,\,r^{-1},\,s,\,r^{-1})}\\
&\ \hskip6.5cm\quad \hbox{(by (D7) \& $(\textrm{D9}_1)$)}\\
&\ =[\,x_3^-(0),\,x_1^-(0),\,x_4^-(0),\,x_2^-(0),\,x_6^-(0),\,x_5^-(0),\,
x_4^-(0),\,x_3^-(1)\,] _{(s,\,s,\,s,\,r^{-1},\,s,\,r^{-1},\,r^{-1})}\\
&\ \hskip6.5cm\quad \hbox{(by (\ref{b:1}))}\\
&\ =[\,[\,x_3^-(0),\,x_1^-(0)\,]_s,\,A\,] _{r^{-2}s^{-1}}+[\,x_1^-(0),\,\underbrace{[\,x_3^-(0),\,A\,]_{(rs)^{-1}}}\,]
_{rs^{-1}}\\
&\ \hskip6.5cm\quad \hbox{(=0 by the following conclusion)}\\
&\ =[\,[\,x_3^-(0),\,x_1^-(0)\,]_s,\,A\,] _{r^{-2}s^{-1}}.
\end{split}
\end{equation*}

Applying the result, we obtain immediately
\begin{equation*}
\begin{split}
[\,x_1^-&(0),\,x_{13456243}^-(1)\,]_{r^{-2}}\hskip4.9cm \hbox{(by definition)}\\
&=[\,x_1^-(0),\,[\,x_3^-(0),\,x_1^-(0)\,]_s,\,A\,]_{(r^{-2}s^{-1},\,r^{-2})}
\qquad\qquad \hbox{(by (\ref{b:1}))}\\
&=[\,\underbrace{[\,x_1^-(0),\,x_3^-(0),\,x_1^-(0)\,]_{(s,\,r^{-1})}},\,A\,]_{r^{-3}s^{-1}}
\qquad\qquad \hbox{(=0 by $(\textrm{D9}_3)$)}\\
&\quad +r^{-1}[\,[\,x_3^-(0),\,x_1^-(0)\,]_s,\,[\,x_1^-(0),\,A\,]_{r^{-1}}\,]_{(rs)^{-1}}
\qquad \hbox{(by (\ref{b:2}))}\\
&=r^{-1}[\,x_3^-(0),\,\underbrace{[\,x_1^-(0),\,
x_1^-(0),\,A\,]_{(r^{-1},\,s^{-1})}}\,]_{r^{-1}s}
\qquad\ \, \hbox{(=0 by $(\textrm{D9}_3)$)}\\
&\quad +(rs)^{-1}[\,[\,x_3^-(0),\,x_1^-(0),\,A\,]_{(r^{-1},\,r^{-1})},\,
x_1^-(0)\,]_{s^2}\\
&=-r^{-1}s[\,x_1^-(0),\, x_{13456243}^-(1)\,]_{s^{-2}}.
\end{split}
\end{equation*}
This leads to $(1+r^{-1}s)\,[\,x_1^-(0),\,
x_{13456243}^-(1)\,]_{(rs)^{-1}}=0$.

So, under the condition  $r\ne -s$, we get
$$
[\,x_1^-(0),\, x_{13456243}^-(1)\,]_{(rs)^{-1}}=0.
$$

\medskip
Finally, we are left to verify
$$
[\,x_3^-(0),\,A\,]_{(rs)^{-1}}=0.
$$

Indeed, it is easy to see that
\begin{equation*}
\begin{split}
[\,x_3^-&(0),\, A\,]_{r^{-2}}\hskip6cm \hbox{(by definition)}\\
&=[\,x_3^-(0),\,[\,x_4^-(0),\,x_3^-(0)\,]_s,\,B\,]_{(r^{-1},\,r^{-2})}
\qquad\qquad\, \hbox{(by (\ref{b:1}))}\\
&=[\,\underbrace{[\,x_3^-(0),\,x_4^-(0),\,x_3^-(0)\,]_{(s,\,r^{-1})}},\,B\,]_{r^{-2}}
\qquad\qquad\, \hbox{(=0 by $(\textrm{D9}_3)$)}\\
&\quad +r^{-1}[\,[\,x_4^-(0),\,x_3^-(0)\,]_s,\,[\,x_3^-(0),\,B\,]_{r^{-1}}\,]_{1}
\qquad\ \;\, \hbox{(by (\ref{b:2}))}\\
&=r^{-1}[\,x_4^-(0),\,\underbrace{[\,x_3^-(0),\,
x_3^-(0),\,B\,]_{(r^{-1},\,s^{-1})}}\,]_{s^{2}}
\qquad\; \hbox{(=0 by $(\textrm{D9}_2)$ \& $(\textrm{D9}_3)$)}\\
&\quad +(rs)^{-1}[\,[\,x_4^-(0),\,x_3^-(0),\,B\,]_{(r^{-1},\,s)},\,
x_3^-(0)\,]_{s^2}\\
&=-r^{-1}s[\,x_3^-(0),\, A\,]_{s^{-2}}.
\end{split}
\end{equation*}
This means $(1+r^{-1}s)\,[\,x_3^-(0),\, A\,]_{(rs)^{-1}}=0$.

So, under the condition $r\ne -s$, we get $[\,x_3^-(0),\, A\,]_{(rs)^{-1}}=0$.

\medskip
Therefore, this completes the proof.  \hfill\qed

\newpage
\noindent
{\it Added in proof.} Part of the work started initially 10 years ago when the first author 
visited l'DMA, l'Ecole Normale Sup\'eieure de Paris
from October to November, 2004, the Fachbereich Mathematik der Universit\"at Hamburg from November
2004 to February 2005. It was not until his visit to ICTP (Trieste, Italy) from March to August, 2006
that Hu found out the explicit formula of the generating function $g_{ij}(z)$ with $\tau$-invariance which leads to 
the inherent definition for the Drinfeld realization in two-parameter setting. A reason for preventing 
the submission of this work for 10 years is the newly found and amending constraint: $\ga\ga'=(rs)^c$.
The original constraint $\ga\ga'=rs$ is apparently dissatisfied since product of two group-likes should be 
still group-like. The current change was realized and made by the first author when the second author tried to construct
level-two vertex operator representation and found that the constraint should be $\ga\ga'=(rs)^2$ in that module.
Finally, the authors would like to express their thanks to Dr. Yunnan Li, who pointed out a necessary revision for the definition
of $\ga$, $\ga'$ in Definition 2.2, which is crucial for our main Theorem 3.9.

\vskip30pt \centerline{\bf ACKNOWLEDGMENT}

\bigskip

Hu is supported in part by the NNSF of China (No. 11271131),  the RSFDP
from the MOE of China.  Zhang would
like to thank the support of NSFC grant (No. 11101258) and Shanghai Leading
Academic Discipline Project (J50101). Both authors are indebted to Marc Rosso for
his recommending Gross\'e's arXiv-preprint earlier in 2004, from where the initial 
motivation stemmed. 

\bigskip
\bibliographystyle{amsalpha}

\end{document}